\documentclass[english]{article}
\usepackage[T1]{fontenc}
\usepackage[latin9]{inputenc}
\usepackage{geometry}
\usepackage{xcolor}
\usepackage{mathtools}
\usepackage{amsmath}
\usepackage{amsthm}
\usepackage{amssymb}
\usepackage{setspace}
\makeatletter
\theoremstyle{plain}
\newtheorem{assumption}{\protect\assumptionname}
\theoremstyle{definition}
\newtheorem{defn}{\protect\definitionname}
\theoremstyle{definition}
\newtheorem{problem}{\protect\problemname}
\theoremstyle{remark}
\newtheorem{rem}{\protect\remarkname}
\theoremstyle{plain}
\newtheorem{prop}{\protect\propositionname}
\theoremstyle{plain}
\newtheorem{thm}{\protect\theoremname}
\theoremstyle{plain}
\newtheorem{lem}{\protect\lemmaname}
\theoremstyle{plain}
\newtheorem{cor}{\protect\corollaryname}
\theoremstyle{remark}
\newtheorem{notation}{\protect\notationname}

\makeatother

\usepackage{babel}
\providecommand{\assumptionname}{Assumption}
\providecommand{\corollaryname}{Corollary}
\providecommand{\definitionname}{Definition}
\providecommand{\lemmaname}{Lemma}
\providecommand{\notationname}{Notation}
\providecommand{\problemname}{Problem}
\providecommand{\propositionname}{Proposition}
\providecommand{\remarkname}{Remark}
\providecommand{\theoremname}{Theorem}

\begin{document}
\title{Structure of Motion under Constraints and non-Holonomic Path-Following
on $\mathbb{R}^{3}$}
\author{Bohuan Lin\thanks{School of Mathematics and Physics, Xi'an Jiaotong-Liverpool University.}, 
Weijia Yao\thanks{School of Artificial Intelligence and Robotics, Hunan University.  Corresponding author: Weijia Yao (wjyao@hnu.edu.cn)}}
\maketitle
\begin{abstract}
In this paper we study a path-following problem on $\mathbb{R}^{3}$
with a non-holonomic constraint. The geometric structure associated
to the velocity constraint is explored, and general principles for
constructing guiding vector fields are obtained, fulfilling the path-following
requirements on a neighborhood of the desired path while allowing
the design of vector fields to be conducted in global coordinates. 
\end{abstract}

\section{\label{sec:PathFollo-R^3}Introduction: Constrained Path-following
in $\mathbb{R}^{3}$}

\subsection{Background and Basic Settings}

Systems with non-holonomic constraints are widely seen in science
and engineering practice. Typical non-holonomic systems include vehicles
and spherical robots \cite{Manzoor2021ControllabilitySphereBot} with
the no-slipping condition between wheels and the ground, and micro-swimmers
in fluids with (high) viscosity \cite{Shapere1989GeometricSwimLowReynolds,Avron2008GeometricSwimming,Kadam2016GeometricControllabilityPurcellSwimmer}.
Due to the importance of non-holonomic systems to the real world,
mathematical theories have been developed for describing and characterizing
motions in these systems. The velocity constraints of a mechanical
system can be considered as a geometric structure on the phase space,
since they usually come with a set of differential $1$-forms, $\beta_{1},...,\beta_{k}$,
on the phase space. More specifically, these constraints impose restriction
on the velocities allowed for the system by requiring every admissible
velocity $v$ to satisfy the equations $\beta_{j}(v)=0$ for all $j=1,...,k$.
This gives rise to a distribution $\mathcal{D}=\underset{j}{\bigcap}\ker\beta_{j}$,
that is, a sub-bundle of the tangent bundle of the phase space. That
the constraints are non-holonomic simply means the distribution $\mathcal{D}$
to be non-integrable, and the degree of non-integrability is then
directly related to the controllability/attainability of the system.
Proper modeling of some specific systems may even make the constraints
into a gauge-invariant connection on certain principal-bundle structure
of the state-control/position-shape space, and then the degree of
controllability/non-integrability can be learned from the curvature
of the connection \cite{Shapere1989GeometricSwimLowReynolds}. 

Due to its geometric nature, motion planning of constrained systems
has received attention for the potential application to this field
of the math from Geometric Control and Sub-Riemannian geometry \cite{Jean2014SubRiemannianControl,Bloch2015nonholonomicMechanics,mashtakov2020extremal}.
Research in this field aims at finding proper paths towards targets
to fulfill certain requirements. For non-holonomic systems, it is
a fundamental problem to approach the target in the most efficient
way, which is about looking for geodesics of the sub-Riemannian structure. 

We believe that path-following control of such systems should
also be able to benefit from the geometry of the constraints, and
the purpose of this work is to explore and reveal this possibility.
Path following is one kind of motion control closely related to motion
planning, but different. Here, a desired path is usually prescribed,
and the focus is on establishing control schemes to approach and follow
this path. An active research topic in recent years is vector-field-based 
methods for motion control, especially for path following and motion
planning \cite{goncalves2010vector,liang2015vector,kapitanyuk2017guiding,kapitanyuk2017guiding4nonholonomic,frew2017tracking,rezende2018robust,wilhelm2019vector,Yao2023GuidingFieldThesis,gushkov2023vector,he2024simultaneous,qiao2024motion,chen2024novel,he2025novel,qiao2025curvature,zhou2025inverse,lu2025versatile}.
Navigating robots through a specially designed vector field is reported
to be accurate and efficient \cite{Sujit2014SurveyUAVpathFollow}.
General principles and methods for constructing guiding vector fields
for path-following in free space (without constraints) have been studied
and established extensively and systematically \cite{goncalves2010vector,kapitanyuk2017guiding,yao2018robotic,yao2022topological}.
In the meantime, for systems with velocity constraints, especially
non-holonomic vehicles, vector-field methods for path-following control
and motion planning have also been substantially developed \cite{liang2015vector,kapitanyuk2017guiding4nonholonomic,rezende2018robust,he2024simultaneous,qiao2024motion,he2025novel,qiao2025curvature}.
While most of the current work \cite{kapitanyuk2017guiding4nonholonomic,qiao2024motion,qiao2025curvature,he2024simultaneous,he2025novel}
on vector-field methods for non-holonomic systems has targeted at
specific models, in this paper we investigate the general principles
for the path-following of non-holonomic systems from a geometric point
of view. 

Due to our limitation, we will only focus on the simplest case, where
the phase space is $\mathbb{R}^{3}$ and the velocity constraint is
Pfaffian, given by a single differential $1$-form $\beta$ with the
equation $\beta(v)=0$ for admissible velocities $v\in\mathrm{T}\mathbb{R}^{3}$. So the kernel of $\beta$, denoted by $\ker\beta$, is the space
of admissible velocities, and it is a smooth subbundle of the tangent
bundle $\mathrm{T}\mathbb{R}^{3}$, or say, a distribution on $\mathbb{R}^{3}$.
The desired path $\mathcal{P}$ is assumed to be a loop in $\mathbb{R}^{3}$,
i.e., $\mathcal{P}\cong S^{1}$. In particular, we are interested
in the scenario where $\mathcal{P}$ is transverse to $\ker\beta$,
which can be considered as the worst case for the task and will be further
justified. To be precise, our investigation on the construction of guiding vector fields for $\mathcal{P}$ will be conducted with the following assumptions:
\begin{assumption}
\label{assu:BasicAssumption}$ $

1. at each $p\in\mathcal{P}$, $\ker\beta\big|_{p}$ is $2$ dimensional;

2. $\mathcal{P}$ is transverse to $\ker\beta$;

3. $\mathcal{P}$ is contained in a single orbit of $\ker\beta$.
\end{assumption}
These three assumptions in Assumption \ref{assu:BasicAssumption} together
imply the orbit containing $\mathcal{P}$ to be $3$ dimensional (larger than the dimension of $\ker\beta$),
and hence the constraint is non-holonomic (at least in a neighborhood
of $\mathcal{P}$). Here, by saying an orbit of $\ker\beta$ we mean an orbit of the set of
all vector fields subject to $\ker\beta$. To be more specific, denote
by $\mathfrak{K}$ the set of all (complete) vector fields on $\mathbb{R}^{3}$
subject to the constraint $\ker\beta$, and then an orbit of $\ker\beta$
starting at some point $p\in\mathbb{R}^{3}$ is the set $\mathcal{O}_{p}$
defined as below:
\begin{defn}
\label{def:Orbit}\cite{Sachkov2022GeometricControl} $q\in\mathcal{O}_{p}$
if and only if there exists a finite set of vector fields $\{X_{1},...,X_{k}\}\subset\mathfrak{K}$
together with time moments $t_{1},...,t_{k}\in\mathbb{R}$, such that,
\[
q=\varphi_{X_{k}}^{t_{k}}\circ...\circ\varphi_{X_{1}}^{t_{1}}(p).
\]
Here, $\varphi_{X_{i}}$ is the flow of $X_{i}$ for each $j=1,...,k$.
\end{defn}
We shall note that the last requirement in Assumption \ref{assu:BasicAssumption} is actually a necessary condition for the path-following task. This is because every motion subject to the constraint remains in a single orbit, and therefore, if $\mathcal{P}$ were not in a single orbit, the motion would have to deviate from $\mathcal{P}$ (from now and then) instead of getting closer and closer to it.

We shall also justify the assumption of transversality. The transversality
of $\mathcal{P}$ to $\ker\beta$ actually makes it impossible to
move on $\mathcal{P}$, and hence $\mathcal{P}$ is not an admissible
path under the constraint. One may argue that this undermines the
phrase ``desired path'' and then the ``path following'' . To justify
such a setting, we note that, in practice, instead of being part of
the design, $\mathcal{P}$ can be an actual curve that exists beyond
the design of the control system. For example, $\mathcal{P}$ can
be the trace left by some organism, or, a ring of particles with chemical
compounds of interests. Therefore, if the task is to trace such a
curve, and, if due to certain limitation the velocity constraint is
unavoidable, then it is the situation with such transversality in
which the machine operates. For another perspective, we can also think
of $\mathcal{P}$ as an intermediate product in motion planning. While
the motion to be designed should produce orbits subject to the constraint,
we may first sketch a curve $\mathcal{P}$ that roughly indicates
how the robot is expected to travel through the space, and this $\mathcal{P}$
does not have to meet the constraint since it does not represent the
final orbit designed for the robot. After sketching $\mathcal{P}$,
the task of motion planning will then be completed by a path-following
scheme to approach and trace $\mathcal{P}$ under the constraint.

\subsection{General Structure of $\mathcal{X}$ and Main Problems}

It is of our concern to establish a systematic approach for constructing a
guiding vector field $\mathcal{X}$ tangent to $\ker\beta$ generating motions (at
least in a neighborhood of the path) that circulate along $\mathcal{P}$
and converge to the path $\mathcal{P}$, and for precision, we will focus the
discussion on a tubular neighborhood $\mathcal{U}$ of $\mathcal{P}$.
Note that $\mathcal{U}$ is diffeomorphic to $\mathbf{B}_{\delta}\times S^{1}$
with $\mathbf{B}_{\delta}=\{\mathbf{z}\in\mathbb{R}^{2}\big|\ |\mathbf{z}|\leq\delta\}$
and $\mathcal{P}\cong\{\mathbf{0}\}\times S^{1}$. Since $\ker\beta$
is transverse to $\mathcal{P}$, $\mathcal{P}$ is not an admissible path, and hence for the path-following
task we will focus on motions in the space
\begin{equation}
\mathcal{U}_{*}=\mathcal{U}\setminus\mathcal{P}\cong\mathbf{B}_{\delta}^{*}\times S^{1}\label{eq:=00005CU_*=00003DU=00005CP}
\end{equation}
with $\mathbf{B}_{\delta}^{*}=\mathbf{B}_{\delta}\setminus\{\mathbf{0}\}$.
Using variables $(x,y,e^{i\theta})$ for the points in $\mathbf{B}_{\delta}\times S^{1}$,
the problem is formulated as below:
\begin{problem}
\label{prob:=00005BGeneralMainProblem=00005D3D-nonHolonomic-PathFollowing}(How
to) Construct a vector field $\mathcal{X}$ (with flow $\varphi_{\mathcal{X}}$)
on $\mathcal{U}$, such that,

1) $\mathcal{X}$ is tangent to $\ker\beta$, i.e., $\mathcal{X}_{p}\in\ker\beta\big|_{p}$,
$\forall p\in\mathcal{U}$;

for any $p\in\mathcal{U}_{*}$, the trajectory $\eta_{t}=\varphi_{\mathcal{X}}^{t}(p)$
has the tendency of 

2) circulating along $\mathcal{P}$
\begin{equation}
\int_{\eta_{[0,T]}}d\theta:=\int_0^Td\theta(\dot{\eta_t})dt\longrightarrow\infty\;\text{\;as\; }T\rightarrow\infty,\label{eq:Circling}
\end{equation}

3) as well as converging towards $\mathcal{P}$
\begin{equation}
\mathrm{dist}(\eta_{t},\mathcal{P})\longrightarrow0\;\text{\;as\;}\;t\rightarrow\infty.\label{eq:Convergence}
\end{equation}
\end{problem}

To get an insight into Problem \ref{prob:=00005BGeneralMainProblem=00005D3D-nonHolonomic-PathFollowing},
we take a look at the structure of a vector field $\mathcal{X}$ subject
to the constraint $\beta(\mathcal{X})=0$. With a Riemannian metric
$\langle\ ,\ \rangle$ on $\mathbb{R}^{3}$, $\beta$ (nondegenerate)
is dual to a vector field $\mathcal{V}_{\beta}$ (nowhere vanishing)
via $\beta(\cdot)=\langle\mathcal{V}_{\beta},\cdot\rangle$ and thence
\[
\beta(\mathcal{X})=0\iff\langle\mathcal{V}_{\beta},\mathcal{X}\rangle=0.
\]
For the sake of practicality, throughout this paper $\langle\ ,\ \rangle$
is taken to be the standard Riemannian metric on $\mathbb{R}^{3}$
(i.e., the usual dot product ``$\cdot$''). Together with the usual
cross product ``$\times$'' of vectors on $\mathbb{R}^{3}$, we
have the following result for representing the vector field $\mathcal{X}$:
\begin{equation}
\langle\mathcal{V}_{\beta},\mathcal{X}\rangle=0\ \ \iff\ \ \mathcal{X}=\mathcal{V}_{\beta}\times\bar{\mathcal{X}}\text{ for some vector (field) }\bar{\mathcal{X}}.\label{eq:StructureOnR^3-=00005CV=00005Ctimes=00005CX}
\end{equation}
Relation (\ref{eq:StructureOnR^3-=00005CV=00005Ctimes=00005CX}) will
be derived and proved in later discussion, and for constructing $\mathcal{X}$
it is then to search for an appropriate $\bar{\mathcal{X}}$. 

Given a specific diffeomorphism between $\mathcal{U}$ and $\mathbf{B}_{\delta}\times S^{1}$,
we may define a function $\mathfrak{H}$ on $\mathcal{U}$ in terms
of the variables $(x,y,e^{i\theta})$ in $\mathbf{B}_{\delta}\times S^{1}$:
\begin{equation}
\mathfrak{H}(x,y,e^{i\theta})=x^{2}+y^{2},\label{eq:LyapnovFunction-H}
\end{equation}
and then 
\[
\mathcal{P}:\ \mathfrak{H}=0.
\]
The gradient $\nabla\mathfrak{H}$ given via duality
\[
d\mathfrak{H}(\cdot)=\langle\nabla\mathfrak{H},\cdot\rangle
\]
is then perpendicular to $\frac{\partial}{\partial\theta}$ since
$\frac{\partial}{\partial\theta}\mathfrak{H}=0$. Meanwhile, by Assumption
\ref{assu:BasicAssumption}, the vector field $\frac{\partial}{\partial\theta}$
is transverse to $\ker\beta$ on $\mathcal{P}$. Shrinking $\mathcal{U}$ and replacing $\beta$ with $-\beta$ if necessary, $\frac{\partial}{\partial\theta}$ is then transverse
to $\ker\beta$ on $\mathcal{U}$ with
\begin{equation}
\langle\mathcal{V}_{\beta},\frac{\partial}{\partial\theta}\rangle=\textcolor{violet}{\beta\big(\frac{\partial}{\partial\theta}\big)>0.\label{eq:=00005Cbeta(=00005Cpartial=00005Ctheta)=00003DNonZero}
}
\end{equation}
Now that $\nabla\mathfrak{H}\big|_{p}\neq\mathbf{0}$ for each $p\in\mathcal{U}_{*}$,
the vector fields $\mathcal{V}_{\beta}$, $\nabla\mathfrak{H}$ and
$\frac{\partial}{\partial\theta}\times\nabla\mathfrak{H}$ constitute
a frame of the tangent bundle $\mathrm{T}_{\mathcal{U}_{*}}\mathbb{R}^{3}$
over $\mathcal{U}_{*}$. As a result, we have a more specific representation
(compared to (\ref{eq:StructureOnR^3-=00005CV=00005Ctimes=00005CX}))
of $\mathcal{X}$ on $\mathcal{U}_{*}$:\textcolor{purple}{
\begin{equation}
\begin{aligned}\mathcal{X}= & \bar{a}\cdot\mathcal{V}_{\beta}\times\nabla\mathfrak{H}+\bar{b}\cdot\mathcal{V}_{\beta}\times\big(\frac{\partial}{\partial\theta}\times\nabla\mathfrak{H}\big)\end{aligned}
,\label{eq:=00005B=00005CX=00005DnonHolonomic-PathFollow=00005B3D=00005D}
\end{equation}
}in which the weight functions $\bar{a},\bar{b}:\mathbb{R}^{3}\rightarrow\mathbb{R}$
can be used to modify the vector field. 

Based on (\ref{eq:=00005B=00005CX=00005DnonHolonomic-PathFollow=00005B3D=00005D}),
the task raised in Problem \ref{prob:=00005BGeneralMainProblem=00005D3D-nonHolonomic-PathFollowing}
is reformulated as follows:
\begin{problem}
\label{prob:MainProblem-concrete=000026specified}{[}main-specific{]}
Suppose that $\beta\wedge d\beta$ is nondegenerate on $\mathcal{U}$.
Find suitable conditions for the weight functions $\bar{a},\bar{b}$
on $\mathcal{U}$ so that the vector field $\mathcal{X}$ given by
(\ref{eq:=00005B=00005CX=00005DnonHolonomic-PathFollow=00005B3D=00005D})
fulfills the requirements of circling (\ref{eq:Circling}) and convergence
(\ref{eq:Convergence}). 
\end{problem}
\begin{rem}
Note that both the functions $\bar{a},\bar{b}$ and the vector field
$\mathcal{X}$ are to be defined and constructed on the whole $\mathcal{U}$
instead of $\mathcal{U}_{*}$, and then it follows directly from (\ref{eq:=00005B=00005CX=00005DnonHolonomic-PathFollow=00005B3D=00005D})
that $\mathcal{X}=\mathbf{0}$ on $\mathcal{P}$. This is an unavoidable consequence of
the convergence requirement (\ref{eq:Convergence}) under the transversality
assumption in Assumption \ref{assu:BasicAssumption}. 
\end{rem}

\subsection{First Analysis on $\mathcal{X}$ and the Main Result}

It is the main task of this paper to look for solutions to Problem \ref{prob:MainProblem-concrete=000026specified}. To see how each part $\mathcal{X}$ in (\ref{eq:=00005B=00005CX=00005DnonHolonomic-PathFollow=00005B3D=00005D})
contributes to the motion, note that
\begin{equation}
d\mathfrak{H}\big(\mathcal{V}_{\beta}\times\nabla\mathfrak{H}\big)=\big(\mathcal{V}_{\beta}\times\nabla\mathfrak{H}\big)\cdot\nabla\mathfrak{H}\equiv0,\label{eq:ConvergenceAnalysis-RotationTerm}
\end{equation}
while $\mathcal{V}_{\beta}\times\big(\frac{\partial}{\partial\theta}\times\nabla\mathfrak{H}\big)=\big(\mathcal{V}_{\beta}\cdot\nabla\mathfrak{H}\big)\frac{\partial}{\partial\theta}-\big(\mathcal{V}_{\beta}\cdot\frac{\partial}{\partial\theta}\big)\nabla\mathfrak{H}$,
and then
\begin{equation}
\begin{aligned}\bigg(\mathcal{V}_{\beta}\times\big(\frac{\partial}{\partial\theta}\times\nabla\mathfrak{H}\big)\bigg)\cdot\nabla\mathfrak{H}= & \bigg(\big(\mathcal{V}_{\beta}\cdot\nabla\mathfrak{H}\big)\frac{\partial}{\partial\theta}-\big(\mathcal{V}_{\beta}\cdot\frac{\partial}{\partial\theta}\big)\nabla\mathfrak{H}\bigg)\cdot\nabla\mathfrak{H}\\
= & 0-\beta\big(\frac{\partial}{\partial\theta}\big)\big|\big|\nabla\mathfrak{H}\big|\big|^{2}.
\end{aligned}
\label{eq:ConvergenceAnalysis-ConvergenceTerm}
\end{equation}
Therefore, the part $\mathcal{V}_{\beta}\times\nabla\mathfrak{H}$
neither increases nor decreases the value of $\mathfrak{H}$, and
the convergence/deviation to/from $\mathcal{P}$ results solely from
$\mathcal{V}_{\beta}\times\big(\frac{\partial}{\partial\theta}\times\nabla\mathfrak{H}\big)$. 

For Problem \ref{prob:MainProblem-concrete=000026specified},
it would be desirable to have appropriate functions $\bar{a}$, $\bar{b}$
such that, for example, $d\theta(\mathcal{X})>0$ and $d\mathfrak{H}(\mathcal{X})<0$
hold simultaneously. However, Proposition \ref{prop:onDiscSignChange}
below suggests that this is not achievable. 
\begin{prop}
\label{prop:onDiscSignChange}Given a specific diffeomorphism $\mathcal{U}\cong\mathcal{B}_{\delta}\times S^{1}$,
if $d\theta\wedge\beta\neq0$ holds everywhere on $\mathcal{P}$,
then for any vector field $\mathcal{X}$ with $\beta(\mathcal{X})=0$,
$d\mathfrak{H}(\mathcal{X})$ and $d\theta(\mathcal{X})$ cannot be
free of zero points simultaneously. Precisely, if $d\theta(\mathcal{X})\neq0$
on some disk $\mathcal{B}_{\delta}^{*}\times\{e^{i\bar{\theta}}\}$,
then $d\mathfrak{H}(\mathcal{X})$ has to change its sign on the disk.
\end{prop}
\begin{proof}
See the appendix.
\end{proof}
Proposition \ref{prop:onDiscSignChange} reflects the difficulty in
constructing and analyzing $\mathcal{X}$ for the main problems,
and it is also the reason in Problems \ref{prob:=00005BGeneralMainProblem=00005D3D-nonHolonomic-PathFollowing},
\ref{prob:MainProblem-concrete=000026specified} we have adopted (\ref{eq:Circling})
and (\ref{eq:Convergence}) as the requirements instead of asking
for $d\mathfrak{H}(\mathcal{X})<0$ and $d\theta(\mathcal{X})>0$
to hold simultaneously.\footnote{It should be noted that Proposition \ref{prop:onDiscSignChange} does
not rule out the possibility to have both $d\mathfrak{H}(\mathcal{X})<0$
and $d\theta(\mathcal{X})>0$ on some trajectory.} The solution in this paper to Problems \ref{prob:=00005BGeneralMainProblem=00005D3D-nonHolonomic-PathFollowing}
and \ref{prob:MainProblem-concrete=000026specified} relies on the
non-integrability of $\ker\beta$, and it is demonstrated below as
the main result of this work:
\begin{thm}
\label{thm:=00005BMainResult=00005D}$\mathrm{[Main Result]}$ Let $\Omega_{\mathfrak{e}}$
be the standard volume form on $\mathbb{R}^{3}$ and suppose that
$\beta\wedge d\beta=\lambda_{\beta}\Omega_{\mathfrak{e}}$ with $|\lambda_{\beta}|>0$
on $\mathcal{U}$. Without loss of generality, assume $\textcolor{violet}{\beta\big(\frac{\partial}{\partial\theta}\big)>0}$. For any  smooth functions $\bar{a},\bar{b}$ on $\mathcal{U}$
satisfying the conditions \textcolor{purple}{$\bar{a}\cdot\lambda_{\beta}>0$},
$\bar{b}>0$ on $\mathcal{U}_{*}$, and $\underset{\mathcal{U}_{*}}{\sup}\frac{\bar{b}}{\mathfrak{H}}<\infty$
(i.e., $\underset{\mathcal{U}_{*}}{\sup}\frac{\bar{b}}{r^{2}}<\infty$),
the vector field $\mathcal{X}$ given in (\ref{eq:=00005B=00005CX=00005DnonHolonomic-PathFollow=00005B3D=00005D}) solves Problem \ref{prob:MainProblem-concrete=000026specified}.
In particular, the pair $(\bar{a},\bar{b})=(\lambda_{\beta},\mathfrak{H})$
fulfills all these conditions and gives a simple solution.
\end{thm}
\begin{rem}
Throughout this paper, we use $(x_{1},x_{2},x_{3})$ as the global
variables for all points in $\mathbb{R}^{3}$, and the standard
volume form on $\mathbb{R}^{3}$ is
\begin{equation}
\Omega_{\mathfrak{e}}:=dx_{1}\wedge dx_{2}\wedge dx_{3}.\label{eq:EuclideanVolumeForm}
\end{equation}
\end{rem}
\begin{rem}
The non-integrability of the constraint is reflected in the theorem
by the condition $\beta\wedge d\beta=\lambda_{\beta}\Omega_{\mathfrak{e}}$
with $|\lambda_{\beta}|>0$ on $\mathcal{U}$. 
\end{rem}
From the analysis with (\ref{eq:ConvergenceAnalysis-RotationTerm})
and (\ref{eq:ConvergenceAnalysis-ConvergenceTerm}) we already know
that, to fulfill the requirement for convergence (\ref{eq:Convergence})
in Problem \ref{prob:=00005BGeneralMainProblem=00005D3D-nonHolonomic-PathFollowing},
it suffices to make $\bar{b}>0$ on $\mathcal{U}_{*}$. The effort
in this paper for proving Theorem \ref{thm:=00005BMainResult=00005D}
will then concentrate on the circling part (\ref{eq:Circling}) of
the problem. 

\section{\label{sec:=00005BPrerequisite=00005D=000026=00005BLayout=00005D}Some
Preparation, and Layout of the Paper}

\subsection{Justification for (\ref{eq:=00005B=00005CX=00005DnonHolonomic-PathFollow=00005B3D=00005D})
and Theorem \ref{thm:=00005BMainResult=00005D}}

We shall justify the representation by (\ref{eq:=00005B=00005CX=00005DnonHolonomic-PathFollow=00005B3D=00005D})
as well as the result in Theorem \ref{thm:=00005BMainResult=00005D}
from a practical point of view. Suppose that this is an actual mechanical
system, and $\mathbb{R}^{3}$ with the variables $\mathbf{x}=(x_{1},x_{2},x_{3})$
is a global (fixed) coordinate system. For the sake of practicality,
a suitable representation for $\mathcal{X}$ should have all the parts
in the expression computable directly from the global coordinates with
basic math operations that are independent of the choice of the diffeomorphism
$\mathcal{U}\rightarrow\mathcal{B}_{\delta}\times S^{1}$. The operations
in (\ref{eq:=00005B=00005CX=00005DnonHolonomic-PathFollow=00005B3D=00005D})
only involves the usual dot product and the cross product of the vectors
in $\mathbb{R}^{3}$, and it remains to justify the vectors $\mathcal{V}_{\beta}$,
$\nabla\mathfrak{H}$ and $\frac{\partial}{\partial\theta}$ that
appear in the formula. With the constraint given by
\[
\beta=b_{1}dx_{1}+b_{2}dx_{2}+b_{3}dx_{3},
\]
the expression for $\mathcal{V}_{\beta}$ in the global coordinates is
\[
\mathcal{V}_{\beta}=(b_{1},b_{2},b_{3})=b_{1}\frac{\partial}{\partial x_{1}}+b_{2}\frac{\partial}{\partial x_{2}}+b_{3}\frac{\partial}{\partial x_{3}}.
\]
In practice \cite{yao2021singularity}, the desired path $\mathcal{P}$
can be given by a set of equations with functions $f,g$ in $\mathbf{x}$
\[
 \mathcal{P} = \{\mathbf{x} \in \mathbb{R}^3 : f(\mathbf{x})=0, \; g(\mathbf{x})=0  \}
\]
such that the gradients $\nabla f$ and $\nabla g$ are linearly independent
on $\mathcal{P}$. In other words, $\mathcal{P}$ is assumed to be a compact regular
level set of the smooth map $\mathfrak{P}:\mathbb{R}^{3}\rightarrow\mathbb{R}^{2}$
given by $\mathfrak{P}=(f,g)$, or at least, $\mathcal{P}\cong S^{1}$
is a connected component of $\mathfrak{P}^{-1}(\mathbf{0})$. By Ehresmann's
theorem (see \cite{Cushman2015GlobalAspectsIntegrableSystem}), there
exists a neighborhood $\mathcal{U}$ of $\mathcal{P}$ together with
a diffeomorphism $\mathcal{U}\xrightarrow{\mathfrak{Eh}}\mathcal{B}_{\delta}\times S^{1}$,
such that
\begin{equation}
\mathfrak{P}=\mathfrak{p}\circ\mathfrak{Eh},\label{eq:Ehresmann-Trivialization}
\end{equation}
where $\mathfrak{p}:(x,y,e^{i\theta})\mapsto(x,y)$ is the natural
projection from $\mathcal{B}_{\delta}\times S^{1}$ to $\mathcal{B}_{\delta}$.
With this diffeomorphism, the Lyapunov function $\mathfrak{H}$ given
in (\ref{eq:LyapnovFunction-H}) is then just
\begin{equation}
\mathfrak{H}(\mathbf{x})=f^{2}(\mathbf{x})+g^{2}(\mathbf{x})\label{eq:=00005BLabCoordinates=00005DLyapnovFunction-H}
\end{equation}
in the global coordinates $\mathbf{x}=(x_{1},x_{2},x_{3})$, and therefore
with the standard gradient operator $\nabla=\big(\frac{\partial}{\partial x_{1}},\frac{\partial}{\partial x_{2}},\frac{\partial}{\partial x_{3}}\big)$
on $\mathbb{R}^{3}$,
\[
\nabla\mathfrak{H}=f\nabla f+g\nabla g,
\]
where the scalar coefficients are omitted without loss of generality. 
The vector field $\frac{\partial}{\partial\theta}$ is tangent to
each fiber $\{\mathbf{z}\}\times S^{1}$ in $\mathcal{B}_{\delta}\times S^{1}$,
while due to the relation (\ref{eq:Ehresmann-Trivialization}), these
$S^{1}$-fibers are exactly the level sets of $\mathfrak{P}$ in $\mathcal{U}$,
which are perpendicular to $\nabla f$ and $\nabla g$. Therefore,
$\nabla f\times\nabla g$ is parallel to $\frac{\partial}{\partial\theta}$,
and hence we can simply replace $\frac{\partial}{\partial\theta}$
by $\nabla f\times\nabla g$ (with a multiple) when using (\ref{eq:=00005B=00005CX=00005DnonHolonomic-PathFollow=00005B3D=00005D}).

\subsection{\label{subsec:Concept-Definition-Notation}Prerequisites and Preparations}

As is mentioned at the end of Section \ref{sec:PathFollo-R^3}, the
solution to the main problem will rely on the integrability of the
distribution $\ker\beta$ of admissible velocities. The major geometric
structure to be studied in this work is then the bundle $\mathcal{U}\xrightarrow{\mathfrak{P}}\mathcal{B}_{\delta}$
equipped with the horizontal distribution (Ehresmann connection) $\ker\beta$.
Through the diffeomorphism $\mathcal{U}\xrightarrow{\mathfrak{Eh}}\mathcal{B}_{\delta}\times S^{1}$,
$\ker\beta$ can also be viewed as a distribution on the space $\mathcal{B}_{\delta}\times S^{1}$.
Therefore, in the following sections, we identify the structures $(\mathcal{U},\ker\beta)\xrightarrow{\mathfrak{P}}\mathcal{B}_{\delta}$
and $(\mathcal{B}_{\delta}\times S^{1},\ker\beta)\xrightarrow{\mathfrak{p}}\mathcal{B}_{\delta}$,
using the corresponding symbols interchangeably. Note that $\mathfrak{P}$
is a submersion with $\ker\mathfrak{P}_{*}=\text{span}\{\frac{\partial}{\partial\theta}\}$,
and then at each $p\in\mathcal{U}$ with $\mathbf{z}=\mathfrak{P}(p)$,
the tangent map $\mathfrak{P}_{*}$ gives an isomorphism between the
tangent (sub)spaces $\ker\beta\big|_{p}$ and $\mathrm{T}_{\mathbf{z}}\mathcal{B}_{\delta}$. In the following, we shall introduce some basic concepts and results related to the structure $(\mathcal{U},\ker\beta)$.
\subsubsection*{Horizontal Lift and Parallel Transport}
``Horizontal Lift'' and ``Parallel Transport'' are basic concepts
associated to a horizontal connection. For any path $\gamma:[0,1]\rightarrow\mathcal{B}_{\delta}$,
a horizontal lift of $\gamma$ in $\mathcal{U}$ is a path $\bar{\gamma}:[0,1]\rightarrow\mathcal{U}$
such that $\frac{d\bar{\gamma}}{dt}\in\ker\beta$. At each $p\in\mathfrak{P}^{-1}\big\{\gamma(0)\big\}$,
$\gamma$ has a unique horizontal lift $\bar{\gamma}$ with $\bar{\gamma}(0)=p$.
Therefore, the points on the fiber $\big\{\gamma(0)\big\}\times S^{1}$
can be connected to those on the fiber $\big\{\gamma(1)\big\}\times S^{1}$
with horizontal lifts $\bar{\gamma}$ of $\gamma$, and we call $\bar{\gamma}(1)$
the parallel transport of $\bar{\gamma}(0)$ on $\big\{\gamma(1)\big\}\times S^{1}$.
Given $p$ with $\mathbf{z}=\mathfrak{P}(p)$, since (the restriction
of) the tangent map $\mathfrak{P}_{*}$ is an isomorphism between
$\ker\beta\big|_{p}$ and $\mathrm{T}_{\mathbf{z}}\mathcal{B}_{\delta}$,
each vector field $\partial$ on $\mathcal{B}_{\delta}$ also has
an unique ``horizontal lift'' on $\mathcal{U}$, which is the vector
field $\bar{\partial}$ uniquely determined by $\bar{\partial}\big|_{p}\in\ker\beta\big|_{p}$
and $\mathfrak{P}_{*}(\bar{\partial}\big|_{p})=\partial\big|_{\mathfrak{P}(p)}$
at each $p\in\mathcal{U}$. Note that $\partial$ can also be treated
as a vector field on $\mathcal{U}$, and then in general, $\bar{\partial}$
takes the form
\begin{equation}
\bar{\partial}=\kappa\frac{\partial}{\partial\theta}+\partial.\label{eq:=00005BGeneralForm=00005DLifted-VectorField}
\end{equation}

The following vector fields and their notations will be used throughout
the discussion in Sections \ref{sec:Structures-of-Constraint} and
\ref{sec:nonHolo-PathFollowing=00005BR^3=00005D}. Notating the points
in $\mathcal{B}_{\delta}$ with the variables $(x,y)$ and the points
in $S^{1}$ with $e^{i\theta}$, we denote by $\bar{\partial}_{x}$
the horizontal lift of the vector field $\frac{\partial}{\partial x}$
on $\mathcal{B}_{\delta}$, and by $\bar{\partial}_{y}$ the lift
of vector field $\frac{\partial}{\partial y}$ on $\mathcal{B}_{\delta}$.
That is, $\bar{\partial}_{x}$ and $\bar{\partial}_{y}$ are the vector
fields on $\mathcal{U}=\mathcal{B}_{\delta}\times S^{1}$ determined
by $\mathfrak{P}_{*}(\bar{\partial}_{x})=\frac{\partial}{\partial x}$
with $\beta(\bar{\partial}_{x})=0$, and, $\mathfrak{P}_{*}(\bar{\partial}_{y})=\frac{\partial}{\partial y}$
with $\beta(\bar{\partial}_{y})=0$, respectively. Let $\partial_{\phi}=x\frac{\partial}{\partial y}-y\frac{\partial}{\partial x}$
and $\partial_{r}=x\frac{\partial}{\partial x}+y\frac{\partial}{\partial y}$,
and note that they can be treated as vector fields either on $\mathcal{U}$
or on $\mathcal{B}_{\delta}$. Then, $\bar{\partial}_{\phi}$ and
$\bar{\partial}_{r}$ are the horizontal lifts of $\partial_{\phi}$
and $\partial_{r}$, respectively. Applying the general form (\ref{eq:=00005BGeneralForm=00005DLifted-VectorField})
of horizontal lifts, we may check that for any top form $\Omega$
on $\mathcal{U}$, it holds
\begin{equation}
\Omega\big(\frac{\partial}{\partial\theta},\bar{\partial}_{x},\bar{\partial}_{y}\big)=\Omega\big(\frac{\partial}{\partial\theta},\frac{\partial}{\partial x},\frac{\partial}{\partial y}\big)\ \text{ and }\ \Omega\big(\frac{\partial}{\partial\theta},\bar{\partial}_{\phi},\bar{\partial}_{r}\big)=\Omega\big(\frac{\partial}{\partial\theta},\partial_{\phi},\partial_{r}\big).\label{eq:BasicRelation4=00005BTopForm=00005D}
\end{equation}

Two constructions based on ``Horizontal Lift'' and ``Parallel Transport''
will play crucial roles in the proofs. The first construction is a
map $\psi_{\bar{\mathbf{z}}}$ associated to an arbitrary $\bar{\mathbf{z}}\in\mathcal{B}_{\delta}$
with the domain $S^{1}\times[0,1]$ and the codomain $\mathcal{U}$,
such that $t\mapsto\psi_{\bar{\mathbf{z}}}^{t}(e^{i\theta})$ is the
horizontal lift of $t\mapsto t\bar{\mathbf{z}}$ with $\psi_{\bar{\mathbf{z}}}^{t=0}(e^{i\theta})=(\mathbf{0},e^{i\theta})$.
The second construction, denoted by $\Theta$ and called the parallel
projection onto $\mathcal{P}$, is the map mapping each $p\in\mathcal{U}$
to its parallel transport $\mathbf{0}^{p}$ on the central fiber $\mathcal{P}=\{\mathbf{0}\}\times S^{1}$.
Following directly from their definitions, $\psi_{\bar{\mathbf{z}}}$
and $\Theta$ are related by
\begin{equation}
\Theta\circ\psi_{\bar{\mathbf{z}}}^{t}(e^{i\theta})=(\mathbf{0},e^{i\theta})\ \text{ and }\ \psi_{\mathfrak{P}(p)}^{1}(\mathbf{0}^{p}).\label{eq:=00005BParallel=00005DProjection=000026Parametrization}
\end{equation}
Both $\Theta$ and $\psi_{\bar{\mathbf{z}}}^{t}$ are smooth maps,
and more details about them will be given in Subsection \ref{subsec:ParallelProjection=000026StrategicLemma}. 

\subsubsection*{Basic Results of Differential Calculus on $\big(\mathcal{U},\ker\beta\big)$}
Given two vector fields $\partial_{0}$ and $\partial_{1}$ as well
as a differential $1$-form $\alpha$, the following formula holds for the
exterior derivative
\begin{equation}
d\alpha(\partial_{0},\partial_{1})=\partial_{0}\big(\alpha(\partial_{1})\big)-\partial_{1}\big(\alpha(\partial_{0})\big)-\alpha\big([\partial_{0},\partial_{1}]\big),\label{eq:ExteriorDerivative-1Form}
\end{equation}
in which $[\partial_{0},\partial_{1}]$ is the Lie Bracket of the
vector fields $\partial_{0},\partial_{1}$. Also note that it is 
custom in differential geometry to identify vector fields with first-order
differential operators. For more details about (\ref{eq:ExteriorDerivative-1Form})
and other basics of the differential calculus on manifolds, we refer
to the classical textbooks \cite{Lee2012SmoothManifold,marsden1999introduction}. According to this formula, for any horizontal vector fields $\bar{\partial}_0$ and $\bar{\partial}_1$ on $\mathcal{U}$  with $\beta(\bar{\partial}_0)=\beta(\bar{\partial}_1)=0$, it holds
\begin{equation}
d\beta(\bar{\partial}_0,\bar{\partial}_1)=-\beta\big([\bar{\partial}_0,\bar{\partial}_1]\big).\label{eq:ExteriorDerivative-beta}
\end{equation}

Note that $\mathcal{U}$ is a compact subspace of $\mathbb{R}^{3}$
with variables $(x_{1},x_{2},x_{3})$, and by mapping it to $\mathcal{B}_{\delta}\times S^{1}$
with variables $(x,y,e^{i\theta})$, $\mathfrak{Eh}$ actually serves
as a change of coordinates. From such a perspective, the relation
(\ref{eq:Ehresmann-Trivialization}) (together with the definitions
of $\mathfrak{P}$ and $\mathfrak{p}$) indicates that, while initially
defined on $\mathbb{R}^{3}$ with variables $(x_{1},x_{2},x_{3})$,
the functions $f,g$ and $\mathfrak{H}:=f^{2}+g^{2}$ restricted to
$\mathcal{U}$ are expressed in the new coordinates $(x,y,e^{i\theta})$
by $f(x,y,e^{i\theta})=x$ and $g(x,y,e^{i\theta})=y$, and,
\[
\mathfrak{H}(x,y,e^{i\theta})=x^{2}+y^{2}=r^{2}.
\]
Now that $\nabla\mathfrak{H}=f\nabla f+g\nabla g$, the gradient $\nabla\mathfrak{H}$
has the expression on $\mathcal{U}$
\[
\nabla\mathfrak{H}(x,y,e^{i\theta})=x\cdot\nabla f\big|_{\mathfrak{Eh}^{-1}(x,y,e^{i\theta})}+y\cdot\nabla g\big|_{\mathfrak{Eh}^{-1}(x,y,e^{i\theta})}.
\]
Note that $\nabla f$ and $\nabla g$ are linearly independent on
$\mathcal{U}$, and then observe from the above equation that $\big|\big|\nabla\mathfrak{H}\big|\big|$
of an infinitesimal of order $r$ as $r\rightarrow0$. To be accurate,
\begin{equation}
\label{Ineq:InfinitesimalOrder=gradH}
0<\underset{\mathcal{U}_{*}}{\inf}\frac{\big|\big|\nabla\mathfrak{H}\big|\big|^{2}}{r^{2}}\leq\underset{\mathcal{U}_{*}}{\sup}\frac{\big|\big|\nabla\mathfrak{H}\big|\big|^{2}}{r^{2}}<\infty.
\end{equation}
While this is a standard result, the inequality will be needed in the discussion, and thus we
provide a proof for it in the appendix.

\subsubsection*{Integrability of $\ker\beta$ and Complete non-Holonomicity}
Since the notion of completely non-holonomic constraints plays the
central role in this work, we shall give some explanation on this
notion within the context of the paper. For an exposition on the general
concept, we refer to \cite{Sachkov2022GeometricControl}. Denote by
$\mathfrak{K}$ the set of all smooth vector fields on $\mathbb{R}^{3}$,
and by $\mathfrak{D}^{\beta}$ the set of all smooth vector fields
that are subject to the constraint $\beta=0$. That is,
\[
\mathfrak{D}^{\beta}:=\{\mathcal{V}\in\mathfrak{K}\big|\ \beta(\mathcal{V})=0\}.
\]
Endowed with the Lie bracket $[\,,\,]$ of smooth vector fields, $\mathfrak{K}$
becomes a Lie algebra. Let $\overline{\mathfrak{D}^{\beta}}$ be the
minimal sub-algebra of $\mathfrak{K}$ containing $\mathfrak{D}^{\beta}$.
The constraint $\beta=0$ is completely non-holonomic on $\mathcal{U}$
if and only if $\mathfrak{D}_{p}^{\beta}=\mathrm{T}_{p}\mathbb{R}^{3}$
for all $p\in\mathcal{U}$. Here $\mathrm{T}_{p}\mathbb{R}^{3}$ is
the tangent space of $\mathbb{R}^{3}$ at $p\in\mathcal{U}$, and,
$\mathfrak{D}_{p}^{\beta}$ is the subspace of $\mathrm{T}_{p}\mathbb{R}^{3}$
spanned by all vectors taking the form $\mathcal{V}_{p}$ for some
$\mathcal{V}\in\mathfrak{D}^{\beta}$. That is,
\[
\mathfrak{D}_{p}^{\beta}=\{v\in\mathrm{T}_{p}\mathbb{R}^{3}\big|\,\exists\mathcal{V}\in\mathfrak{D}^{\beta}\ \text{s.t.}\,v=\mathcal{V}_{p}\},\ \forall p\in\mathcal{U}.
\]
Now that the distribution $\ker\beta$ is a $2$-dimensional distribution
on the $3$-dimensional space $\mathcal{U}$, and then the constraint
$\beta=0$ being completely non-holonomic on $\mathcal{U}$ is equivalent
to the condition
\begin{equation}
\beta\big([\bar{\partial}_{x},\bar{\partial}_{y}]\big)\neq0\ \text{at every }p\in\mathcal{U}.\label{eq:EuiqvalentCondition1-=00005BCompletelynonHolonomic=00005D}
\end{equation}
From (\ref{eq:ExteriorDerivative-1Form}) and (\ref{eq:ExteriorDerivative-beta}) we have $d\beta(\bar{\partial}_{x},\bar{\partial}_{y})=-\beta\big([\bar{\partial}_{x},\bar{\partial}_{y}]\big)$,
and hence this condition also means
\begin{equation}
d\beta(\bar{\partial}_{x},\bar{\partial}_{y})\neq0\ \text{at every }p\in\mathcal{U}.\label{eq:EuiqvalentCondition2-=00005BCompletelynonHolonomic=00005D}
\end{equation}
Furthermore, check that $\beta\wedge d\beta(\frac{\partial}{\partial\theta},\bar{\partial}_{x},\bar{\partial}_{y})=\beta\big(\frac{\partial}{\partial\theta}\big)\cdot d\beta(\bar{\partial}_{x},\bar{\partial}_{y})$.
Since by assumption $\beta\big(\frac{\partial}{\partial\theta}\big)\neq0$
holds everywhere on $\mathcal{U}$, we thus obtain the third equivalent
statement of the condition:
\begin{equation}
\beta\wedge d\beta\text{ is nondegenerate on }\mathcal{U},\label{eq:EuiqvalentCondition3-=00005BCompletelynonHolonomic=00005D}
\end{equation}
which also means the existence of smooth functions $\rho,\lambda_{\beta}$ with $\rho,\lambda_{\beta}\neq 0$
everywhere on $\mathcal{U}$ such that
\begin{equation}
\beta\wedge d\beta=\rho\cdot d\theta\wedge dx\wedge dy = \lambda_{\beta}\cdot dx_1\wedge dx_2 \wedge dx_3.\label{eq:EuiqvalentCondition4-=00005BCompletelynonHolonomic=00005D}
\end{equation}
In the following discussion, we will use (\ref{eq:EuiqvalentCondition1-=00005BCompletelynonHolonomic=00005D})
- (\ref{eq:EuiqvalentCondition4-=00005BCompletelynonHolonomic=00005D})
interchangeably for the complete nonholonomicity of $\beta$.

\subsection{Layout of the Paper}

Here we shall give an introduction to the content and the layout of
this paper. Note that (\ref{eq:=00005B=00005CX=00005DnonHolonomic-PathFollow=00005B3D=00005D})
has been derived from Relation (\ref{eq:StructureOnR^3-=00005CV=00005Ctimes=00005CX}),
which reflects a general structure of constrained motions. For the
sake of clarity and generality, we devote Section \ref{sec:General-Structure-of-Motion}
for an introduction to a general representation of vector fields subject
to constraints in terms of differential forms based on duality (\ref{eq:BasicDuality})
(given in Subsection \ref{subsec:VectorFieldsByDuality}). The result
about such a representation is formally stated as Theorem \ref{thm:GeneralStructure-of-Motion}
in Subsection \ref{subsec:Systems-with-Constraints}. We have to point
out that this is a standard result, and on Euclidean spaces, the representation
is essentially the wedge product of the vector representations of
the constraints, which has been known and used in the literature of
vakonomic/Nambu mechanics \cite{Llibre2016InverseProblem,Llibre2023DynamicsODE}. 

To prepare for the the final proof of Theorem \ref{thm:=00005BMainResult=00005D},
in Section \ref{sec:Structures-of-Constraint}, we study the geometric
structure $\big(\mathcal{U},\ker\beta\big)$. $(\mathcal{U},\ker\beta)\xrightarrow{\mathfrak{P}}\mathcal{B}_{\delta}$,
or equivalently, $(\mathcal{B}_{\delta}\times S^{1},\ker\beta)\xrightarrow{\mathfrak{p}}\mathcal{B}_{\delta}$,
is then a bundle equipped with the horizontal connection $\ker\beta$.
Note that $\ker\beta$ is just a general Ehresmann connection, which
we do not assume to be invariant under any possible $S^{1}$ action
on the space $\mathcal{U}\cong\mathcal{B}_{\delta}\times S^{1}$ that
preserves the fibers. That is, $\big(\mathcal{U},\ker\beta\big)$
is not a principal $S^{1}$-bundle with an invariant connection, and
thus the curvature of $\ker\beta$ cannot be characterized as straightforwardly
as that of an invariant connection on a principal bundle with a $2$-form.
However, for the circling (\ref{eq:Circling}) we still need to extract
and demonstrate this geometric feature of the non-integrability of
$\ker\beta$, and this is what has been done in Section \ref{sec:Structures-of-Constraint}.
At the end of this section, we also answer Problem \ref{prob:=00005BGeneralMainProblem=00005D3D-nonHolonomic-PathFollowing}
with Theorem \ref{thm:=00005BMainResult=00005DSolution-to-GeneralProblem},
which confirms the existence of desirable vector fields in the case
where the constraint is completely non-holonomic. 
\begin{rem}
For an introduction to invariant connections on principal bundles,
see for example, \cite{Walschap2004MetricStructure}. For an exposition
on general Ehresmann connections as well as the Ehresmann Fibration
Theorem, we refer to \cite{Cushman2015GlobalAspectsIntegrableSystem}. 
\end{rem}
However, Theorem \ref{thm:=00005BMainResult=00005DSolution-to-GeneralProblem}
does not directly answer Problem \ref{prob:MainProblem-concrete=000026specified},
since it does not address $\mathcal{X}$ through \eqref{eq:=00005B=00005CX=00005DnonHolonomic-PathFollow=00005B3D=00005D}.
Therefore, at the end of Section \ref{sec:Structures-of-Constraint},
conditions for constructing the weight functions $\bar{a},\bar{b}$
still remains to be specified. Following the methodology established
in Section \ref{sec:Structures-of-Constraint}, in Section \ref{sec:nonHolo-PathFollowing=00005BR^3=00005D}
we analyze the motion generated by $\mathcal{X}$ based on \ref{eq:=00005B=00005CX=00005DnonHolonomic-PathFollow=00005B3D=00005D}
and finish the discussion in the paper with a proof for Theorem \ref{thm:=00005BMainResult=00005D}.
It is worth pointing out that, while from the methodological point
of view, the discussion in Section \ref{sec:Structures-of-Constraint}
lays down the foundation for Section \ref{sec:nonHolo-PathFollowing=00005BR^3=00005D},
the exploration in Section \ref{sec:nonHolo-PathFollowing=00005BR^3=00005D}
has been conducted and demonstrated in a self-contained way so that
it can be read independently of Section \ref{sec:nonHolo-PathFollowing=00005BR^3=00005D}.

\section{\label{sec:General-Structure-of-Motion}General Representation of
Motion under Constraints}

\subsection{\label{subsec:VectorFieldsByDuality}Representing Vector Fields via
Duality}

Consider a smooth manifold $\mathcal{M}$ of dimension $m$ with a
volume form $\Omega$. For each $p\in\mathcal{M}$, denote by $\Lambda^k_p$ the space of algebraic $k$-forms on the tangent space $\mathrm{T}_p\mathcal{M}$. $\Lambda^k(\mathcal{M}):=\underset{p\in\mathcal{M}}{\bigsqcup}\Lambda^k_p$ is then the bundle of $k$-forms on $\mathcal{M}$. Let $\sigma$ be an $(m-1)$-form on $\mathcal{M}$.
For each $p\in\mathcal{M}$, the space $\Lambda_p^{m}$ of top forms on $\mathrm{T}_p\mathcal{M}$ is isomorphic
to $\mathbb{R}$, and then the mapping 
\[
\Lambda_p\ni\mathfrak{a}\mapsto\mathfrak{a}\wedge\sigma_p\in\Lambda_p^{m}
\]
defines a linear function on $\Lambda_p$ of
the $1-$forms. By duality there exists a unique vector field $\mathcal{X}_{\sigma}$
such that 
\begin{equation}
\alpha\wedge\sigma=\alpha(\mathcal{X}_{\sigma})\Omega\label{eq:BasicDuality}
\end{equation}
holds for any differential $1$-form $\alpha\in\Lambda(\mathcal{M})$. The results below
are straightforward:
\begin{lem}
\label{lem:(non-)Vanishing-Property}For any $p\in\mathcal{M}$, $\mathcal{X}_{\sigma}(p)=0$
iff $\sigma_{p}=0$.
\end{lem}
\begin{lem}
\label{lem:Conservation-Law}Suppose $\sigma=\alpha\wedge\tau$, where
$\alpha$ is a $1$-form and $\tau$ is an $(m-2)$-form. Then, $\alpha(\mathcal{X}_{\sigma})=0$.
In particular, if $\alpha=df$ for some function $f$, then $f$ is
a conservation law of the dynamics $\mathcal{X}_{\sigma}$.
\end{lem}
\begin{proof}
By definition, $\alpha(\mathcal{X}_{\sigma})\Omega=\alpha\wedge\sigma=\alpha\wedge\big(\alpha\wedge\tau\big)=0$.
\end{proof}
\begin{lem}
\label{lem:DualityByDifferentVolumes}Let $\Omega_{0}$ and $\Omega_{1}$
be non-degenerate top forms on $\mathcal{M}$. For any $\sigma\in\Lambda^{m-1}(\mathcal{M})$,
denote by $\mathcal{X}_{\sigma}^{0}$ and $\mathcal{X}_{\sigma}^{1}$
the vector fields determined by (\ref{eq:BasicDuality}) with $\Omega=\Omega_{0}$
and $\Omega=\Omega_{1}$, respectively. Then, $\mathcal{X}_{\sigma}^{1}=\frac{1}{\lambda}\mathcal{X}_{\sigma}^{0}$,
where $\lambda$ is the smooth function on $\mathcal{M}$ with $\Omega_{1}=\lambda\Omega_{0}$
and it is nonzero everywhere on $\mathcal{M}$.
\end{lem}
\begin{proof}
It follows directly from the relation below for all differential $1$-form
$\alpha\in\Lambda(\mathcal{M})$:
\[
\alpha(\mathcal{X}_{\sigma}^{0})\Omega_{0}=\alpha\wedge\sigma=\alpha(\mathcal{X}_{\sigma}^{1})\Omega_{1}=\alpha(\mathcal{X}_{\sigma}^{1})\cdot\varrho\Omega_{0}=\alpha(\varrho\mathcal{X}_{\sigma}^{1})\Omega_{0}.
\]
\end{proof}

\subsubsection*{Similar Structure in Hamiltonian Mechanics}

Let $\mathcal{M}$ be a $2n$-dimensional symplectic manifold with
the symplectic form $\omega$, and take the volume form $\Omega=\omega^{n}$.
Given a smooth function $H:\mathcal{M}\rightarrow\mathbb{R}$, the
Hamiltonian vector field $\mathcal{X}_{H}$ is defined by
\begin{equation}
\iota_{\mathcal{X}_{H}}\omega=dH.\label{eq:HamiltonianVectorField}
\end{equation}
Note that $\sigma=dH\wedge\omega^{n-1}$ is a $(2n-1)$-form, and
it holds
\begin{thm}
\label{thm:Formulation-for-Hamiltonian-Mechanics}$\mathcal{X}_{H}=n\mathcal{X}_{\sigma}$. 
\end{thm}
\begin{proof}
To see this, take any $1$-form $\beta$ and check that, on the one
hand, (\ref{eq:BasicDuality}) yields
\[
\iota_{\mathcal{X}_{H}}\bigg(\beta(\mathcal{X}_{\sigma})\Omega\bigg)=\iota_{\mathcal{X}_{H}}\bigg(\beta\wedge dH\wedge\omega^{n-1}\bigg)=\beta(\mathcal{X}_{H})dH\wedge\omega^{n-1},
\]
while on the other hand, 
\[
\iota_{\mathcal{X}_{H}}\bigg(\beta(\mathcal{X}_{\sigma})\Omega\bigg)=\beta(\mathcal{X}_{\sigma})\cdot\iota_{\mathcal{X}_{H}}\omega^{n}=n\beta(\mathcal{X}_{\sigma})dH\wedge\omega^{n-1}.
\]
Obviously, $\beta(\mathcal{X}_{H})\bigg|_{p}=n\beta(\mathcal{X}_{\sigma})\bigg|_{p}$
holds whenever $dH\big|_{p}\wedge\omega_{p}^{n-1}\neq0$, which is,
due to (\ref{eq:HamiltonianVectorField}) as well as the non-degeneracy
of $\omega_{p}^{n}$, equivalent to $dH\big|_{p}\neq0$. It remains
to check $\beta(\mathcal{X}_{H})\bigg|_{p}=n\beta(\mathcal{X}_{\sigma})\bigg|_{p}$
for the places where $dH\big|_{p}=0$. Note that at these points,
$\mathcal{X}_{\sigma}(p)=\mathcal{X}_{H}(p)=0$, which concludes the
proof.
\end{proof}

\subsection{\label{subsec:Systems-with-Constraints}Systems with Constraints}

Let $\alpha_{1},...,\alpha_{k}$ be $1$-forms such that $\alpha_{1}\wedge...\wedge\alpha_{k}\neq0$
everywhere on $\mathcal{M}$. Consider the (possibly non-holonomic)
constraints 
\begin{equation}
\mathfrak{C}:\ \alpha_{1}=0,...,\alpha_{k}=0.\label{eq:Constraints}
\end{equation}
According to Lemma \ref{lem:Conservation-Law}, for any $(m-k-1)$-form
$\tau$ on $\mathcal{M}$ with the $(m-1)$-form
\[
\sigma_{\tau}:=\tau\wedge\alpha_{1}\wedge...\wedge\alpha_{k},
\]
the mapping
\[
\Phi:\ \tau\mapsto\sigma_{\tau}\mapsto\mathcal{X}_{\sigma_{\tau}}
\]
is a linear map sending each $\tau$ to a vector field $\mathcal{X}_{\sigma_{\tau}}$
satisfying $\alpha_{i}(\mathcal{X}_{\sigma_{\tau}})=0$ for $i=1,...,k$.
That is, the motion generated by $\mathcal{X}_{\sigma_{\tau}}$ is
subject to the constraints $\mathfrak{C}$. 

In fact, it is also true in the other direction. With the non-degeneracy
condition $\alpha_{1}\wedge...\wedge\alpha_{k}\neq0$ (everywhere
on $\mathcal{M}$), for any motion $\mathcal{X}$ subject to $\mathfrak{C}$,
we can find a corresponding $(m-k-1)$-form $\tau$ such that $\mathcal{X}=\mathcal{X}_{\sigma_{\tau}}$.
That is, we have the following result:
\begin{thm}
\label{thm:GeneralStructure-of-Motion}Given any vector field $\mathcal{X}$
on $\mathcal{M}$, it is subject to $\mathfrak{C}$ if and only if
there exists an $(m-k-1)$-form $\tau$ such that $\mathcal{X}=\mathcal{X}_{\sigma_{\tau}}$
with
\begin{equation}
\sigma_{\tau}:=\tau\wedge\alpha_{1}\wedge...\wedge\alpha_{k}.\label{eq:Structure-of-Constrained-Motion}
\end{equation}
\end{thm}
\begin{proof}
If $\mathcal{X}=\mathcal{X}_{\sigma_{\tau}}$ with (\ref{eq:Structure-of-Constrained-Motion}),
then it is straightforward to check with (\ref{eq:BasicDuality})
that $\mathcal{X}$ is subject to the constraint $\mathfrak{C}$.
It thus remains to show the other direction: if $\mathcal{X}$ is
subject to $\mathfrak{C}$, there exists $\sigma_{\tau}$ as in (\ref{eq:Structure-of-Constrained-Motion})
such that $\mathcal{X}=\mathcal{X}_{\sigma_{\tau}}$.

Note that the correspondence $\sigma\mapsto\mathcal{X}_{\sigma}$
by duality $\beta\wedge\sigma=\beta(\mathcal{X}_{\sigma})\Omega$
between (smooth) $(m-1)$-forms and (smooth) vector fields is bijective.
As a result, given $\mathcal{X}$ subject to $\mathfrak{C}$, $\mathcal{X}=\mathcal{X}_{\sigma}$
for certain $(m-1)$-form $\sigma$, and from Lemma \ref{lem:Conservation-Law}
we know that $\alpha_{i}\wedge\sigma=0$ for $i=1,...,k$. Therefore,
it suffices to show the following implication for any smooth $(m-1)$-form $\sigma$,
\[
\alpha_{1}\wedge\sigma=...=\alpha_{k}\wedge\sigma=0\ \ \implies\ \ \sigma=\tau\wedge\alpha_{1}\wedge...\wedge\alpha_{k}
\]
with some smooth $(m-k-1)$-form $\tau$.

Assume that we already prove for each point $p\in\mathcal{M}$ the
existence of a neighborhood $\mathcal{U}$ and a local $(m-k-1)$-form
$\tau_{\mathcal{U}}$ with
\[
\sigma\big|_{\mathcal{U}}=\big(\alpha_{1}\wedge...\wedge\alpha_{k}\big)\big|_{\mathcal{U}}\wedge\tau_{\mathcal{U}}.
\]
Then, there is a locally finite open cover $\big\{\mathcal{U}_{i}\big\}$
of $\mathcal{M}$ such that on each $\mathcal{U}_{i}$,
\[
\sigma\big|_{\mathcal{U}_{i}}=\big(\alpha_{1}\wedge...\wedge\alpha_{k}\big)\big|_{\mathcal{U}_{i}}\wedge\tau_{\mathcal{U}_{i}}.
\]
Let $\{\varphi_{i}\}$ be a partition of unity of $\mathcal{M}$ subordinated
to $\big\{\mathcal{U}_{i}\big\}$, and define
\[
\tau:=\underset{i}{\sum}\varphi_{i}\cdot\tau_{\mathcal{U}_{i}}.
\]
Check that
\[
\begin{aligned}\alpha_{1}\wedge...\wedge\alpha_{k}\wedge\tau= & \alpha_{1}\wedge...\wedge\alpha_{k}\wedge\big(\underset{i}{\sum}\varphi_{i}\cdot\tau_{\mathcal{U}_{i}}\big)\\
= & \underset{i}{\sum}\varphi_{i}\cdot\big(\alpha_{1}\wedge...\wedge\alpha_{k}\wedge\tau_{\mathcal{U}_{i}}\big)\\
= & \underset{i}{\sum}\varphi_{i}\cdot\sigma\big|_{\mathcal{U}_{i}}\\
= & \sigma.
\end{aligned}
\]
Hence, the proof will be concluded by proving the existence of $\mathcal{U}$
and $\tau_{\mathcal{U}}$ for each $p$, which is done by the lemma
below.
\end{proof}
\begin{lem}
\label{lem:SplittingLemma}Suppose that $\sigma$ is a smooth $(m-1)$-form with 
\[
\alpha_{1}\wedge\sigma=...=\alpha_{k}\wedge\sigma=0.
\]
Then for each $p\in\mathcal{M}$, there exists a neighborhood $\mathcal{U}\ni p$
and a smooth $(m-k-1)$-form $\tau_{\mathcal{U}}$ on $\mathcal{U}$,
such that
\[
\sigma\big|_{\mathcal{U}}=\big(\alpha_{1}\wedge...\wedge\alpha_{k}\big)\big|_{\mathcal{U}}\wedge\tau_{\mathcal{U}}.
\]
\end{lem}
\begin{proof}
Given any $p\in\mathcal{M}$, take a local chart $\mathcal{U}$ containing
$p$. Since $\alpha_{1}\big|_{\mathcal{U}}\neq0$, by shrinking $\mathcal{U}$
if necessary, it can be extended to a local frame $\alpha_{1}\big|_{\mathcal{U}},\beta_{1},...,\beta_{m-1}$
of the cotangent bundle $T^{*}\mathcal{M}$ over $\mathcal{U}$. Therefore,
there exist $m$ functions on $\mathcal{U}$, $c_{1},...,c_{m-1}$
and $b$, such that
\[
\sigma\big|_{\mathcal{U}}=\sum_{i=1}^{m-1}c_{i}\cdot\alpha_{1}\big|_{p}\wedge\bigg(\bigwedge_{j\neq i}\beta_{j}\bigg)+b\cdot\bigwedge_{i=1}^{m-1}\beta_{i}.
\]
As a result, $\alpha_{1}\big|_{\mathcal{U}}\wedge\sigma\big|_{\mathcal{U}}=b_{p}\cdot\alpha_{1}\big|_{\mathcal{U}}\wedge (\bigwedge_{i=1}^{m-1}\beta_{i})=0$,
and then $b=0$ due to the non-degeneracy of $\alpha_{1}\big|_{\mathcal{U}}\wedge \left(\underset{i-1}{\overset{m-1}{\bigwedge}}\beta_{i} \right)$.
With this discussion, we show the existence of an $(m-2)$-form on $\mathcal{U}$,
\[
\Theta_{m-2}:=\sum_{i=1}^{m-1}c_{i}\cdot\bigg(\bigwedge_{j\neq i}\beta_{j}\bigg),
\]
such that $\sigma\big|_{\mathcal{U}}=\alpha_{1}\big|_{\mathcal{U}}\wedge\Theta_{1}$,
and in the following we complete the proof by induction.

Assuming for some integer $l\in[2,k]$ we already find an $(m-l)$-form
$\Theta_{m-l}$ on $\mathcal{U}$ such that
\[
\sigma\big|_{\mathcal{U}}=\alpha_{1}\big|_{\mathcal{U}}\wedge...\wedge\alpha_{l-1}\big|_{\mathcal{U}}\wedge\Theta_{m-l},
\]
we proceed to show the existence of $\Theta_{m-(l+1)}$ on $\mathcal{U}$
with
\[
\sigma\big|_{\mathcal{U}}=\alpha_{1}\big|_{\mathcal{U}}\wedge...\wedge\alpha_{l}\big|_{\mathcal{U}}\wedge\Theta_{m-(l+1)}.
\]
Since $\alpha_{1}\big|_{p},...,\alpha_{l}\big|_{p}$ are linearly
independent, they can be extended to a (local) frame $\alpha_{1}\big|_{\mathcal{U}}$,...,
$\alpha_{l}\big|_{\mathcal{U}}$, $\tilde{\beta}_{1}$,...,$\tilde{\beta}_{m-l-1}$
of $T_{\mathcal{U}}^{*}\mathcal{M}$. While $\Theta_{m-l}$ is not
unique, due to the assumption for induction, it suffices to consider
\[
\Theta_{m-l}=\alpha_{l}\big|_{\mathcal{U}}\wedge\bigg(\sum_{i_{1},...,i_{m-l-1}}\tilde{c}_{i_{1},...,i_{m-l-1}}\cdot\tilde{\beta}_{i_{1}}\wedge....\wedge\tilde{\beta}_{i_{m-l-1}}\bigg)+\tilde{b}\cdot\bigwedge_{i=1}^{m-l-1}\tilde{\beta}_{i}.
\]
Due to the non-degeneracy of $\alpha_{1}\wedge...\wedge\alpha_{l}\wedge (\bigwedge_{i=1}^{m-l-1}\beta_{i})$
on $\mathcal{U}$,
\[
\alpha_{l}\wedge\sigma=\tilde{b}\cdot\alpha_{l}\wedge\alpha_{1}\wedge...\wedge\alpha_{l}\wedge \left( \bigwedge_{i=1}^{m-l-1}\beta_{i}\right)  =0\ \ \implies\ \ \tilde{b}=0,
\]
and therefore, 
\[
\sigma=\big(\alpha_{1}\wedge...\wedge\alpha_{l}\big)\big|_{\mathcal{U}}\wedge\bigg(\sum_{i_{1},...,i_{m-l-1}}\tilde{c}_{i_{1},...,i_{m-l-1}}\cdot\beta_{i_{1}}\wedge....\wedge\beta_{i_{m-l-1}}\bigg)=\big(\alpha_{1}\wedge...\wedge\alpha_{l}\big)\big|_{\mathcal{U}}\wedge\Theta_{m-l-1}.
\]
By induction, we come to the conclusion that there exists an $(m-k-1)$-form $\tau_{\mathcal{U}}=\Theta_{m-k-1}$ on $\mathcal{U}$ such that
\[
\sigma\big|_{\mathcal{U}}=\big(\alpha_{1}\wedge...\wedge\alpha_{k}\big)\big|_{\mathcal{U}}\wedge\tau_{\mathcal{U}}.
\]
\end{proof}
\begin{cor}
\label{cor:VecField-on-R^3}(Structure of $\mathcal{X}$ for Problem
\ref{sec:PathFollo-R^3}) Let $\beta$ be a non-degenerate $1$-form
on an open set $\mathcal{U}\subset\mathbb{R}^{3}$ and $\Omega$ be
a non-degenerate top form. For any vector field $\mathcal{X}$ on
$\mathcal{U}$, $\beta(\mathcal{X})\equiv0$ if and only if there
is a $1$-form $\tau$ on $\mathcal{U}$ such that, for any (other)
$1$-form $\alpha$,
\[
\alpha(\mathcal{X})\Omega=\alpha\wedge\beta\wedge\tau.
\]
\end{cor}

\subsubsection*{Non-Holonomic Motion with Conservation Law}

In practice, in addition to (\ref{eq:Constraints}), there can also
be extra constraints due to conservation laws of certain observables.
For the sake of generality, we consider a map $\mathfrak{F}$ between
smooth manifolds $\mathcal{M}$ and $\mathcal{O}$
\[
\mathfrak{F}:(\mathcal{M},\Omega)\rightarrow(\mathcal{O},\varsigma),
\]
where $\Omega$ and $\varsigma$ are top forms on $\mathcal{M}$ and
$\mathcal{O}$, respectively, and $\Omega$ is nondegenerate. Here,
$\mathcal{M}$ is the state space and $\mathcal{O}$ is the space
of observables. We further assume $\dim\mathcal{O}+k<\dim\mathcal{M}$,
where $k$ is the number of constraints in (\ref{eq:Constraints}).

To generate a motion (vector field) subject to the constraints $\mathfrak{C}$
and $\mathfrak{F}\equiv o$ for some $o\in\mathcal{O}$, we may look
for an $n$-form $\Theta$ on $\mathcal{M}$ and let
\begin{equation}
\sigma=\Theta\wedge\varsigma^{*}\wedge(\alpha_{1}\wedge...\wedge\alpha_{k}),\label{eq:=00005Csigma-subjectTo-constraints}
\end{equation}
where $\varsigma^{*}=\mathfrak{F}^{*}\varsigma$ is the pullback.
Then,
\begin{thm}
With $\sigma$ in (\ref{eq:=00005Csigma-subjectTo-constraints}),
the vector field $\mathcal{X}_{\sigma}$ defined by (\ref{eq:BasicDuality})
generates a motion subject to $\mathfrak{C}$ and preserves $\mathfrak{F}$.
\end{thm}
\begin{proof}
Let $\varphi_{\sigma}$ be the flow generated by $\mathcal{X}_{\sigma}$.
Since Lemma \ref{lem:Conservation-Law} confirms $\alpha_{i}(\mathcal{X}_{\sigma})=0$,
it remains to show $\varphi_{\sigma}$ preserves the map $\mathfrak{F}$.
That is,
\[
\mathfrak{F}\circ\varphi_{\sigma}=\mathfrak{F}.
\]
To this end, we need to prove
\[
\frac{d}{dt}\mathfrak{F}\circ\varphi_{\sigma}^{t}(p)=\mathfrak{F}_{*}\bigg(\mathcal{X}_{\sigma}\circ\varphi_{\sigma}^{t}(p)\bigg)\equiv0.
\]
It suffices to show that, for an arbitrary $1$-form $\beta$ on $\mathcal{O}$,
it holds $\beta\big(\mathfrak{F}_{*}\mathcal{X}_{\sigma}\big)=\beta^{*}(\mathcal{X}_{\sigma})\equiv0$.
Check that
\[
\beta\big(\mathfrak{F}_{*}\mathcal{X}_{\sigma}\big)\Omega=\beta^{*}\wedge\Theta\wedge\varsigma^{*}\wedge(\alpha_{1}\wedge...\wedge\alpha_{k})=(-1)^{n}\Theta\wedge\beta^{*}\wedge\varsigma^{*}\wedge(\alpha_{1}\wedge...\wedge\alpha_{k}),
\]
while $\beta^{*}\wedge\varsigma^{*}=\mathfrak{F}^{*}\big(\beta\wedge\varsigma\big)=0$
since $\varsigma$ is a top form on $\mathcal{O}$. As a result, $\beta\big(\mathfrak{F}_{*}\mathcal{X}_{\sigma}\big)\Omega=0$
and then due to Lemma \ref{lem:(non-)Vanishing-Property}, we conclude that
$\beta\big(\mathfrak{F}_{*}\mathcal{X}_{\sigma}\big)=0$ for any $\beta$.
\end{proof}

\subsection{General Structure of $\mathcal{X}$ on $\mathbb{R}^{3}$}

We show that the equivalence relation \eqref{eq:StructureOnR^3-=00005CV=00005Ctimes=00005CX}
follows from Corollary \ref{cor:VecField-on-R^3} by taking $\langle\cdot,\cdot\rangle$
to be the standard Riemannian metric on $\mathbb{R}^{3}$:
\begin{equation}
\langle\underset{i=1}{\overset{3}{\sum}}u_{i}\frac{\partial}{\partial x_{i}},\underset{i=1}{\overset{3}{\sum}}v_{i}\frac{\partial}{\partial x_{i}}\rangle:=\underset{i=1}{\overset{3}{\sum}}u_{i}v_{i}\label{eq:StandardRiemannianMetric}
\end{equation}
and $\Omega_{\mathfrak{e}}$ the associated Riemannian volume form:
\begin{equation}
\Omega_{\mathfrak{e}}=dx_{1}\wedge dx_{2}\wedge dx_{3}.\label{eq:StandardRiemannianForm}
\end{equation}
With $\beta=b_{1}dx_{1}+b_{2}dx_{2}+b_{3}dx_{3}$, it holds that
\[
\mathcal{V}_{\beta}=b_{1}\frac{\partial}{\partial x_{1}}+b_{2}\frac{\partial}{\partial x_{2}}+b_{3}\frac{\partial}{\partial x_{3}}.
\]
Then, with $\tau=t_{1}dx_{1}+t_{2}dx_{2}+t_{3}dx_{3}$, for an arbitrary
$1$-form 
\[
\alpha=a_{1}dx_{1}+a_{2}dx_{2}+a_{3}dx_{3},
\]
it holds that
\[
\alpha\wedge\beta\wedge\tau=\det\left[\begin{array}{ccc}
a_{1} & a_{2} & a_{3}\\
b_{1} & b_{2} & b_{3}\\
t_{1} & t_{2} & t_{3}
\end{array}\right]dx_{1}\wedge dx_{2}\wedge dx_{3}=\alpha\big(\mathcal{V}_{\beta}\times\mathcal{V}_{\tau}\big)\Omega,
\]
where $\mathcal{V}_{\tau}=t_{1}\frac{\partial}{\partial x_{1}}+t_{2}\frac{\partial}{\partial x_{2}}+t_{3}\frac{\partial}{\partial x_{3}}$.
By Corollary \ref{cor:VecField-on-R^3}, $\mathcal{X}$ is tangent
to $\ker\beta$ if and only if there exists a $1$-form $\tau$ s.t.
\begin{equation}
\alpha\big(\mathcal{X}\big)\Omega=\alpha\wedge\beta\wedge\tau=\alpha\big(\mathcal{V}_{\beta}\times\mathcal{V}_{\tau}\big)\Omega,\label{eq:=00005BDuality=00005Dwedge-cross}
\end{equation}
which, with the arbitrariness of $\alpha$, exactly means $\mathcal{X}=\mathcal{V}_{\beta}\times\mathcal{V}_{\tau}$.
\eqref{eq:StructureOnR^3-=00005CV=00005Ctimes=00005CX} is then verified
by taking $\bar{\mathcal{X}}=\mathcal{V}_{\tau}$.
\begin{rem}
\label{rem:=00005BDuality-on-R^3=00005Dcross=000026wedge}Note that
with the standard Riemannian metric (\ref{eq:StandardRiemannianMetric}),
the vector (field) $\mathcal{V}_{\tau}$ is just the Riesz representation
of $\tau$, i.e., $\langle\mathcal{V}_{\tau},\cdot\rangle=\tau$.
In fact, from the above argument we see that, endowing $\mathbb{R}^{3}$
with the standard metric (\ref{eq:StandardRiemannianMetric}) and
the associated volume form (\ref{eq:StandardRiemannianForm}), for
any $1$-forms $\tau_{1}$ and $\tau_{2}$ on $\mathbb{R}^{3}$, the
dual to $\sigma=\tau_{1}\wedge\tau_{2}$ is exactly $\mathcal{X}_{\sigma}=\mathcal{V}_{\tau_{1}}\times\mathcal{V}_{\tau_{2}}$.
\end{rem}

\section{\label{sec:Structures-of-Constraint}Structures of $\big(\mathcal{U},\ker\beta\big)$}

Now we come back to the discussion of Problem \ref{prob:=00005BGeneralMainProblem=00005D3D-nonHolonomic-PathFollowing}.
It is essentially a path-following problem on the $3$-dimensional
space $\mathcal{U}\cong\mathcal{B}_{\delta}\times S^{1}$, which is
(diffeomorphic to) a tubular neighborhood of a loop in any orientable
$3$-manifold $\mathcal{M}$. As is mentioned previously, despite
the obstruction indicated in Proposition \ref{prop:onDiscSignChange},
circling in (\ref{eq:Circling}) and converging in (\ref{eq:Convergence})
can still be achieved simultaneously by exploiting the non-integrability
of $\ker\beta$. 
\begin{rem}
\label{rem:Normalization}Since $\beta(\frac{\partial}{\partial\theta})\neq0$
on $\mathcal{U}$, for simplicity we may replace $\beta$ with $\frac{1}{\beta(\frac{\partial}{\partial\theta})}\beta$
and still denote it by $\beta$. This does not change the distribution
$\ker\beta$, while the expression of $\beta$ becomes
\begin{equation}
\beta=d\theta+\hat{a}dx+\hat{b}dy.\label{eq:=00005Cbeta-Normalized}
\end{equation}
For simplicity, this form of $\beta$ will be adopted in this section
and the next. 
\end{rem}

\subsection{Winding Component $\mathcal{V}_{\beta}\times\nabla\mathfrak{H}$
and Curvature of $\ker\beta$}

Recall from (\ref{eq:=00005B=00005CX=00005DnonHolonomic-PathFollow=00005B3D=00005D})
that, any vector field $\mathcal{X}$ on $\mathcal{U}\cong\mathcal{B}_{\delta}\times S^{1}$
with $\beta(\mathcal{X})\equiv0$ takes the form
\[
\begin{aligned}\mathcal{X}= & \bar{a}\cdot\mathcal{V}_{\beta}\times\nabla\mathfrak{H}+\bar{b}\cdot\mathcal{V}_{\beta}\times\big(\frac{\partial}{\partial\theta}\times\nabla\mathfrak{H}\big)\end{aligned}
\]
with $\mathfrak{H}(x,y,e^{i\theta}):=x^{2}+y^{2}$. In this subsection
we look into the motion generated solely by $\mathcal{X}_{\phi}=\mathcal{V}_{\beta}\times\nabla\mathfrak{H}$,
which will provide an insight for modifying the vector field.

Each trajectory of $\mathcal{X}_{\phi}$ remains on a single level
set of $\mathfrak{H}$ since 
\[
d\mathfrak{H}\big(\mathcal{X}_{\phi}\big)=\nabla\mathfrak{H}\cdot\big(\mathcal{V}_{\beta}\times\nabla\mathfrak{H}\big)=0.
\]
We already see $\nabla\mathfrak{H}\perp\frac{\partial}{\partial\theta}$
from $d\mathfrak{H}(\frac{\partial}{\partial\theta})=0$. Also, $\nabla\mathfrak{H}$
is perpendicular to $\frac{\partial}{\partial\phi}:=x\frac{\partial}{\partial y}-y\frac{\partial}{\partial x}$
since $d\mathfrak{H}=xdx+ydy$ and then
\[
\nabla\mathfrak{H}\cdot\frac{\partial}{\partial\phi}=d\mathfrak{H}\big(x\frac{\partial}{\partial y}-y\frac{\partial}{\partial x}\big)=0.
\]
Now that $\mathcal{X}_{\phi}$ is perpendicular to $\nabla\mathfrak{H}$,
it lies in the space spanned by $\frac{\partial}{\partial\theta}$
and $\frac{\partial}{\partial\phi}$, and thence
\begin{equation}
\mathcal{X}_{\phi}=\bar{h}\frac{\partial}{\partial\theta}+\lambda\frac{\partial}{\partial\phi}\ \text{ on }\mathcal{U}_{*}.\label{eq:=00005BStructure=00005DRotationTerm-VecForm}
\end{equation}
Moreover, $\beta(\mathcal{X}_{\phi})=0$ yields $\bar{h}=-\lambda\beta(\frac{\partial}{\partial\phi})$,
and therefore, wherever $\lambda=0$, it would be $\mathcal{X}_{\phi}=\mathbf{0}$.
However, since $\mathcal{V}_{\beta}\cdot\frac{\partial}{\partial\theta}=\beta(\frac{\partial}{\partial\theta})=1$
and $\nabla\mathfrak{H}\cdot\frac{\partial}{\partial\theta}=\frac{\partial}{\partial\theta}\mathfrak{H}=0$,
we know that on $\mathcal{U}_{*}$ (here $\nabla\mathfrak{H}\neq\mathbf{0}$),
$\mathcal{V}_{\beta}$ and $\nabla\mathfrak{H}$ are not collinear
and then $\mathcal{X}_{\phi}\neq\mathbf{0}$ everywhere on $\mathcal{U_{*}}$,
which means
\[
\lambda\neq0\ \text{ everywhere on }\mathcal{U_{*}}.
\]
Note that this implies $\mathcal{X}_{\phi}=\lambda\bar{\partial}_{\phi}$
on $\mathcal{U}_{*}$ with the horizontal lift $\bar{\partial}_{\phi}$
of the vector field $\partial_{\phi}=x\frac{\partial}{\partial y}-y\frac{\partial}{\partial x}$
on $\mathcal{B}_{\delta}$. As a consequence, the trajectories of
$\mathcal{X}_{\phi}$ on $\mathcal{U}_{*}$ coincide with those of
$\bar{\partial}_{\phi}$. Moreover, since $\nabla\mathfrak{H}=\mathbf{0}$
on the fiber $\{\mathbf{0}\}\times S^{1}$,
\[
\mathcal{X}_{\phi}\bigg|_{\{\mathbf{0}\}\times S^{1}}=\mathbf{0}=\bar{\partial}_{\phi}\bigg|_{\{\mathbf{0}\}\times S^{1}}
\]
and hence for both $\mathcal{X}_{\phi}$ and $\bar{\partial}_{\phi}$,
each point on $\{\mathbf{0}\}\times S^{1}$ constitutes an entire
trajectory. In fact, all these results can be unified and enhanced
by showing that $\lambda$ extends smoothly to the central fiber $\mathcal{P}$
and $\lambda\neq0$ everywhere on the whole neighborhood $\mathcal{U}$.
This is given as the following lemma:
\begin{lem}
\label{lem:RotationComponent} $\mathcal{X}_{\phi}=\lambda\bar{\partial}_{\phi}$
with $\lambda$ being the smooth function on $\mathcal{U}$ for $\bar{\Omega}=\lambda\Omega_{\mathfrak{e}}$,
where $\Omega_{\mathfrak{e}}$ is the standard volume form on $\mathbb{R}^{3}$
given in (\ref{eq:StandardRiemannianForm}), and $\bar{\Omega}:=d\theta\wedge dx\wedge dy$. 
\end{lem}
\begin{proof}
Consider Remark \ref{rem:=00005BDuality-on-R^3=00005Dcross=000026wedge}
and take the differential form $\tau$ in (\ref{eq:=00005BDuality=00005Dwedge-cross})
to be $\tau=d\mathfrak{H}$. $\mathcal{V}_{\tau}$ is then the Riesz
representation of $d\mathfrak{H}$ with respect to the standard Riemannian
metric (\ref{eq:StandardRiemannianMetric}), i.e., $\mathcal{V}_{\tau}=\nabla\mathfrak{H}$.
As is noted in Remark \ref{rem:=00005BDuality-on-R^3=00005Dcross=000026wedge},
$\mathcal{X}_{\phi}=\mathcal{V}_{\beta}\times\nabla\mathfrak{H}$
is then the vector field determined by
\begin{equation}
\alpha\big(\mathcal{X}_{\phi}\big)\Omega=\alpha\wedge\beta\wedge d\mathfrak{H},\forall\alpha\in\Lambda(\mathcal{U}),\label{eq:=00005BStructure=00005DRotationTerm-DiffForm}
\end{equation}
with the standard Riemannian volume form $\Omega$ on $\mathbb{R}^{3}$
as given in (\ref{eq:StandardRiemannianForm}). 

For the volume form $\bar{\Omega}=d\theta\wedge dx\wedge dy$, with
$\beta$ particularly taken to be in the form of (\ref{eq:=00005Cbeta-Normalized}),
it holds
\[
\bar{\Omega}=d\theta\wedge dx\wedge dy=\beta\wedge dx\wedge dy.
\]
We will show that $\bar{\partial}_{\phi}$ is exactly the vector field
satisfying
\[
\alpha\big(\bar{\partial}_{\phi}\big)\bar{\Omega}=\alpha\wedge\beta\wedge d\mathfrak{H},\forall\alpha\in\Lambda(\mathcal{U}),
\]
and it will then follow from Lemma (\ref{lem:DualityByDifferentVolumes})
that $\lambda\bar{\partial}_{\phi}=\mathcal{X}_{\phi}$ with the function
$\lambda$ for $\bar{\Omega}=\lambda\Omega$, which concludes the
proof. Note that the equation $\alpha\big(\mathcal{X}_{d\mathfrak{H}}\big)\bar{\Omega}=\alpha\wedge\beta\wedge d\mathfrak{H}$
defines a vector field $\mathcal{X}_{d\mathfrak{H}}$ on $\mathcal{U}$,
and therefore this is to prove $\mathcal{X}_{d\mathfrak{H}}=\bar{\partial}_{\phi}$.
Since both $\mathcal{X}_{d\mathfrak{H}}$ and $\bar{\partial}_{\phi}$
are tangent to $\ker\beta$, it suffices to show
\[
\mathfrak{p}_{*}(\mathcal{X}_{d\mathfrak{H}})=\mathfrak{p}_{*}(\bar{\partial}_{\phi})=\partial_{\phi}.
\]

Let $\mathfrak{h}$ be the function on $\mathcal{B}_{\delta}$ with $\mathfrak{h}=x^{2}+y^{2}$.
Then $d\mathfrak{h}=xdx+ydy$ is a differential $1$-form on $\mathcal{B}_{\delta}$,
and $d\mathfrak{H}$ is its pullback by the bundle projection $\mathfrak{p}:\mathcal{B}_{\delta}\times S^{1}\rightarrow\mathcal{B}_{\delta}$,
i.e., $d\mathfrak{H}=\mathfrak{p}^{*}(d\mathfrak{h})$. For any differential
$1$-form $\varsigma=a\cdot dx+b\cdot dy$ on $\mathcal{B}_{\delta}$,
check that
\[
\varsigma\wedge d\mathfrak{h}=(ay-bx)dx\wedge dy=\varsigma\big(-\partial_{\phi}\big)dx\wedge dy,
\]
where $\partial_{\phi}$ is the vector field on $\mathcal{B}_{\delta}$
with $\partial_{\phi}=x\frac{\partial}{\partial y}-y\frac{\partial}{\partial x}$.
Note that the pullback $\varsigma^{*}=\mathfrak{p}^{*}(\varsigma)$
as the expression: $\varsigma^{*}=(a\circ\mathfrak{p})\cdot dx+(b\circ\mathfrak{p})\cdot dy$
with $dx$, $dy$ being differential forms on $\mathcal{B}_{\delta}\times S^{1}$.
Since $d\mathfrak{H}=xdx+ydy$, it is straightforward to check that
\[
\varsigma^{*}\wedge\beta\wedge d\mathfrak{H}=\varsigma^{*}\wedge d\theta\wedge d\mathfrak{H}=-d\theta\wedge\mathfrak{p}^{*}(\varsigma\wedge d\mathfrak{h})=\varsigma\big(\partial_{\phi}\big)\bar{\Omega}.
\]
As a result, $\varsigma^{*}\big(\mathcal{X}_{d\mathfrak{H}}\big)\bar{\Omega}=\varsigma\big(\partial_{\phi}\big)\bar{\Omega}$
and hence
\[
\varsigma\circ\mathfrak{p}_{*}(\mathcal{X}_{d\mathfrak{H}})=\varsigma^{*}\big(\mathcal{X}_{d\mathfrak{H}}\big)=\varsigma\big(\partial_{\phi}\big).
\]
The arbitrariness of $\varsigma$ on $\mathcal{B}_{\delta}$ implies
$\mathcal{X}_{d\mathfrak{H}}=\partial_{\phi}$.
\end{proof}
\begin{rem}
From the structure of $\mathcal{X}_{\phi}$ by (\ref{eq:=00005BStructure=00005DRotationTerm-VecForm}),
we obtain that $r^{2}\lambda=d\phi(\mathcal{X}_{\phi})$, combining which
with (\ref{eq:=00005BStructure=00005DRotationTerm-DiffForm}) gives
\[
d\phi\wedge\beta\wedge d\mathfrak{H}=d\phi(\mathcal{X}_{\phi})\Omega=r^{2}\lambda\Omega=r^{2}\bar{\Omega}.
\]
\end{rem}
According to Lemma \ref{lem:RotationComponent}, for understanding
the motion generated solely by $\mathcal{X}_{\phi}$, it suffices to
study the orbits of $\bar{\partial}_{\phi}$. Let $\bar{\varphi}_{\phi}$
be the flow of $\bar{\partial}_{\phi}$ and note that $\bar{\partial}_{\phi}=\mathbf{0}$
on $\{\mathbf{0}\}\times S^{1}$ implies the points on this fiber
to be fixed points of $\bar{\varphi}_{\phi}^{s}$ for all $s$. Since
$\bar{\partial}_{\phi}$ is the horizontal lift of $\partial_{\phi}$,
$\bar{\varphi}_{\phi}$ covers the flow $\varphi_{\phi}$ of $\partial_{\phi}$
on $\mathcal{B}_{\delta}$, which can be explicitly given by
\begin{equation}
\mathcal{B}_{\delta}\ni\mathbf{z}\xmapsto{\varphi_{\phi}^{s}}e^{s\cdot\mathbf{i}}\mathbf{z}\in\mathcal{B}_{\delta},\ \forall s\in\mathbb{R}.\label{eq:RotationOn=00005CB(ase)}
\end{equation}
Here $e^{s\cdot\mathbf{i}}$ is simply the matrix $\left[\begin{array}{cc}
\cos s & -\sin s\\
\sin s & \cos s
\end{array}\right]$. To be specific, 
\begin{equation}
\mathfrak{p}\circ\bar{\varphi}_{\phi}^{s}=\varphi_{\phi}^{s}\circ\mathfrak{p},\ \forall s\in\mathbb{R}.\label{eq:Rot-rot}
\end{equation}
In particular, $\varphi_{\phi}^{2\pi}=e^{2\pi\mathbf{i}}$ is the
identity map, and then $\mathfrak{p}\circ\bar{\varphi}_{\phi}^{2\pi}=\varphi_{\phi}^{2\pi}\circ\mathfrak{p}=\mathfrak{p}$.
This means that, starting from a point $\bar{p}=(\bar{\mathbf{z}},e^{i\bar{\theta}})$
at time $s=0$, the orbit $\bar{\varphi}_{\phi}^{s}(\bar{p})$ winds
around $\{\mathbf{0}\}\times S^{1}$ in $\mathcal{U}$ and comes back
to the fiber $\{\bar{\mathbf{z}}\}\times S^{1}$ at $s=2\pi$, since
its projection $\mathfrak{p}\circ\bar{\varphi}_{\phi}^{s}(\bar{p})=\varphi{}_{\phi}^{s}(\bar{\mathbf{z}})$
is a circle centered at $\mathbf{0}$ in $\mathcal{B}_{\delta}$, and,
\[
\mathfrak{p}\circ\bar{\varphi}_{\phi}^{2\pi}(\bar{p})=\mathfrak{p}(\bar{p})=\bar{\mathbf{z}}.
\]
Based on this picture, we may characterize the motion by the variation
of $\int d\theta$ over the path $\eta(s):=\bar{\varphi}_{\phi}^{s}(\bar{p})$,
$s\in[0,2\pi]$. 
\begin{thm}
\label{thm:=00005Brotation-lift=00005DStructureOf=00005Cbeta}If $\beta\wedge d\beta=\rho\cdot d\theta\wedge dx\wedge dy$
with a negative function $\rho$ on $\mathcal{B}_{\delta}\times S^{1}$,
then for any $\bar{p}\in\mathcal{B}_{\delta}^{*}\times S^{1}$, it
holds
\[
\int_{\eta}d\theta>0.
\]
Here $\eta$ is the path $\eta(s):=\bar{\varphi}_{\phi}^{s}(\bar{p})$,
$s\in[0,2\pi]$.
\end{thm}
\begin{rem}
As is mentioned in Subsection \ref{subsec:Concept-Definition-Notation},
the complete nonholonomicity is equivalent to (\ref{eq:EuiqvalentCondition1-=00005BCompletelynonHolonomic=00005D})
with $\underset{\mathcal{U}}{\min}|\rho|>0$. Therefore, the condition
in Theorem \ref{thm:=00005Brotation-lift=00005DStructureOf=00005Cbeta}
implies the constraint to be completely non-holonomic.
\end{rem}
To get an idea for proving Theorem \ref{thm:=00005Brotation-lift=00005DStructureOf=00005Cbeta},
we think of its implication on the orbit $\bar{\varphi}_{\phi}^{s}(\bar{p})$:
starting from $\bar{p}$, when the orbit turns back to the same fiber
at $s=2\pi$, the new spot $\bar{\varphi}_{\phi}^{2\pi}(\bar{p})$
should be ``above'' $\bar{p}$ (with a larger $\theta$). To better
depict this scenario, we think of the horizontal lift $p_{t}$ with
$p_{t=1}=\bar{p}$ of the line $\mathbf{z}_{t}=t\bar{\mathbf{z}}$
in $\mathcal{B}_{\delta}$ ($\bar{\mathbf{z}}$ is the projection
of $\bar{p}$ as above) and its image $\bar{\varphi}_{\phi}^{2\pi}(p_{t})$.
Theorem \ref{thm:=00005Brotation-lift=00005DStructureOf=00005Cbeta}
exactly means that the curve $\bar{\varphi}_{\phi}^{2\pi}(p_{t})$
is ``above'' $p_{t}$ for $0<t\leq1$. Note that both of the curves
are in the ``sheet'' $\mathbf{z}_{[0,1]}\times S^{1}$. Moreover,
curve $p_{t}$ is among the other horizontal lifts of the line $\mathbf{z}_{t}=t\bar{\mathbf{z}}$,
all of which form a foliation of $\mathbf{z}_{[0,1]}\times S^{1}$.
To be concrete, we introduce the following definition, which will
be used in the rest of this paper:
\begin{defn}
\label{def:ParallelParametrization}For any $\bar{\mathbf{z}}\in\mathcal{B}_{\delta}$,
$\psi_{\bar{\mathbf{z}}}$ is the map defined as below 
\begin{equation}
S^{1}\times[0,t]\ni(w,t)\xmapsto{\psi_{\bar{\mathbf{z}}}}p_{t}^{w}\in\mathbf{z}_{[0,1]}\times S^{1},\label{eq:ParallelParametersOnMainSheet}
\end{equation}
where $t\mapsto p_{t}^{w}$ is the horizontal lift of $t\mapsto t\bar{\mathbf{z}}$
with $p_{t=0}^{w}=(\mathbf{0},w)$. 
\end{defn}
The core idea for proving Theorem \ref{thm:=00005Brotation-lift=00005DStructureOf=00005Cbeta}
is then to show $\bar{\varphi}_{\phi}^{2\pi}(p_{t})$ to be transverse
to the ``coordinate lines'' $p_{t}^{w}$ wherever $\bar{\varphi}_{\phi}^{2\pi}(p_{t})=p_{t}^{w}$
for $t>0$. Observe that $\dot{p}_{t}^{w}$ is collinear with $\bar{\partial}_{r}$
for $0<t\leq1$, where $\bar{\partial}_{r}$ is the horizontal lift
of $\partial_{r}=x\frac{\partial}{\partial x}+y\frac{\partial}{\partial y}$
(a vector field on $\mathcal{B}_{\delta}$). The strategy is then
to compare the tangent $\frac{d}{dt}\bar{\varphi}_{\phi}^{2\pi}(p_{t})=\bar{\varphi}_{\phi,*}^{2\pi}(\dot{p}_{t})$
to $\bar{\partial}_{r}$ (or $\ker\beta$), that is, to show that
at each $\bar{\varphi}_{\phi}^{2\pi}(p_{t})$,
\[
\bar{\varphi}_{\phi,*}^{2\pi}(\dot{p}_{t})=a\frac{\partial}{\partial\theta}+c\bar{\partial}_{r}+0\cdot\bar{\partial}_{\phi}
\]
for some $a>0$ and $c\in\mathbb{R}^{2}$. The coefficient for $\bar{\partial}_{\phi}$
in such a decomposition is always $0$ since $\frac{\partial}{\partial\theta}$
and $\bar{\partial}_{r}$ already span the tangent planes of $\mathbf{z}_{[0,1]}\times S^{1}$. 

To this end, take the (immersed) surface $\mathcal{S}=\mathrm{Im}\Phi$
given as the image of the map $\Phi$ defined below
\begin{equation}
[0,2\pi]\times[0,1]\ni(s,t)\xmapsto{\ \Phi\ }\bar{\varphi}_{\phi}^{s}(p_{t})\in\mathcal{U},\label{eq:DefiningSurface=00005CS}
\end{equation}
and consider the variation of the following vector field on $\mathcal{S}$
\[
\hat{\mathcal{\partial}}_{\mathcal{S}}:=\Phi_{*}(\frac{\partial}{\partial t})=\bar{\varphi}_{\phi,*}^{s}(\dot{p}_{t}).
\]
Since $\mathfrak{p}(p_{t})=\mathbf{z}_{t}=t\bar{\mathbf{z}}$, it
holds $\mathfrak{p}_{*}(\dot{p}_{t})=\frac{1}{t}\partial_{r}\bigg|_{\mathbf{z}_{t}}$
and then $\dot{p}_{t}=\frac{1}{t}\bar{\partial}_{r}\bigg|_{p_{t}}$
for $0<t\leq1$. Applying the relation (\ref{eq:Rot-rot}) yields
\[
\mathfrak{p}_{*}\circ\bar{\varphi}_{\phi,*}^{s}(\dot{p}_{t})=\varphi_{\phi,*}^{s}\circ\mathfrak{p}_{*}(\dot{p}_{t})=\frac{d}{dt}(e^{s\cdot\mathbf{i}}\mathbf{z}_{t})=\frac{1}{t}\partial_{r}\bigg|_{\varphi_{\phi}^{s}(\mathbf{z}_{t})},
\]
from which we deduce
\begin{equation}
\hat{\mathcal{\partial}}_{\mathcal{S}}\bigg|_{\bar{\varphi}_{\phi}^{s}(p_{t})}=\bar{\varphi}_{\phi,*}^{s}(\dot{p}_{t})=\kappa\frac{\partial}{\partial\theta}+\frac{|\bar{\mathbf{z}}|}{r}\bar{\partial}_{r}+0\cdot\bar{\partial}_{\phi}.\label{eq:CharacteristicRotationField}
\end{equation}
Here $r:=\sqrt{x^{2}+y^{2}}$ and hence $\frac{|\bar{\mathbf{z}}|}{r}=\frac{1}{t}$.
It turns out that, with the condition assumed in Theorem \ref{thm:=00005Brotation-lift=00005DStructureOf=00005Cbeta}, the function $\kappa$ is positive everywhere on $\mathcal{S}$ except
on the ``initial'' curve $p_{t}$, where $\kappa(p_{t})=0$, which
is formalized and proved as Theorem \ref{thm:StructureEq.-=00005Cbeta}
below.
\begin{rem}
\label{rem:EquivalentConditions}The condition $\beta\wedge d\beta=\rho\cdot d\theta\wedge dx\wedge dy$
with $\rho<0$ just means 
\begin{equation}
\beta\big([\bar{\partial}_{r},\bar{\partial}_{\phi}]\big)=-d\beta(\bar{\partial}_{r},\bar{\partial}_{\phi})>0\ \text{ on }\mathcal{U}_{*}.\label{eq:EquivalentConditions}
\end{equation}
To see $d\beta(\bar{\partial}_{r},\bar{\partial}_{\phi})<0$, it suffices
to note that
\[
\beta(\frac{\partial}{\partial\theta})\cdot d\beta(\bar{\partial}_{r},\bar{\partial}_{\phi})=\beta\wedge d\beta(\frac{\partial}{\partial\theta},\bar{\partial}_{r},\bar{\partial}_{\phi})=\rho\cdot d\theta\wedge dx\wedge dy(\frac{\partial}{\partial\theta},\bar{\partial}_{r},\bar{\partial}_{\phi}),
\]
in which $\beta(\frac{\partial}{\partial\theta})=1$ as is assumed,
and on $\mathcal{U}_{*}$, 
\[
d\theta\wedge dx\wedge dy\big(\frac{\partial}{\partial\theta},\bar{\partial}_{r},\bar{\partial}_{\phi}\big)=r^{2}\cdot d\theta\wedge dx\wedge dy\big(\frac{\partial}{\partial\theta},\frac{\partial}{\partial x},\frac{\partial}{\partial y}\big)>0.
\]
Then $\beta\big([\bar{\partial}_{r},\bar{\partial}_{\phi}]\big)>0$
simply follows from
\[
\begin{aligned}d\beta(\bar{\partial}_{r},\bar{\partial}_{\phi})= & \bar{\partial}_{r}\bigg(\beta(\bar{\partial}_{\phi})\bigg)-\bar{\partial}_{\phi}\bigg(\beta(\bar{\partial}_{r})\bigg)-\beta\big([\bar{\partial}_{r},\bar{\partial}_{\phi}]\big)\\
= & 0-0-\beta\big([\bar{\partial}_{r},\bar{\partial}_{\phi}]\big).
\end{aligned}
\]
\end{rem}
\begin{thm}
\label{thm:StructureEq.-=00005Cbeta}The equation below holds for
the function $\kappa$ on $\mathcal{S}_{*}=\mathcal{S}\cap\mathcal{U}_{*}$
\begin{equation}
\bar{\partial}_{\phi}\kappa+\kappa\cdot d\beta(\frac{\partial}{\partial\theta},\bar{\partial}_{\phi})=\frac{|\bar{\mathbf{z}}|}{r}\beta\big([\bar{\partial}_{r},\bar{\partial}_{\phi}]\big)\label{eq:StructuralEq-for-=00005CS=000026=00005Cbeta}
\end{equation}
As a result, if $\beta\big([\bar{\partial}_{r},\bar{\partial}_{\phi}]\big)>0$
holds on $\mathcal{U}_{*}$, $\kappa>0$ holds everywhere on $\mathcal{S}_{*}$
except on the curve $p_{t}$, where $\kappa(p_{t})=0$. 
\end{thm}
\begin{proof}
Since $\hat{\partial}_{\mathcal{S}}=\Phi_{*}\big(\frac{\partial}{\partial t}\big)$
and $\bar{\partial}_{\phi}=\Phi_{*}\big(\frac{\partial}{\partial s}\big)$
on $\mathcal{S}$, their Lie bracket vanishes:
\[
[\hat{\partial}_{\mathcal{S}},\bar{\partial}_{\phi}]=\big[\Phi_{*}\big(\frac{\partial}{\partial t}\big),\Phi_{*}\big(\frac{\partial}{\partial s}\big)\big]=\Phi_{*}\big[\frac{\partial}{\partial t},\frac{\partial}{\partial s}\big]=\mathbf{0}.
\]
As a result, we have
\[
0=\beta\big([\hat{\partial}_{\mathcal{S}},\bar{\partial}_{\phi}]\big)=\beta\big([\kappa\frac{\partial}{\partial\theta},\bar{\partial}_{\phi}]\big)+\beta\big([\frac{|\bar{\mathbf{z}}|}{r}\bar{\partial}_{r},\bar{\partial}_{\phi}]\big).
\]
Equation \eqref{eq:StructuralEq-for-=00005CS=000026=00005Cbeta} then
follows from $[\frac{|\bar{\mathbf{z}}|}{r}\bar{\partial}_{r},\bar{\partial}_{\phi}]=\frac{|\bar{\mathbf{z}}|}{r}[\bar{\partial}_{r},\bar{\partial}_{\phi}]$
(since $\bar{\partial}_{\phi}\frac{|\bar{\mathbf{z}}|}{r}=0$) and
\[
\begin{aligned}d\beta\big(\kappa\frac{\partial}{\partial\theta},\bar{\partial}_{\phi}\big)= & \kappa\frac{\partial}{\partial\theta}\beta(\bar{\partial}_{\phi})-\bar{\partial}_{\phi}\beta(\kappa\frac{\partial}{\partial\theta})-\beta\big([\kappa\frac{\partial}{\partial\theta},\bar{\partial}_{\phi}]\big)\\
= & 0-\bar{\partial}_{\phi}\kappa\cdot1+\beta\big([\frac{|\bar{\mathbf{z}}|}{r}\bar{\partial}_{r},\bar{\partial}_{\phi}]\big).
\end{aligned}
\]
To see this equation implies $\kappa>0$ at each $\bar{\varphi}_{\phi}^{s}(p_{t})$
for all $s,t>0$, it suffices to note that, by pulling $\kappa$ back
to the $(s,t)$-square $[0,2\pi]\times[0,1]$ through $\Phi$, \eqref{eq:StructuralEq-for-=00005CS=000026=00005Cbeta}
becomes the following equation about $\kappa_{s,t}:=\kappa\circ\bar{\varphi}_{\phi}^{s}(p_{t})$
\[
\frac{\partial}{\partial s}\kappa+\kappa\cdot d\beta(\frac{\partial}{\partial\theta},\bar{\partial}_{\phi})=\frac{|\bar{\mathbf{z}}|}{r}\beta\big([\bar{\partial}_{r},\bar{\partial}_{\phi}]\big)
\]
with the initial condition $\kappa_{s,t}\bigg|_{s=0}=0$ (since $\hat{\partial}_{\mathcal{S}}\bigg|_{p_{t}}=\dot{p}_{t}=\frac{|\bar{\mathbf{z}}|}{r}\bar{\partial}_{r}$),
and thence for $t>0$,
\begin{equation}
\kappa_{s,t}=e^{-\int_{0}^{s}d\beta(\frac{\partial}{\partial\theta},\bar{\partial}_{\phi})}\cdot\int_{0}^{s}e^{\int_{0}^{s'}d\beta(\frac{\partial}{\partial\theta},\bar{\partial}_{\phi})}\cdot\frac{|\bar{\mathbf{z}}|}{r}\beta\big([\bar{\partial}_{r},\bar{\partial}_{\phi}]\big)ds',\label{eq:StructuralFormula}
\end{equation}
which is positive whenever $s>0$.
\end{proof}
It follows directly from Theorem \ref{thm:StructureEq.-=00005Cbeta}
that $\frac{d}{dt}\bar{\varphi}_{\phi}^{2\pi}(p_{t})=\kappa\frac{\partial}{\partial\theta}+\dot{p}_{t}^{w}$
with $\kappa>0$ whenever $t>0$, which suggests that the curve $t\mapsto\bar{\varphi}_{\phi}^{2\pi}(p_{t})$
keeps ``going up'' across every $t\mapsto p_{t}^{w}$ (horizontal
lifts of $\mathbf{z}_{t}=t\bar{\mathbf{z}}$) it meets. The math behind
this picture is better viewed in the space $S^{1}\times[0,1]$ through
the diffeomorphism $\psi_{\bar{\mathbf{z}}}$ defined in (\ref{eq:ParallelParametersOnMainSheet}).
Note that each path $t\mapsto p_{t}^{w}$ in $\mathbf{z}_{[0,1]}\times S^{1}$
becomes $t\mapsto(w,t)$ in $S^{1}\times[0,1]$. Also, with $p_{0}=(\mathbf{0},e^{i\cdot\theta_{0}})$,
the path $t\mapsto\bar{\varphi}_{\phi}^{2\pi}(p_{t})$ becomes $\psi_{\bar{\mathbf{z}}}^{-1}\circ\bar{\varphi}_{\phi}^{2\pi}(p_{t})=\big(e^{i\vartheta_{t}},t\big)$
with
\begin{equation}
\vartheta_{t}=\theta_{0}+\int_{0}^{t}\kappa\circ\bar{\varphi}_{\phi}^{2\pi}(p_{t'})dt'.\label{eq:=00005Cvartheta_t}
\end{equation}
Theorem \ref{thm:StructureEq.-=00005Cbeta} simply implies
\begin{equation}
\dot{\vartheta}_{t}=\bar{\kappa}_{2\pi,t}=\kappa\circ\bar{\varphi}_{\phi}^{2\pi}(p_{t})>0\text{ for }t\in(0,1].\label{eq:=00005Cvartheta-increasing}
\end{equation}

To bridge these results to Theorem \ref{thm:=00005Brotation-lift=00005DStructureOf=00005Cbeta},
we pay attention to three paths: $\eta(s):=\bar{\varphi}_{\phi}^{s}(\bar{p})$
with $s\in[0,2\pi]$, $\gamma(t):=\psi_{\bar{\mathbf{z}}}\big(e^{i\vartheta_{t}},1\big)$
with $t\in[0,1]$, and, path $\zeta$ defined by joining $p_{1-t}$
with $\bar{\varphi}_{\phi}^{2\pi}(p_{t})$:
\[
\zeta(t):=\begin{cases}
p_{1-2t} & t\in[0,\frac{1}{2}]\\
\bar{\varphi}_{\phi}^{2\pi}(p_{2t-1}) & t\in[\frac{1}{2},1]
\end{cases}.
\]
Note that the boundary of $\mathcal{S}$ is exactly $\partial\mathcal{S}=\eta\cup\zeta^{-1}$
with $\zeta^{-1}(t)=\zeta(1-t)$. By the Stokes Theorem, it holds
\begin{equation}
0=\int_{\mathcal{S}}d\big(d\theta\big)=\int_{\partial\mathcal{S}}d\theta=\int_{\eta}d\theta-\int_{\zeta}d\theta.\label{eq:BoundaryIntegral-on-=00005CS}
\end{equation}
On the other hand, $\gamma$ is homotopic to $\zeta$ the space $\mathbf{z}_{[0,1]}\times S^{1}$
(and hence in $\mathcal{U}$ as well). This is because in the space
$S^{1}\times[0,1]$, the path $\psi_{\bar{\mathbf{z}}}^{-1}\circ\gamma(t)=\big(e^{i\theta_{t}},1\big)$
is homotopic to
\[
\psi_{\bar{\mathbf{z}}}^{-1}\circ\zeta(t)=\begin{cases}
\big(e^{i\theta_{0}},1-2t\big) & t\in[0,\frac{1}{2}]\\
\big(e^{i\theta_{2t-1}},1\big) & t\in[\frac{1}{2},1]
\end{cases}.
\]
Since $d\theta$ is a closed form, it implies
\begin{equation}
\int_{\eta}d\theta=\int_{\zeta}d\theta=\int_{\gamma}d\theta.\label{eq:3-Path-Integrals}
\end{equation}

Recall that $\psi_{\bar{\mathbf{z}}}\big(w,t\big)=p_{t}^{w}$ is a
horizontal lift of $\mathbf{z}_{t}=t\bar{\mathbf{z}}$, and then
\[
\mathfrak{p}\big(\gamma(t)\big)=\mathfrak{p}\circ\psi_{\bar{\mathbf{z}}}\big(e^{i\vartheta_{t}},1\big)\equiv\bar{\mathbf{z}}.
\]
This means that $\gamma(t)$ is a path on the fiber $\{\bar{\mathbf{z}}\}\times S^{1}$,
and hence $\dot{\gamma}(t)=\nu_{t}\frac{\partial}{\partial\theta}$.
As a result, $\beta\big(\dot{\gamma}(t)\big)=\nu_{t}=d\theta\big(\dot{\gamma}(t)\big)$,
combining which with (\ref{eq:3-Path-Integrals}) yields:
\begin{equation}
\int_{\eta}d\theta=\int_{\gamma}d\theta=\int_{0}^{1}\nu_{t}dt=\int_{\gamma}\beta.\label{eq:=00005Cint(d=00005Ctheta)=00003D=00005Cint(=00005Cbeta)}
\end{equation}
Theorem \ref{thm:=00005Brotation-lift=00005DStructureOf=00005Cbeta}
is then confirmed by the following result about the pull-back $\psi_{\bar{\mathbf{z}}}^{*}(\beta)$.
Here we use variables $(e^{i\vartheta},t)$ for points in the space
$S^{1}\times[0,1]$, so that the difference between the $1$-form
$d\vartheta$ on $S^{1}\times[0,1]$ and $d\theta$ on $\mathbf{z}_{[0,1]}\times S^{1}$
are better suggested. 
\begin{lem}
\label{lem:=00005Cbeta-in-ParallelParameterization}For the map $\psi_{\bar{\mathbf{z}}}$
defined in (\ref{eq:ParallelParametersOnMainSheet}), there exists
some positive function $\mu$ on $S^{1}\times[0,1]$ such that $\psi_{\bar{\mathbf{z}},*}\big(\frac{\partial}{\partial\vartheta}\big)=\mu\frac{\partial}{\partial\theta}$
and hence $\psi_{\bar{\mathbf{z}}}^{*}(\beta)=\mu d\vartheta$.
\end{lem}
\begin{proof}
The diffeomorphism $\psi_{\bar{\mathbf{z}}}$ maps each fiber $S^{1}\times\{t\}$
to $\{\mathbf{z}_{t}\}\times S^{1}$ and each horizontal line $\{e^{i\vartheta}\}\times[0,1]$
to the horizontal lift $p_{t}^{w}$ of $\mathbf{z}_{t}=t\bar{\mathbf{z}}$
with $w=e^{i\vartheta}$. It follows directly from its definition that
\[
\psi_{\bar{\mathbf{z}},*}(\frac{\partial}{\partial t}\big|_{(w,t)})=\dot{p}_{t}^{w}\in\ker\beta,
\]
or equivalently, $\frac{\partial}{\partial t}\in\ker\psi_{\bar{\mathbf{z}}}^{*}(\beta)$.
Moreover, we see that $\psi_{\bar{\mathbf{z}}}$ preserves the orientation
of the $S^{1}$-fibers since $\psi_{\bar{\mathbf{z}}}(w,0)=p_{0}^{w}=(\mathbf{0},w)$,
and therefore, there is a function $\mu>0$ on $S^{1}\times[0,1]$
such that $\psi_{\bar{\mathbf{z}},*}\big(\frac{\partial}{\partial\vartheta}\big)=\mu\frac{\partial}{\partial\theta}$.
As a result, $\psi_{\bar{\mathbf{z}}}^{*}(\beta)=\mu d\vartheta$.
\end{proof}
\begin{proof}[Proof of Theorem \ref{thm:=00005Brotation-lift=00005DStructureOf=00005Cbeta}]
Note that $\psi_{\bar{\mathbf{z}}}^{-1}\circ\bar{\varphi}_{\phi}^{2\pi}(p_{t})=\big(e^{i\vartheta_{t}},t\big)$
with $\vartheta_{t}$ given in (\ref{eq:=00005Cvartheta_t}). By (\ref{eq:=00005Cint(d=00005Ctheta)=00003D=00005Cint(=00005Cbeta)})
and (\ref{eq:=00005Cvartheta-increasing}) we have
\[
\int_{\eta}d\theta=\int_{\psi_{\bar{\mathbf{z}}}^{-1}(\gamma)}\psi_{\bar{\mathbf{z}}}^{*}(\beta)=\int_{0}^{1}\mu_{t}\dot{\vartheta}_{t}dt>0.
\]
\end{proof}

\subsection{Rotation $\bar{\varphi}_{\phi,*}^{s}$ on $\mathcal{P}$}

In the previous subsection we have studied $\bar{\varphi}_{\phi,*}^{s}(\dot{p}_{t})$
for $t>0$ as well as the structure equation (\ref{eq:StructuralEq-for-=00005CS=000026=00005Cbeta})
on $\mathcal{S}_{*}$. For completeness, we now consider the
tangent map $\bar{\varphi}_{\phi,*}^{s}$ on the central fiber $\mathcal{P}=\{\mathbf{0}\}\times S^{1}$. 

Since each point $p\in\mathcal{P}$ is a fixed point of the flow $\bar{\varphi}_{\phi}^{s}$,
the tangent maps $\{\bar{\varphi}_{\phi,*}^{s}|s\in\mathbb{R}\}$
at $p$ constitute an $\mathbb{R}$-action on the tangent space $\mathrm{T}_{p}\mathcal{U}$.
Moreover, for each $s\in\mathbb{R}$, it holds $\bar{\varphi}_{\phi,*}^{s}(\frac{\partial}{\partial\theta})=\frac{\partial}{\partial\theta}$
since $\bar{\varphi}_{\phi}^{s}(\mathbf{0},e^{i\theta})=(\mathbf{0},e^{i\theta})$
for all $\theta\in\mathbb{R}$. Also, since $\bar{\varphi}_{\phi}^{s}$
covers the flow $\varphi_{\phi}^{s}$ of $\frac{\partial}{\partial\phi}=x\frac{\partial}{\partial y}-y\frac{\partial}{\partial x}$
on $\mathcal{B}_{\delta}$ by (\ref{eq:Rot-rot}), $\bar{\varphi}_{\phi,*}^{s}$
covers $\big(\varphi_{\phi}^{s}\big)_{*}$ with
\begin{equation}
\mathfrak{p}_{*}\circ\bar{\varphi}_{\phi,*}^{s}=\big(\varphi_{\phi}^{s}\big)_{*}\circ\mathfrak{p}_{*}.\label{eq:=00005BRot=00005D-=00005Brot=00005D}
\end{equation}
In particular, at $p\in\mathcal{P}$ the right-hand side is
\[
c_{0}\frac{\partial}{\partial\theta}+c_{1}\frac{\partial}{\partial x}+c_{2}\frac{\partial}{\partial y}\xmapsto{\mathfrak{p}_{*}}c_{1}\frac{\partial}{\partial x}+c_{2}\frac{\partial}{\partial y}\xmapsto{\big(\varphi_{\phi}^{s}\big)_{*}}\big(\frac{\partial}{\partial x},\frac{\partial}{\partial y}\big)\left[\begin{array}{cc}
\cos s & -\sin s\\
\sin s & \cos s
\end{array}\right]\left[\begin{array}{c}
c_{1}\\
c_{2}
\end{array}\right].
\]
Taking $(c_{0},c_{1},c_{2})$ to be $(0,1,0)$ and $(0,0,1)$, respectively,
we conclude that
\[
\bar{\varphi}_{\phi,*}^{s}\big(\frac{\partial}{\partial x}\big)=a_{s}\frac{\partial}{\partial\theta}+\cos s\frac{\partial}{\partial x}+\sin s\frac{\partial}{\partial y}
\]
and
\[
\bar{\varphi}_{\phi,*}^{s}\big(\frac{\partial}{\partial y}\big)=b_{s}\frac{\partial}{\partial\theta}-\sin s\frac{\partial}{\partial x}+\cos s\frac{\partial}{\partial y}
\]
for some functions $a_{s}$ and $b_{s}$ in $s$. This means that,
with respect to the basis $\big\{\frac{\partial}{\partial\theta},\frac{\partial}{\partial x},\frac{\partial}{\partial y}\big\}$
of the linear space $\mathrm{T}_{p}\mathcal{U}$, the matrix representation
of $\bar{\varphi}_{\phi,*}^{s}$ is
\begin{equation}
[\bar{\varphi}_{\phi,*}^{s}]=\left[\begin{array}{ccc}
1 & a_{s} & b_{s}\\
0 & \cos s & -\sin s\\
0 & \sin s & \cos s
\end{array}\right].\label{eq:LandScapeMatrix}
\end{equation}
Of course, $[\bar{\varphi}_{\phi,*}^{s}]$ is the identity matrix,
and hence $a_{0}=b_{0}=0$. To further understand $\bar{\varphi}_{\phi,*}^{s}$
on $\mathcal{P}$, we calculate the form of the functions $a_{s}$
and $b_{s}$. 

Now that $\bar{\varphi}_{\phi,*}^{s+s'}=\bar{\varphi}_{\phi,*}^{s}\circ\bar{\varphi}_{\phi,*}^{s'}$,
it holds for the matrix representation
\[
[\bar{\varphi}_{\phi,*}^{s+s'}]=[\bar{\varphi}_{\phi,*}^{s}]\cdot[\bar{\varphi}_{\phi,*}^{s'}].
\]
In particular, we have $[\bar{\varphi}_{\phi,*}^{2\pi}]=[\bar{\varphi}_{\phi,*}^{\pi}]\cdot[\bar{\varphi}_{\phi,*}^{\pi}]$,
and direct computation shows (regardless of the specific values of
$a_{\pi}$ and $b_{\pi}$)
\[
\left[\begin{array}{ccc}
1 & a_{2\pi} & b_{2\pi}\\
0 & 1 & 0\\
0 & 0 & 1
\end{array}\right]=\left[\begin{array}{ccc}
1 & a_{\pi} & b_{\pi}\\
0 & -1 & 0\\
0 & 0 & -1
\end{array}\right]\cdot\left[\begin{array}{ccc}
1 & a_{\pi} & b_{\pi}\\
0 & -1 & 0\\
0 & 0 & -1
\end{array}\right]=\left[\begin{array}{ccc}
1 & 0 & 0\\
0 & 1 & 0\\
0 & 0 & 1
\end{array}\right].
\]
That is, at any point $p\in\mathcal{P}$, $s\mapsto\bar{\varphi}_{\phi,*}^{s}\big|_{p}$
is periodic with $\bar{\varphi}_{\phi,*}^{2\pi}\big|_{p}=\bar{\varphi}_{\phi,*}^{0}\big|_{p}=\mathbf{Id}$,\footnote{So, $s\mapsto\bar{\varphi}_{\phi,*}^{s}$ can be seen as an $S^{1}$
action on $\mathrm{T}_{p}\mathcal{U}$.} and therefore the relation $a_{2\pi}=b_{2\pi}=0$ always holds. To
see the form of $a_{s}$,$b_{s}$ for the other $s\in(0,2\pi)$, we
resort to the differential equation
\[
\frac{d}{ds}[\bar{\varphi}_{\phi,*}^{s}]=\frac{d}{dt}\bigg|_{t=0}[\bar{\varphi}_{\phi,*}^{s+t}]=[\bar{\varphi}_{\phi,*}^{s}]\cdot\frac{d}{dt}\bigg|_{t=0}[\bar{\varphi}_{\phi,*}^{t}],
\]
that is,
\[
\left[\begin{array}{ccc}
1 & \dot{a}_{s} & \dot{b}_{s}\\
0 & -\sin s & -\cos s\\
0 & \cos s & -\sin s
\end{array}\right]=\left[\begin{array}{ccc}
1 & a_{s} & b_{s}\\
0 & \cos s & -\sin s\\
0 & \sin s & \cos s
\end{array}\right]\cdot\left[\begin{array}{ccc}
1 & \dot{a}_{0} & \dot{b}_{0}\\
0 & 0 & -1\\
0 & 1 & 0
\end{array}\right].
\]
As a result, we obtain the following ODE for $a_{s}$,$b_{s}$:
\[
\left[\begin{array}{c}
\dot{a}_{s}\\
\dot{b}_{s}
\end{array}\right]=\left[\begin{array}{cc}
0 & 1\\
-1 & 0
\end{array}\right]\left[\begin{array}{c}
a_{s}\\
b_{s}
\end{array}\right]+\left[\begin{array}{c}
\dot{a}_{0}\\
\dot{b}_{0}
\end{array}\right],
\]
the solution of which is
\begin{equation}
\left[\begin{array}{c}
a_{s}\\
b_{s}
\end{array}\right]=\int_{0}^{s}e^{(s-t)\left[\begin{array}{cc}
0 & 1\\
-1 & 0
\end{array}\right]}\cdot\left[\begin{array}{c}
\dot{a}_{0}\\
\dot{b}_{0}
\end{array}\right]dt=\left[\begin{array}{cc}
\sin s & 1-\cos s\\
\cos s-1 & \sin s
\end{array}\right]\cdot\left[\begin{array}{c}
\dot{a}_{0}\\
\dot{b}_{0}
\end{array}\right].\label{eq:StructuralRot}
\end{equation}
Once the values of $\dot{a}_{0}$ and $\dot{b}_{0}$ are determined,
the action $\bar{\varphi}_{\phi,*}^{s}$ is also determined.

To further characterize $\bar{\varphi}_{\phi,*}^{s}$, consider the quantity $\kappa_{s,t}$ defined by Eq. (\ref{eq:CharacteristicRotationField}),
that is,
\begin{equation}
\bar{\varphi}_{\phi,*}^{s}(\dot{p}_{t})=\kappa_{s,t}\frac{\partial}{\partial\theta}+\frac{|\bar{\mathbf{z}}|}{r}\bar{\partial}_{r}.\label{eq:HorizontalDecomposition}
\end{equation}
Although the vector field $\hat{\partial}_{\mathcal{S}}:=\bar{\varphi}_{\phi,*}^{s}(\dot{p}_{t})$
is not well-defined at the point $p_{0}$ on $\mathcal{S}$, $\kappa_{s,t}$
is well defined as a continuous function in $(s,t)\in\mathbb{R}\times[0,1]$
through the above equation. On the other hand, it has been shown in
the proof of Theorem \ref{thm:StructureEq.-=00005Cbeta} that $\kappa_{s,t}$
can be expressed by (\ref{eq:StructuralFormula}) for $t>0$.
Denote by $\bar{\partial}_{x}$ and $\bar{\partial}_{y}$ the horizontal
lifts (subject to $\ker\beta$) on $\mathcal{U}=\mathcal{B}_{\delta}\times S^{1}$
of the vector fields $\frac{\partial}{\partial x}$ and $\frac{\partial}{\partial y}$
on $\mathcal{B}_{\delta}$, respectively. We have
\[
\bar{\partial}_{\phi}=x\bar{\partial}_{y}-y\bar{\partial}_{x}\ \text{ and }\ \bar{\partial}_{r}=x\bar{\partial}_{x}+y\bar{\partial}_{y},
\]
and then direct computation shows
\[
[\bar{\partial}_{r},\bar{\partial}_{\phi}]=(x^{2}+y^{2})[\bar{\partial}_{x},\bar{\partial}_{y}]=r^{2}[\bar{\partial}_{x},\bar{\partial}_{y}].
\]
Combined with the relation $r=t|\bar{\mathbf{z}}|$, (\ref{eq:StructuralFormula})
becomes (at least for $t>0$)
\[
\kappa_{s,t}=e^{-\int_{0}^{s}d\beta(\frac{\partial}{\partial\theta},\bar{\partial}_{\phi})}\cdot t|\bar{\mathbf{z}}|^{2}\int_{0}^{s}e^{\int_{0}^{s'}d\beta(\frac{\partial}{\partial\theta},\bar{\partial}_{\phi})}\beta\big([\bar{\partial}_{x},\bar{\partial}_{y}]\big)ds'.
\]
Taking the limit as $t\rightarrow0$ yields
\[
\kappa_{s,0}=\lim_{t\rightarrow0}\kappa_{s,t}=0,
\]
and hence $\bar{\varphi}_{\phi,*}^{s}(\dot{p}_{0})\in\ker\beta\big|_{p_{0}}$
for any $p_{0}\in\mathcal{P}$. Furthermore, from (\ref{eq:Rot-rot})
it holds $\mathfrak{p}_{*}\circ\bar{\varphi}_{\phi,*}^{s}(\dot{p}_{0})=\varphi_{\phi,*}^{s}(\dot{\mathbf{z}}_{0})$,
and hence $\bar{\varphi}_{\phi,*}^{s}(\dot{p}_{0})\in\ker\beta$ is
the horizontal lift of $\varphi_{\phi,*}^{s}(\dot{\mathbf{z}}_{0})$.

We sum up all these results with the theorem and the corollaries below,
which fully describe $\bar{\varphi}_{\phi,*}^{s}$ at any point on
$\mathcal{P}$.
\begin{thm}
\label{thm:CentralRotation}At each $p\in\mathcal{P}$, $s\mapsto\bar{\varphi}_{\phi,*}^{s}|_{p}$
is a group homomorphism with periodicity $\bar{\varphi}_{\phi,*}^{s+2\pi}\big|_{p}=\bar{\varphi}_{\phi,*}^{s}\big|_{p}$,
and it covers $s\mapsto e^{s\cdot\mathbf{i}}$ in the way of (\ref{eq:=00005BRot=00005D-=00005Brot=00005D}).
Moreover, $\bar{\varphi}_{\phi,*}^{s}\big(\frac{\partial}{\partial\theta}\big|_{p}\big)=\frac{\partial}{\partial\theta}\big|_{p}$
and $\bar{\varphi}_{\phi,*}^{s}\big(\ker\beta\big|_{p}\big)=\ker\beta\big|_{p}$. 
\end{thm}
\begin{cor}
\label{cor:Matrix=00005BCentralRot=00005D}With respect to the basis
$\big\{\frac{\partial}{\partial\theta},\frac{\partial}{\partial x},\frac{\partial}{\partial y}\big\}$
of the tangent space $\mathrm{T}_{p}\mathcal{U}$, the matrix representation
$[\bar{\varphi}_{\phi,*}^{s}]$ takes the form (\ref{eq:LandScapeMatrix}),
in which the functions $a_{s},b_{s}$ are given by (\ref{eq:StructuralRot}).
\end{cor}
\begin{cor}
\label{cor:Limit-on-CentralFiber}For each $s\in\mathbb{R}$, $\underset{t\rightarrow0}{\lim}\frac{|\bar{\mathbf{z}}|}{r}\bar{\partial}_{r}\bigg|_{\bar{\varphi}_{\phi}^{s}(p_{t})}=\bar{\varphi}_{\phi,*}^{s}(\dot{p}_{0})$
and it is the horizontal lift of $\varphi_{\phi,*}^{s}(\frac{d\mathbf{z}_{t}}{dt}\big|_{t=0})$.
\end{cor}

\subsection{\label{subsec:HelicalShape-BetterCharacterzation}Helical Shape of
$\eta_{s}=\bar{\varphi}_{\phi}^{s}(\bar{p})$: a better characterization}

For a better understanding of the curvature of $\ker\beta$ as well
as to stimulate the discussion in the next section, we improve the characterization
in Theorem \ref{thm:=00005Brotation-lift=00005DStructureOf=00005Cbeta}
of the helical shape of the curve $\eta_{s}=\bar{\varphi}_{\phi}^{s}(\bar{p})$
with the help of its parallel projection $\Theta\circ\eta_{s}=\mathbf{0}^{\bar{\varphi}_{\phi}^{s}(\bar{p})}$
on the central fiber $\mathcal{P}=\{\mathbf{0}\}\times S^{1}$. Here,
the point $\mathbf{0}^{\bar{\varphi}_{\phi}^{s}(\bar{p})}$ is the
parallel transport of the point $\bar{\varphi}_{\phi}^{s}(\bar{p})$.
See Subsection \ref{subsec:Concept-Definition-Notation} for an introduction
to these concepts. More details can be found in Subsection \ref{subsec:ParallelProjection=000026StrategicLemma}
below together with formal statements of the definitions.

\begin{thm}
\label{thm:CharacteristicHelical}Suppose that $d\beta\big(\bar{\partial}_{\phi},\bar{\partial}_{r}\big)>0$
holds on $\mathcal{U}_{*}$. For any $\bar{p}\in\mathcal{U}_{*}$,
$\xi_{\bar{p}}^{s}:=\mathbf{0}^{\bar{\varphi}_{\phi}^{s}(\bar{p})}$
moves in the positive direction on $\{\mathbf{0}\}\times S^{1}$ in
the sense that $\xi_{\bar{p}}^{s}=(\mathbf{0},e^{i\bar{\vartheta}_{s}})$
with function $\bar{\vartheta}_{s}$ increases in $s$. 
\end{thm}
\begin{proof}
First of all, note that as a continuous lift of the curve $\xi^{s}$
in $S^{1}$ to the cover space $\mathbb{R}$, $s\mapsto\bar{\vartheta}_{s}$
is uniquely determined when the specific value of $\bar{\vartheta}_{0}$
is fixed. In other words, if $\hat{\vartheta}_{s}$ is another function
such that $\xi^{s}=(\mathbf{0},e^{i\hat{\vartheta}_{s}})$, then $\hat{\vartheta}_{s}=\bar{\vartheta}_{s}+2k\pi$
for some integer $k\in\mathbb{Z}$. 

Also, we shall point out that, for proving $\bar{\vartheta}_{s}$
to be increasing, it suffices to show for any $\bar{p}\in\mathcal{U}_{*}$
the result $\bar{\vartheta}_{s}>\bar{\vartheta}_{0}$ for all $s>0$.
Assume this to be true and consider an arbitrary pair $s_{0},s_{1}$
with $\Delta s=s_{1}-s_{0}>0$. Note that $\bar{\varphi}_{\phi}^{s_{1}}(\bar{p})=\bar{\varphi}_{\phi}^{\Delta s}\big(\bar{\varphi}_{\phi}^{s_{0}}(\bar{p})\big)$
and $\bar{\varphi}_{\phi}^{s_{0}}(\bar{p})$ is also a point in $\mathcal{U}_{*}$.
The curves $s\mapsto\xi_{\bar{\varphi}_{\phi}^{s_{0}}(\bar{p})}^{s}$
and $s\mapsto\xi_{\bar{p}}^{s}$ are thus related by
\[
\xi_{\bar{\varphi}_{\phi}^{s_{0}}(\bar{p})}^{s}:=\mathbf{0}^{\bar{\varphi}_{\phi}^{s}\big(\bar{\varphi}_{\phi}^{s_{0}}(\bar{p})\big)}=\mathbf{0}^{\bar{\varphi}_{\phi}^{s+s_{0}}(\bar{p})}=:\xi_{\bar{p}}^{s+s_{0}}.
\]
Hence, with $\xi_{\bar{\varphi}_{\phi}^{s_{0}}(\bar{p})}^{s}=(\mathbf{0},e^{i\bar{\theta}_{s}})$
and $\xi_{\bar{p}}^{s}=(\mathbf{0},e^{i\bar{\vartheta}_{s}})$, it
holds $e^{i\bar{\vartheta}_{s+s_{0}}}=e^{i\bar{\theta}_{s}}$ and
then $\exists k\in\mathbb{Z}$ s.t. $\bar{\vartheta}_{s+s_{0}}=\bar{\theta}_{s}+2k\pi$
for all $s\in\mathbb{R}$. According to the assumption, we have $\bar{\theta}_{\Delta s}>\bar{\theta}_{0}$,
and then
\[
\bar{\vartheta}_{s_{1}}=\bar{\vartheta}_{\Delta s+s_{0}}=\bar{\theta}_{\Delta s}+2k\pi>\bar{\theta}_{0}+2k\pi=\bar{\vartheta}_{s_{0}}.
\]

Now we proceed to prove the assumption adopted above, that is, for
an arbitrary $\bar{p}\in\mathcal{U}_{*}$ with $\xi_{\bar{p}}^{s}=(\mathbf{0},e^{i\bar{\vartheta}_{s}})$,
$\bar{\vartheta}_{s}>\bar{\vartheta}_{0}$. Given $s>0$, with the
condition $d\beta\big(\bar{\partial}_{\phi},\bar{\partial}_{r}\big)=\beta\big([\bar{\partial}_{r},\bar{\partial}_{\phi}]\big)>0$
on $\mathcal{U}_{*}$, Theorem \ref{thm:StructureEq.-=00005Cbeta}
assures $\kappa>0$ at those points $\bar{\varphi}_{\phi}^{s}(p_{t})$
with $t>0$. That is,
\[
\frac{d}{dt}\bar{\varphi}_{\phi}^{s}(p_{t})=\kappa\frac{\partial}{\partial\theta}+\frac{1}{t}\bar{\partial}_{r},\ \forall t>0.
\]
 Let $\bar{\mathbf{z}}^{s}=e^{s\cdot\mathbf{i}}\bar{\mathbf{z}}$
and $\psi_{\bar{\mathbf{z}}^{s}}$ be the map given in Definition
\ref{def:ParallelParametrization} for $\bar{\mathbf{z}}^{s}$. As
before, we use the variables $(e^{i\vartheta},t)$ for the points
in the domain $S^{1}\times[0,1]$ of $\psi_{\bar{\mathbf{z}}^{s}}$.
For each $e^{i\vartheta}\in S^{1}$, $t\mapsto\psi_{\bar{\mathbf{z}}^{s}}(e^{i\vartheta},t)$
is the horizontal lift (in $\mathcal{U}$) of the path $t\mapsto t\bar{\mathbf{z}}^{s}$
(in $\mathcal{B}_{\delta}$), and therefore $\psi_{\bar{\mathbf{z}}^{s},*}^{-1}\big(\frac{1}{t}\bar{\partial}_{r}\big)=\frac{\partial}{\partial t}$.
By Lemma \ref{lem:=00005Cbeta-in-ParallelParameterization}, there
is a positive function $\mu^{s}$ with the variables $(\vartheta,t)$
s.t.
\begin{equation}
\psi_{\bar{\mathbf{z}}^{s},*}\big(\frac{\partial}{\partial\vartheta}\big)=\mu^{s}\frac{\partial}{\partial\theta}\label{eq:FiberStretch-=00005BParallelParameter=00005D}
\end{equation}
that is, $\psi_{\bar{\mathbf{z}}^{s},*}^{-1}\big(\frac{\partial}{\partial\theta}\big)=\frac{1}{\mu^{s}}\frac{\partial}{\partial\vartheta}$.
As a result, the curve $t\mapsto\bar{\varphi}_{\phi}^{s}(p_{t})$
takes the form $\psi_{\bar{\mathbf{z}}^{s}}^{-1}\big(\bar{\varphi}_{\phi}^{s}(p_{t})\big)=(e^{i\vartheta_{s}^{t}},t)$
in the space $S^{1}\times[0,1]$ with
\begin{equation}
\frac{d\vartheta_{s}^{t}}{dt}=\frac{\kappa_{s,t}}{\mu_{\vartheta_{s}^{t},t}^{s}}>0,\label{eq:MovingOnCentralFiber}
\end{equation}
which yields
\begin{equation}
\vartheta_{s}^{t=1}-\vartheta_{s}^{t=0}=\int_{0}^{1}\frac{\kappa_{s,t}}{\mu_{\vartheta_{s}^{t},t}^{s}}dt>0.\label{eq:AuxiliaryLowerBound}
\end{equation}
It remains to show that $\vartheta_{s}^{t=1}-\vartheta_{s}^{t=0}=\bar{\vartheta}_{s}-\bar{\vartheta}_{0}$.
According to the definition of $\psi_{\bar{\mathbf{z}}^{s}}$, we
have $(\mathbf{0},e^{i\vartheta_{s}^{t}})=\Theta\circ\bar{\varphi}_{\phi}^{s}(p_{t})$,
and in particular, when $t=1$,
\[
(\mathbf{0},e^{i\vartheta_{s}^{1}})=\Theta\circ\bar{\varphi}_{\phi}^{s}(\bar{p})=(\mathbf{0},e^{i\bar{\vartheta}_{s}}),
\]
that is, $e^{i\vartheta_{s}^{1}}=e^{i\bar{\vartheta}_{s}}$. Provided
that $s\mapsto\vartheta_{s}^{1}$ is continuous, this implies $\bar{\vartheta}_{s}=\vartheta_{s}^{1}+2k\pi$
for some $k\in\mathbb{Z}$ and thus $\bar{\vartheta}_{s}-\bar{\vartheta}_{0}=\vartheta_{s}^{1}-\vartheta_{0}^{1}$,
and we will then conclude the proof by noting $\vartheta_{0}^{t}=\vartheta_{s}^{0}$
for all $(s,t)$, combining which with (\ref{eq:AuxiliaryLowerBound})
yields
\[
\bar{\vartheta}_{s}-\bar{\vartheta}_{0}=\vartheta_{s}^{t=1}-\vartheta_{0}^{t=1}=\vartheta_{s}^{1}-\vartheta_{s}^{0}>0.
\]

It turns out that $(s,t)\mapsto\vartheta_{s}^{t}$ is a smooth map.
Due to the smoothness of both $\Theta$ (see Theorem \ref{thm:Smoothness-=00005B=00005CTheta=00005D})
and $\bar{\varphi}_{\phi}$, their composition
\[
(s,t)\mapsto\bar{\varphi}_{\phi}^{s}(p_{t})\mapsto\Theta\circ\bar{\varphi}_{\phi}^{s}(p_{t})=(\mathbf{0},e^{i\vartheta_{s}^{t}})
\]
 is also smooth. As a result, $(s,t)\mapsto(e^{i\vartheta_{s}^{t}})$
is a smooth map, which factors through $\mathbb{R}$ as
\[
(s,t)\mapsto\vartheta_{s}^{t}\xmapsto{\exp}e^{i\vartheta_{s}^{t}}.
\]
Since $\vartheta\xmapsto{\exp}e^{i\vartheta}$ is a local diffeomorphism
from $\mathbb{R}$ to $S^{1}$, the factor $(s,t)\mapsto\vartheta_{s}^{t}$
is also smooth. 

For checking $\vartheta_{0}^{t}=\vartheta_{s}^{0}$ for all $(s,t)$,
it suffices to note that at $s=0$, $\bar{\varphi}_{\phi}^{s=0}(p_{t})=p_{t}$
is the horizontal lift of $t\mapsto t\bar{\mathbf{z}}$ and hence
\[
(\mathbf{0},e^{\vartheta_{s=0}^{t}})=\mathbf{0}^{p_{t}}\equiv\mathbf{0}^{\bar{p}},
\]
and also, at $t=0$, $p_{0}\in\mathcal{P}$ is a fixed point of $\bar{\varphi}_{\phi}^{s}$
and hence
\[
(\mathbf{0},e^{\vartheta_{s}^{t=0}})=\mathbf{0}^{\bar{\varphi}_{\phi}^{s}(p_{0})}\equiv\mathbf{0}^{p_{0}}=\mathbf{0}^{\bar{p}}.
\]
That $e^{\vartheta_{s=0}^{t}}=e^{\vartheta_{s}^{t=0}}$ then implies
\[
\vartheta_{0}^{t}-\vartheta_{s}^{0}\equiv\text{constant}=\vartheta_{0}^{t=0}-\vartheta_{s=0}^{0}=0.
\]
\end{proof}
Here we give a brief discussion for a further estimation on how fast
$\xi_{\bar{p}}^{s}$ moves on $\mathcal{P}$ as $s$ increases. Again,
we shall exploit the smooth map $(s,t)\mapsto\vartheta_{s}^{t}$ constructed
in the above proof, as well as the relation
\[
\bar{\vartheta}_{s}-\bar{\vartheta}_{0}=\vartheta_{s}^{1}-\vartheta_{s}^{0}>0.
\]
 From (\ref{eq:StructuralFormula}) we know that
\[
\kappa_{s,t}=t|\bar{\mathbf{z}}|^{2}\int_{0}^{s}e^{\int_{s'}^{s}d\beta(\bar{\partial}_{\phi},\frac{\partial}{\partial\theta})}\cdot\beta\big([\bar{\partial}_{x},\bar{\partial}_{y}]\big)ds',
\]
and then from (\ref{eq:MovingOnCentralFiber}) we get
\[
\vartheta_{s}^{t}-\vartheta_{s}^{0}=\int_{0}^{t}\frac{\kappa_{s,t'}}{\mu^{s}}dt'\geq\frac{t^{2}|\bar{\mathbf{z}}|^{2}}{2}\cdot\frac{\underset{\mathcal{U}}{\inf}\beta\big([\bar{\partial}_{x},\bar{\partial}_{y}]\big)}{\underset{\vartheta,t}{\sup}\mu_{\vartheta,t}^{s}}\cdot\int_{0}^{s}e^{\bar{c}_{0}(s-s')}ds',
\]
where $\bar{c_{0}}:=\underset{\mathcal{U}}{\inf}d\beta(\bar{\partial}_{\phi},\frac{\partial}{\partial\theta})$,
and, $\beta\big([\bar{\partial}_{x},\bar{\partial}_{y}]\big)>0$ by
assumption. In fact, the positive quantity $\mu_{\vartheta,t}^{s}$
smoothly depends on all the three parameters $(s,\vartheta,t)$ and
satisfies the following relation
\[
0<\underset{s,\vartheta,t}{\inf}\mu_{\vartheta,t}^{s}\leq\underset{s,\vartheta,t}{\sup}\mu_{\vartheta,t}^{s}<\infty.
\]
We will see in the next subsection that these properties of $\mu_{\vartheta,t}^{s}$
simply follow from the smoothness of $\Theta$. With these results,
we then have
\[
\vartheta_{s}^{t}-\vartheta_{s}^{0}\geq\frac{t^{2}|\bar{\mathbf{z}}|^{2}}{2}\cdot\frac{\underset{\mathcal{U}}{\inf}\beta\big([\bar{\partial}_{x},\bar{\partial}_{y}]\big)}{\underset{s,\vartheta,t}{\sup}|\mu_{\vartheta,t}^{s}|}\cdot\frac{e^{|\bar{c}_{0}|s}-1}{|\bar{c}_{0}|},
\]
the factor $\frac{e^{|\bar{c}_{0}|s}-1}{|\bar{c}_{0}|}$ in which
is to be replaced by $\underset{c\rightarrow0}{\lim}\frac{e^{c\cdot s}-1}{c}=s$
when $|\bar{c}_{0}|=0$. Taking $t=1$, we get
\begin{equation}
\bar{\vartheta}_{s}-\bar{\vartheta}_{0}=\vartheta_{s}^{t=1}-\vartheta_{0}^{t=1}\geq\frac{|\bar{\mathbf{z}}|^{2}}{2}\cdot\frac{\underset{\mathcal{U}}{\inf}\beta\big([\bar{\partial}_{x},\bar{\partial}_{y}]\big)}{\underset{s,\vartheta,t}{\sup}|\mu_{\vartheta,t}^{s}|}\cdot\frac{e^{|\bar{c}_{0}|s}-1}{|\bar{c}_{0}|},\label{eq:MinimumCircling}
\end{equation}
combining which with Theorem \ref{thm:CharacteristicHelical} yields
$\underset{s\rightarrow\infty}{\lim}\bar{\vartheta}_{s}=\infty$.

\subsection{\label{subsec:ParallelProjection=000026StrategicLemma}Parallel Projection
$\Theta$ and Strategic Lemma for (\ref{eq:Circling})}

From the proof of Theorem \ref{thm:CharacteristicHelical} and the
discussion thereafter, we see that the smoothness of the parallel
projection $\Theta$ plays a central role in the reasoning. Due to
its importance to the math of this work, we devote this subsection
to a detailed exposition on $\Theta$. For better reference, we state
formally its definition as below:
\begin{defn}
\label{def:ParallelProjection}The parallel projection of $\mathcal{U}$
onto $\mathcal{P}$ is the map 
\begin{equation}
\mathcal{U}\ni p\xmapsto{\Theta}\mathbf{0}^{p}\in\mathcal{P}\label{eq:ParallelProjection}
\end{equation}
such that for each point $p=(\mathbf{z},e^{i\theta})$ in $\mathcal{U}$,
$\mathbf{0}^{p}=\gamma(0)$ with $\gamma$ being the horizontal lift
of the path $t\mapsto t\mathbf{z}$ determined by $\gamma(1)=p$. 
\end{defn}
Note that the map $\Theta$ is essentially a parallel transport of
the $S^{1}$-fibers of $\mathcal{U}=\mathcal{B}_{\delta}\times S^{1}$
to the central fiber $\mathcal{P}=\{\mathbf{0}\}\times S^{1}$. It
is a standard result that such a map is a smooth submersion from $\mathcal{U}$
to $\mathcal{P}$, and, its restriction to each fiber $\{\mathbf{z}\}\times S^{1}$
is a diffeomorphism between the circles that preserves the orientation
of the fibers. A similar construction to $\Theta$ can be found, for
example, in the proof of the Ehresmann Fibration Theorem given in
\cite{Cushman2015GlobalAspectsIntegrableSystem}, the smoothness of
which essentially relies on the smooth dependence of the flow of a
vector field on the parameters. For the sake of completeness of our
discussion, we shall provide a formal treatment to the smoothness
of $\Theta$:
\begin{thm}
\label{thm:Smoothness-=00005B=00005CTheta=00005D}$\Theta$ is a submersion
from $\mathcal{U}$ to $\mathcal{P}$, and its restriction to each
fiber $\{\mathbf{z}\}\times S^{1}$ is a diffeomorphism preserving
the natural orientation of the fibers.
\end{thm}
\begin{proof}
Out of consideration from a technical point of view, we take a larger
space $\hat{\mathcal{U}}=\mathcal{B}_{\hat{\delta}}\times S^{1}$
with $\hat{\delta}>\delta$ that contains $\mathcal{U}$, and assume
that $\beta$ extends to a nondegenerate $1$-form on $\hat{\mathcal{U}}$.
Note that this is always possible, since we may first extend $\beta$
to $\hat{\mathcal{U}}$, and by the non-degeneracy of $\beta$ on $\mathcal{U}$
as well as the compactness of $\mathcal{U}$, there is always some
smaller $\hat{\mathcal{U}}$ on which $\beta$ is nondegenerate. 

The key is to construct a smooth and complete vector field $\mathcal{V}_{\mathfrak{b}}$
on the space $\hat{\mathcal{U}}\times\mathcal{B}_{\hat{\delta}}$. For clarity we denote by $p=(x,y,e^{i\theta})$ a point in $\hat{\mathcal{U}}$,
and by $\mathbf{z}=(u,v)$ a point in $\mathcal{B}_{\hat{\delta}}$.
We first construct a vector field $\mathcal{V}$ on $\hat{\mathcal{U}}\times\mathcal{B}_{\hat{\delta}}$
as follows: for each $(p,\mathbf{z})\in\mathcal{U}\times\mathcal{B}_{\delta}$,
define $\mathcal{V}(p,\mathbf{z})$ to be
\begin{equation}
\mathcal{V}(p,\mathbf{z})=\hbar_{u,v}\frac{\partial}{\partial\theta}-u\frac{\partial}{\partial x}-v\frac{\partial}{\partial y}+0\cdot\frac{\partial}{\partial u}+0\cdot\frac{\partial}{\partial y},\label{eq:Defining=00005CV}
\end{equation}
in which the coefficient $\hbar_{u,v}$ is determined by
\[
\hbar_{u,v}\frac{\partial}{\partial\theta}-u\frac{\partial}{\partial x}-v\frac{\partial}{\partial y}\in\ker\beta.
\]
From the defining equation (\ref{eq:Defining=00005CV}), we know that
each orbit $(p_{t},\mathbf{z}_{t})=(x_{t},y_{t},e^{i\theta_{t}},u_{t},v_{t})$
of $\mathcal{V}$ takes the form
\begin{equation}
(u_{t},v_{t})\equiv(u_{0},v_{0})\ \text{ and }\ (x_{t},y_{t})=(x_{0},y_{0})-t\cdot(u_{0},v_{0}).\label{eq:Orbits-of-=00005CV}
\end{equation}
Considering (\ref{eq:Orbits-of-=00005CV}) with the convexity of $\mathcal{B}_{\delta}$,
we know that for any $(\bar{p},\bar{\mathbf{z}})\in\mathcal{U}\times\mathcal{B}_{\hat{\delta}}$,
there exist time moments $\bar{t}_{\max}\geq0\geq\bar{t}_{\min}$
such that $(p_{t},\mathbf{z}_{t})\in\mathcal{U}\times\mathcal{B}_{\hat{\delta}}$
for all $t\in[\bar{t}_{\min},\bar{t}_{\max}]$.

Let $\mathfrak{b}$ be a smooth bump function on $\mathcal{B}_{\hat{\delta}}$
such that $\mathfrak{b}=1$ on $\mathcal{B}_{\hat{\delta}}$, $\mathfrak{b}=0$
on the boundary $\partial\mathcal{B}_{\hat{\delta}}$, and, $0<\mathfrak{b}\leq1$
in the interior $\mathrm{int}\;\mathcal{B}_{\hat{\delta}}$ of $\mathcal{B}_{\hat{\delta}}$.
Define $\mathcal{V}_{\mathfrak{b}}$ at each $(p,\mathbf{z})=(x,y,e^{i\theta},u,v)$
by
\[
\mathcal{V}_{\mathfrak{b}}(p,\mathbf{z}):=\mathfrak{b}(x,y)\cdot\mathcal{V}(p,\mathbf{z}),
\]
$\mathcal{V}_{\mathfrak{b}}$ is then a smooth and complete vector
field on $\hat{\mathcal{U}}\times\mathcal{B}_{\hat{\delta}}$, and
it coincides with $\mathcal{V}$ on $\mathcal{U}\times\mathcal{B}_{\hat{\delta}}$.
Moreover, in $\big(\mathrm{int}\;\hat{\mathcal{U}}\big)\times\mathcal{B}_{\hat{\delta}}$,
the trajectories of $\mathcal{V}_{\mathfrak{b}}$ also coincides with
those of $\mathcal{V}$. To be precise, an orbit $\gamma_{\mathfrak{b}}$
of $\mathcal{V}_{\mathfrak{b}}$ in $\big(\mathrm{int}\;\hat{\mathcal{U}}\big)\times\mathcal{B}_{\hat{\delta}}$
takes the form of $\gamma_{\mathfrak{b}}(t)=\gamma\circ s(t)$, where
$\gamma$ is an orbit of $\mathcal{V}$, and, $t\mapsto s(t)$ is
the solution to the following ODE on $\mathbb{R}$: 
\[
\frac{ds}{dt}=\mathfrak{b}\circ\gamma(s)\ \text{ with }s(0)=0.
\]
Then, for $(\bar{p},\bar{\mathbf{z}})\in\mathcal{U}\times\mathcal{B}_{\hat{\delta}}$,
since $\bar{t}_{\max}\geq0\geq\bar{t}_{\min}$, $\frac{ds}{dt}=1$
at any $t\in[\bar{t}_{\min},\bar{t}_{\max}]$ and hence $s(t)\equiv t$
on $[\bar{t}_{\min},\bar{t}_{\max}]$.

Denote by $\Phi_{\mathfrak{b}}$ the flow of $\mathcal{V}_{\mathfrak{b}}$.
$\Phi_{\mathfrak{b}}^{t}(\bar{p},\bar{\mathbf{z}})\in\mathcal{U}\times\mathcal{B}_{\hat{\delta}}$
for all $t\in[\bar{t}_{\min},\bar{t}_{\max}]$ and $t\mapsto\Phi_{\mathfrak{b}}^{t}(\bar{p},\bar{\mathbf{z}})$
is also an orbit of $\mathcal{V}$ in this time period. Consequently,
if $(\bar{p},\bar{\mathbf{z}})=(\bar{x},\bar{y},e^{i\bar{\theta}},\bar{u},\bar{v})$
with $(\bar{x},\bar{y})=(\bar{u},\bar{v})\in\mathcal{B}_{\delta}$,
then
\[
\Phi_{\mathfrak{b}}^{t}(\bar{p},\bar{\mathbf{z}})=(x_{t},y_{t},e^{i\theta_{t}},u_{t},v_{t})
\]
with $(u_{t},v_{t})\equiv(\bar{x},\bar{y})$, and for $t\in[\bar{t}_{\min},\bar{t}_{\max}]$,
\[
(x_{t},y_{t})=(\bar{x},\bar{y})-t(\bar{x},\bar{y}).
\]
It is then straightforward to see that $\bar{t}_{\max}\geq2$, and
hence $\Phi_{\mathfrak{b}}^{1}(\bar{p},\bar{\mathbf{z}})\in\mathcal{P}\times\{\bar{\mathbf{z}}\}$,
that is,
\[
\Phi_{\mathfrak{b}}^{1}\big(\bar{p},\mathfrak{p}(\bar{p})\big)\in\mathcal{P}\times\{\mathfrak{p}(\bar{p})\},\forall\bar{p}\in\mathcal{U}.
\]

To show that $\Theta$ is a submersion, it suffices to check that
\begin{equation}
\Phi_{\mathfrak{b}}^{1}\big(\bar{p},\mathfrak{p}(\bar{p})\big)=\big(\Theta(\bar{p}),\mathfrak{p}(\bar{p})\big).\label{eq:EssentialSmoothness}
\end{equation}
To see this, note that the derivative of $\Phi_{\mathfrak{b}}^{t}\big(\bar{p},\mathfrak{p}(\bar{p})\big)=(x_{t},y_{t},e^{i\theta_{t}},u_{t},v_{t})$
on $[\bar{t}_{\min},\bar{t}_{\max}]$ is
\[
\frac{d}{dt}\Phi_{\mathfrak{b}}^{t}\big(\bar{p},\mathfrak{p}(\bar{p})\big)=\mathcal{V}\big(\bar{p},\mathfrak{p}(\bar{p})\big),
\]
that is, 
\[
\dot{\theta}_{t}\frac{\partial}{\partial\theta}+\dot{x}_{t}\frac{\partial}{\partial x}+\dot{y}_{t}\frac{\partial}{\partial y}=\hbar_{u_{0},v_{0}}\frac{\partial}{\partial\theta}-u_{0}\frac{\partial}{\partial x}-v_{0}\frac{\partial}{\partial y}.
\]
Therefore, the curve $(x_{t},y_{t},e^{i\theta_{t}})$ in $\mathcal{U}$
is tangent to $\ker\beta$, that is, it is the horizontal lift of
\[
(x_{t},y_{t})=(\bar{x},\bar{y})-t(\bar{x},\bar{y}).
\]
The relation (\ref{eq:EssentialSmoothness}) then follows directly
from the definition of $\Theta$, and hence $\Theta$ is smooth and
submersive. 

For the preservation of fiber orientation by $\Theta$, note that,
for each $\theta$, $(\bar{x},\bar{y},e^{i\theta})$ and $\Theta(\bar{x},\bar{y},e^{i\theta})$
are the two ends of a horizontal lift of $t\mapsto t\bar{\mathbf{z}}$.
Any two of these lifts of $t\mapsto t\bar{\mathbf{z}}$ does not intersect
with each other, while all of them lie in the $2$-dimensional sheet
$\mathbf{z}_{[0,1]}\times S^{1}$. As a result, $\theta\mapsto(\bar{x},\bar{y},e^{i\theta})$
and $\theta\mapsto\Theta(\bar{x},\bar{y},e^{i\theta})$ have to rotate
in the same direction.
\end{proof}
The circling of $\eta$ in (\ref{eq:Circling}) (see Problem \ref{prob:=00005BGeneralMainProblem=00005D3D-nonHolonomic-PathFollowing})
can be characterized through its parallel projection $\hat{\eta}:=\Theta(\eta)=\mathbf{0}^{\eta}$.
This is specified and proved in the lemma below, and it serves as
the foundation in the following discussion for proving the circling
(\ref{eq:Circling}).
\begin{lem}
\label{lem:CirclingByParallelProjection}With $\mathbf{0}^{\eta_{t}}=(\mathbf{0},e^{i\hat{\theta}_{t}})$,
(\ref{eq:Circling}) holds if and only if $\underset{t\rightarrow\infty}{\lim}\hat{\theta}_{t}=\infty$. 
\end{lem}
\begin{proof}
Now that $\text{\ensuremath{\Theta}}$ is smooth and the restriction
of $\Theta$ to each $\{\mathbf{z}\}\times S^{1}$ is a diffeomorphism
between the fiber and $\mathcal{P}$, $\Theta$ induces an isomorphism
$\mathfrak{I}$ on the bundle $\mathcal{B}_{\delta}\times S^{1}$
\[
\mathcal{B}_{\delta}\times S^{1}\ni p=(\mathbf{z},e^{i\theta})\xmapsto{\mathfrak{I}}(\mathbf{z},e^{i\hat{\theta}})\in\mathcal{B}_{\delta}\times S^{1},
\]
in which $e^{i\hat{\theta}}$ is the element in $S^{1}$ such that
$(\mathbf{0},e^{i\hat{\theta}})=\mathbf{0}^{p}$. $\mathfrak{I}$
is then lifted to a homeomorphism (denoted by $\bar{\mathfrak{I}}$)
on the universal covering space $\mathcal{B}_{\delta}\times\mathbb{R}$
\[
(\mathbf{z},\theta)\xmapsto{\bar{\mathfrak{I}}}(\mathbf{z},\hat{\theta}).
\]
The continuity of $\bar{\mathfrak{I}}$ implies that it maps each
bounded area $\mathcal{B}_{\delta}\times[-N,N]$ into another bounded
area $\mathcal{B}_{\delta}\times[-N',N']$. Moreover, $\bar{\mathfrak{I}}$
maps each fiber $\{\mathbf{z}\}\times\mathbb{R}$ homeomorphically
to itself. Check that $\mathfrak{I}\big|_{\mathcal{P}}$ is the identity
map on $\mathcal{P}=\{\mathbf{0}\}\times S^{1}$, and hence $\bar{\mathfrak{I}}$
can be taken in such a way that the restriction $\bar{\mathfrak{I}}\big|_{\{\mathbf{0}\}\times\mathbb{R}}$
is the identity map on $\{\mathbf{0}\}\times\mathbb{R}$. This has
the implication that $\bar{\mathfrak{I}}$ preserves the orientation
of the fibers. With $\eta_{t}=(\mathbf{z}_{t},e^{i\theta_{t}})$ and
$\Theta\circ\eta_{t}=\mathbf{0}^{\eta_{t}}=(\mathbf{0},e^{i\hat{\theta}_{t}})$,
it holds $\mathfrak{I}\circ\eta_{t}=(\mathbf{z}_{t},e^{i\hat{\theta}_{t}})$
with $\bar{\mathfrak{I}}(\mathbf{z}_{t},\theta_{t})=(\mathbf{z}_{t},\hat{\theta}_{t})$.
If $\hat{\theta}_{t}\rightarrow\infty$ as $t\rightarrow\infty$,
then $\exists T>0$ s.t. $\hat{\theta}_{t}\in(N',\infty)$ for all
$t>T$, and as a result, $\theta_{t}>N$ for all $t>T$. Considering
the arbitrariness of $N>0$, we conclude that $\underset{t\rightarrow\infty}{\lim}\hat{\theta}_{t}\rightarrow\infty$
implies $\underset{t\rightarrow\infty}{\lim}\theta_{t}\rightarrow\infty$.
For the other direction, it suffices to note that, the inverse $\mathfrak{I}^{-1}$
is also a continuous map preserving the orientation of the fibers
$\{\mathbf{z}\}\times\mathbb{R}$, and hence $\mathfrak{I}^{-1}$
also maps any bounded region $\mathcal{B}_{\delta}\times[-N',N']$
into another region $\mathcal{B}_{\delta}\times[-N'',N'']$. The proof
is then completed with the same line of argument.
\end{proof}
The isomorphism $\mathfrak{I}$ on $\mathcal{B}_{\delta}\times S^{1}$
introduced in the above proof will be needed again in later discussion.
For better reference, we formalize its definition as below:
\begin{defn}
\label{def:BundleIsomorphism-=00005CI}Define a map $\mathfrak{I}$
on $\mathcal{B}_{\delta}\times S^{1}$ by letting $(\mathbf{z},e^{i\theta})\xmapsto{\mathfrak{I}}(\mathbf{z},e^{i\hat{\theta}})$
for each $(\mathbf{z},e^{i\theta})\in\mathcal{B}_{\delta}\times S^{1}$,
such that $(\mathbf{0},e^{i\hat{\theta}})$ is the parallel projection
of $(\mathbf{z},e^{i\theta})$, i.e.,
\[
(\mathbf{0},e^{i\hat{\theta}})=\Theta(\mathbf{z},e^{i\theta})=\mathbf{0}^{(\mathbf{z},e^{i\theta})}.
\]
\end{defn}
The smoothness of $\mathfrak{I}$ follows directly from its construction
and the smoothness of $\Theta$, and we shall look into the tangent
maps of $\Theta$ and $\mathfrak{I}$. It follows directly from the
definition of $\Theta$ that, for any $\bar{p}\in\mathcal{U}_{*}$
with $\bar{\mathbf{z}}=\mathfrak{p}(\bar{p})$, if $t\mapsto p_{t}$
is the horizontal lift of the path $\mathbf{z}_{t}=t\bar{\mathbf{z}}$,
then 
\begin{equation}
\Theta(p_{t})=\mathbf{0}^{p_{t}}=\mathbf{0}^{\bar{p}}=(\mathbf{0},e^{\hat{\theta}}),\ \forall t\in[0,1],\label{eq:RadiusAnnihilation}
\end{equation}
which implies $\Theta_{*}(\dot{p}_{t})=\mathbf{0}$. Since $\dot{p}_{t}=\frac{1}{t}\bar{\partial}_{r}$
for $t\in(0,1]$, it yields
\begin{equation}
\Theta_{*}(\bar{\partial}_{r})=\mathbf{0}.\label{eq:RadiusAnnihilation*-=00005CTheta}
\end{equation}
Since $\Theta$ preserves the orientation of the fibers, there is
a positive function $\hat{\mu}$ on $\mathcal{U}$ such that
\begin{equation}
\Theta_{*}\bigg(\frac{\partial}{\partial\theta}\bigg)=\hat{\mu}\frac{\partial}{\partial\theta}.\label{eq:FiberStretch*-=00005CTheta}
\end{equation}
In general, there exists some smooth function $\hat{\nu}$ on $\mathcal{U}$
s.t. 
\begin{equation}
\Theta_{*}(\bar{\partial}_{\phi})=\hat{\nu}\frac{\partial}{\partial\theta}.\label{eq:Rotation*-=00005CTheta}
\end{equation}
Note that $\mathfrak{p}_{*}(\bar{\partial}_{\phi})=\partial_{\phi}$
and $\mathfrak{p}_{*}(\bar{\partial}_{r})=\partial_{r}$, and then
by the construction of $\mathfrak{I}$ it holds
\begin{equation}
\mathfrak{I}_{*}\big(\bar{\partial}_{r}\big)=\partial_{r},\ \mathfrak{I}_{*}\big(\bar{\partial}_{\phi}\big)=\hat{\nu}\frac{\partial}{\partial\theta}+\partial_{\phi},\ \mathfrak{I}_{*}\bigg(\frac{\partial}{\partial\theta}\bigg)=\hat{\mu}\frac{\partial}{\partial\theta}.\label{eq:=00005CI*}
\end{equation}
As a result, 
\begin{equation}
\mathfrak{I}^{*}\big(d\theta\big)=\hat{\mu}d\theta+\hat{\nu}d\phi.\label{eq:=00005CI*(d=00005Ctheta)}
\end{equation}
Moreover, given any top form $\Omega$ on $\mathcal{U}$, it holds
\[
\mathfrak{I}^{*}\Omega\big(\frac{\partial}{\partial\theta},\partial_{r},\partial_{\phi}\big)=\mathfrak{I}^{*}\Omega\big(\frac{\partial}{\partial\theta},\bar{\partial}_{r},\bar{\partial}_{\phi}\big)=\hat{\mu}\Omega\big(\frac{\partial}{\partial\theta},\partial_{r},\partial_{\phi}\big).
\]
Since the vector fields $\frac{\partial}{\partial\theta},\partial_{r},\partial_{\phi}$
constitute a frame of the tangent bundle on $\mathcal{U}_{*}$, this
means $\mathfrak{I}^{*}\Omega=\hat{\mu}\Omega$ at least on $\mathcal{U}_{*}$,
and then by continuity it holds 
\begin{equation}
\mathfrak{I}^{*}\Omega=\hat{\mu}\Omega\ \text{ on }\mathcal{U}.\label{eq:VolumeExpansion}
\end{equation}

\subsection{Solution to Problem \ref{prob:=00005BGeneralMainProblem=00005D3D-nonHolonomic-PathFollowing}}

We end this section with the theorem below, which serves, from a theoretic
point of view, an answer to the general problem, i.e., Problem \ref{prob:=00005BGeneralMainProblem=00005D3D-nonHolonomic-PathFollowing},
and confirms the existence of a desirable control $\mathcal{X}$ for
the problem:
\begin{thm}
\label{thm:=00005BMainResult=00005DSolution-to-GeneralProblem}Suppose
that $\beta\wedge d\beta=\rho\cdot d\theta\wedge dx\wedge dy$ with
$\underset{\mathcal{U}}{\max}\rho<0$. Given any smooth functions
$\hat{c}_{\phi},\hat{c}_{r}$ on $\mathcal{U}$ with $\underset{\mathcal{U}}{\inf}\hat{c}_{\phi}>0$,
$\hat{c}_{r}>0$ on $\mathcal{U}_{*}$ and $\underset{\mathcal{U}}{\sup}\frac{\hat{c}_{r}}{r^{2}}<\infty$,
the vector field $\mathcal{X}=\hat{c}_{\phi}\cdot\bar{\partial}_{\phi}-\hat{c}_{r}\cdot\bar{\partial}_{r}$
solves Problem \ref{prob:=00005BGeneralMainProblem=00005D3D-nonHolonomic-PathFollowing}.
\end{thm}
We make some preparation before proving Theorem \ref{thm:=00005BMainResult=00005DSolution-to-GeneralProblem}. First of all, note that by Lemma \ref{lem:CirclingByParallelProjection} it suffices to show the limit $\underset{s\rightarrow\infty}{\lim}\bar{\vartheta}_{s}=\infty$ for the function $s\mapsto\bar{\vartheta}_{s}$ in $\Theta\circ\bar{\varphi}_{\phi}^{s}(\bar{p})=(\mathbf{0},e^{i\bar{\vartheta}_{s}})$.
Since (\ref{eq:RadiusAnnihilation*-=00005CTheta}), (\ref{eq:FiberStretch*-=00005CTheta})
and (\ref{eq:Rotation*-=00005CTheta}) characterize the tangent map
$\Theta_{*}$, and therefore we shall look into the functions $\hat{\mu}$
and $\hat{\nu}$. 

We should first specify the relation between $\hat{\mu}$ and the
function $(s,\vartheta,l)\mapsto\mu_{\vartheta,l}^{s}$ in the previous
subsection. By the construction of $\psi_{\bar{\mathbf{z}}^{s}}$
and $\Theta$ we have
\[
\Theta\circ\psi_{\bar{\mathbf{z}}^{s}}^{l}(e^{i\vartheta})=(\mathbf{0},e^{i\vartheta}).
\]
Combining this with (\ref{eq:FiberStretch*-=00005CTheta}) and the
definition of the quantity $\mu_{\vartheta,l}^{s}$ in (\ref{eq:FiberStretch-=00005BParallelParameter=00005D})
gives
\[
\frac{\partial}{\partial\vartheta}=\Theta_{*}\circ\psi_{\bar{\mathbf{z}}^{s},*}^{l}\bigg(\frac{\partial}{\partial\vartheta}\bigg)=\hat{\mu}\big|_{\psi_{\bar{\mathbf{z}}^{s}}^{l}(e^{i\vartheta})}\cdot\mu_{\vartheta,l}^{s}\cdot\frac{\partial}{\partial\vartheta}.
\]
Note that there is abuse of notation here, and the $\frac{\partial}{\partial\vartheta}$
on the left and the right sides above actually refers to the vector
(field) $\frac{\partial}{\partial\theta}$ on $\mathcal{U}$. As a
result, we get
\begin{equation}
\mu_{\vartheta,l}^{s}=\frac{1}{\hat{\mu}\big|_{\psi_{\bar{\mathbf{z}}^{s}}^{l}(e^{i\vartheta})}},\label{eq:FiberStretch-=00005BProjection-v.s.-Parametrization=00005D}
\end{equation}
and hence $(s,\vartheta,l)\mapsto\mu_{\vartheta,l}^{s}$ is a smooth
map. Moreover, from the relation above we also have
\begin{equation}
0<\underset{\mathcal{U}}{\inf}\hat{\mu}\leq\frac{1}{\underset{s,\vartheta,l}{\sup}\mu_{\vartheta,l}^{s}}=\inf_{s,\vartheta,l}\hat{\mu}\big|_{\psi_{\bar{\mathbf{z}}^{s}}^{l}(e^{i\vartheta})}\leq\sup_{\mathcal{U}}\hat{\mu}<\infty.\label{eq:Bounds-for-FiberStretch}
\end{equation}
These results make up the final pieces of the reasoning for (\ref{eq:MinimumCircling})
$\underset{s\rightarrow0}{\lim}\bar{\vartheta}_{s}=\infty$ in Subsection
\ref{subsec:HelicalShape-BetterCharacterzation}, where the assumption
$\beta\big([\bar{\partial}_{x},\bar{\partial}_{y}]\big)>0$ on $\mathcal{U}$
is adopted. Note that this assumption is equivalent to the condition
$\rho<0$ since
\[
\rho=\beta\wedge d\beta\big(\frac{\partial}{\partial\theta},\bar{\partial}_{x},\bar{\partial}_{y}\big)=d\beta(\bar{\partial}_{x},\bar{\partial}_{y})=-\beta\big([\bar{\partial}_{x},\bar{\partial}_{y}]\big).
\]

Now we turn to the function $\hat{\nu}$. Since $\bar{\partial}_{\phi}\big|_{\bar{p}}=\frac{d}{ds}\bar{\varphi}_{\phi}^{s}(\bar{p})$
at every $\bar{p}\in\mathcal{U}_{*}$, we have
\[
\Theta_{*}(\bar{\partial}_{\phi}\big|_{\bar{p}})=\frac{d}{ds}\Theta\circ\bar{\varphi}_{\phi}^{s}(\bar{p})=\frac{d}{ds}(\mathbf{0},e^{i\bar{\vartheta}_{s}})=\frac{d\bar{\vartheta}_{s}}{ds}\cdot\frac{\partial}{\partial\theta}.
\]
Combining with (\ref{eq:Rotation*-=00005CTheta}) yields $\hat{\nu}(\bar{p})=\frac{d\bar{\vartheta}_{s}}{ds}\bigg|_{s=0}$.
Applying (\ref{eq:MinimumCircling}) to $\frac{d\bar{\vartheta}_{s}}{ds}$
with (\ref{eq:Bounds-for-FiberStretch}) we get
\begin{equation}
\begin{aligned}\frac{d\bar{\vartheta}_{s}}{ds}\bigg|_{s=0}=\lim_{s\rightarrow0_{+}}\frac{\bar{\vartheta}_{s}-\bar{\vartheta}_{0}}{s-0}\geq & \frac{|\bar{\mathbf{z}}|^{2}}{2}\cdot\frac{\underset{\mathcal{U}}{\inf}\beta\big([\bar{\partial}_{x},\bar{\partial}_{y}]\big)}{\underset{s,\vartheta,t}{\sup}\mu_{\vartheta,t}^{s}}\cdot\lim_{s\rightarrow0_{+}}\frac{e^{|\bar{c}_{0}|s}-1}{|\bar{c}_{0}|s}\\
\geq & \frac{|\bar{\mathbf{z}}|^{2}}{2}\cdot\underset{\mathcal{U}}{\inf}\hat{\mu}\cdot\underset{\mathcal{U}}{\inf}|\rho|.
\end{aligned}
\label{eq:EssentialCirclingRate}
\end{equation}
Considering the arbitrariness of $\bar{p}=(\bar{x},\bar{y},e^{i\bar{\theta}})$
in $\mathcal{U}_{*}$ and the relation $\hat{\nu}(\bar{p})=\frac{d\bar{\vartheta}_{s}}{ds}\bigg|_{s=0}$,
we have for all $(x,y,e^{i\theta})$ in $\mathcal{U}$
\[
\hat{\nu}(x,y,e^{i\theta})\geq\frac{x^{2}+y^{2}}{2}\cdot\underset{\mathcal{U}}{\inf}\hat{\mu}\cdot\underset{\mathcal{U}}{\inf}|\rho|,
\]
or equivalently, 
\begin{equation}
\hat{\nu}(p)\geq\frac{\mathfrak{H}(p)}{2}\cdot\underset{\mathcal{U}}{\inf}\big(\hat{\mu}\cdot|\rho|\big),\ \forall p\in\mathcal{U}.\label{eq:=00005BCurvature=00005D=00005Cnu-LowerBound}
\end{equation}

The change of the expression from (\ref{eq:EssentialCirclingRate})
to (\ref{eq:=00005BCurvature=00005D=00005Cnu-LowerBound}) is meaningful.
While the function $s\mapsto\bar{\vartheta}_{s}$ and its derivative
$\frac{d\bar{\vartheta}_{s}}{ds}$ are tied to the vector field $\bar{\partial}_{\phi}$,
the function $\hat{\nu}$ comes directly from $\Theta$, which is
independent of any specific vector field and depends solely on $\ker\beta$.
In other words, from the expression of (\ref{eq:=00005BCurvature=00005D=00005Cnu-LowerBound}),
it is clear and natural that this result can be used on $\mathcal{X}$.
Now we are ready to prove Theorem \ref{thm:=00005BMainResult=00005DSolution-to-GeneralProblem}:
\begin{proof}[Proof of Theorem \ref{thm:=00005BMainResult=00005DSolution-to-GeneralProblem}]

Note that
\[
d\mathfrak{H}=2(xdx+ydy)
\]
and then $\bar{\partial}_{\phi}\in\ker d\mathfrak{H}$. From $\mathcal{X}=\hat{c}_{\phi}\cdot\bar{\partial}_{\phi}-\hat{c}_{r}\cdot\bar{\partial}_{r}$
we have on $\mathcal{U}_{*}$ the following inequality
\begin{equation}
d\mathfrak{H}(\mathcal{X})=-\hat{c}_{r}d\mathfrak{H}(\bar{\partial}_{r})=-\hat{c}_{r}\cdot r^{2}<0,\label{eq:s-to-r^2}
\end{equation}
as a result of which the requirement for convergence (\ref{eq:Convergence})
is fulfilled, that is,
\[
r_{s}^{2}:=r^{2}\big|_{\varphi_{\mathcal{X}}^{s}(\bar{p})}\rightarrow0\ \text{ as }s\rightarrow\infty.
\]

Since $\mathfrak{H}=r^{2}$, the result $d\mathfrak{H}(\mathcal{X})<0$
on $\mathcal{U}_{*}$ also implies that, for an arbitrary $\bar{p}\in\mathcal{U}_{*}$,
the function
\[
s\mapsto\varphi_{\mathcal{X}}^{s}(\bar{p})\mapsto r^{2}\big|_{\varphi_{\mathcal{X}}^{s}(\bar{p})}=r_{s}^{2}
\]
is strictly decreasing such that 
\[
\frac{dr_{s}^{2}}{ds}=d\mathfrak{H}(\mathcal{X})=-\hat{c}_{r}\circ\varphi_{\mathcal{X}}^{s}(\bar{p})\cdot r_{s}^{2}<0.
\]
Consequently, the inverse function $r^{2}\xmapsto{\mathfrak{s}}s$
is also smooth, and, its derivative is
\[
\frac{ds}{dr^{2}}=-\frac{1}{\hat{c}_{r}\circ\varphi_{\mathcal{X}}^{\mathfrak{s}(r^{2})}(\bar{p})\cdot r^{2}}.
\]
From the condition $\underset{\mathcal{U}_{*}}{\sup}\frac{\hat{c}_{r}}{r^{2}}<\infty$
we know that on $\mathcal{U}_{*}$ it holds
\[
\hat{c}_{r}=r^{2}\cdot\frac{\hat{c}_{r}}{r^{2}}\leq r^{2}\cdot\bigg(\underset{\mathcal{U}_{*}}{\sup}\frac{\hat{c}_{r}}{r^{2}}\bigg)=\mathfrak{\hat{c}}\cdot r^{2},
\]
where $\mathfrak{\hat{c}}=\underset{\mathcal{U}_{*}}{\sup}\frac{\hat{c}_{r}}{r^{2}}$
is a positive real number. As a result, 
\begin{equation}
\frac{ds}{dr^{2}}=-\frac{1}{\hat{c}_{r}\circ\varphi_{\mathcal{X}}^{\mathfrak{s}(r^{2})}(\bar{p})\cdot r^{2}}\geq-\frac{1}{\mathfrak{\hat{c}}\cdot r^{4}},\label{eq:=00005CRadius-to-Time-=00005BMinimumRate=00005D}
\end{equation}

On the other hand, with
\[
(\mathbf{0},e^{i\bar{\vartheta}_{s}})=\mathbf{0}^{\varphi_{\mathcal{X}}^{s}(\bar{p})}=\Theta\circ\varphi_{\mathcal{X}}^{s}(\bar{p}),
\]
we have for every $s\in[0,\infty)$ the relation
\[
\frac{d\bar{\vartheta}_{s}}{ds}\cdot\frac{\partial}{\partial\theta}=\Theta_{*}(\mathcal{X})=\hat{c}_{\phi}\cdot\hat{\nu}\cdot\frac{\partial}{\partial\theta},
\]
that is, $\frac{d\bar{\vartheta}_{s}}{ds}=\hat{c}_{\phi}\cdot\hat{\nu}$.
Consequently, we have
\[
\begin{aligned}\bar{\vartheta}_{s}-\bar{\vartheta}_{0}=\int_{0}^{s}\frac{d\bar{\vartheta}_{s'}}{ds'}ds' & =\int_{0}^{s}\hat{c}_{\phi}\cdot\hat{\nu}ds'\\
 & \geq\frac{\underset{\mathcal{U}}{\inf}\hat{c}_{\phi}\cdot\underset{\mathcal{U}}{\inf}\big(\hat{\mu}\cdot|\rho|\big)}{2}\int_{0}^{s}r_{s'}^{2}ds'.
\end{aligned}
\]
Applying (\ref{eq:=00005CRadius-to-Time-=00005BMinimumRate=00005D})
to $\int_{0}^{s}r_{s'}^{2}ds'$ yields
\[
\int_{0}^{s}r_{s'}^{2}ds'=\int_{r_{0}}^{r_{s}}r_{s'}^{2}\cdot\frac{ds'}{dr^{2}}dr^{2}\geq-\int_{r_{0}}^{r_{s}}\frac{r^{2}}{\mathfrak{\hat{c}}\cdot r^{4}}dr^{2}=\frac{1}{\mathfrak{\hat{c}}}\ln\frac{r_{0}^{2}}{r_{s}^{2}}.
\]
Since $r_{s}^{2}\xrightarrow{s\rightarrow\infty}0_{+}$, we have $\ln r_{s}^{2}\rightarrow-\infty$,
and then from
\begin{equation}
\bar{\vartheta}_{s}-\bar{\vartheta}_{0}\geq\frac{\underset{\mathcal{U}}{\inf}\big(\hat{c}_{\phi}\cdot\hat{\mu}\cdot|\rho|\big)}{2\mathfrak{\hat{c}}}\cdot\bigg(\ln r_{0}^{2}-\ln r_{s}^{2}\bigg)\label{eq:LowerBound-=00005Bcircling/convergence=00005D}
\end{equation}
we deduce that $\underset{s\rightarrow\infty}{\lim}\bar{\vartheta}_{s}=\infty$, which by Lemma \ref{lem:CirclingByParallelProjection} completes the proof.
\end{proof}

\section{\label{sec:nonHolo-PathFollowing=00005BR^3=00005D}Path-Following
on $\mathbb{R}^{3}$ with Completely non-Holonomic Constraint}

Although the general problem (i.e., Problem \ref{prob:=00005BGeneralMainProblem=00005D3D-nonHolonomic-PathFollowing})
has been answered by Theorem \ref{thm:=00005BMainResult=00005DSolution-to-GeneralProblem},
conditions on the weight functions $\bar{a},\bar{b}$ for constructing
$\mathcal{X}$ with (\ref{eq:=00005B=00005CX=00005DnonHolonomic-PathFollow=00005B3D=00005D})
are not specified in this theorem, and thus Theorem \ref{thm:=00005BMainResult=00005D},
which directly answers Problem \ref{prob:MainProblem-concrete=000026specified},
still remains unproven. In this section, we study the motion generated
by $\mathcal{X}$ with (\ref{eq:=00005B=00005CX=00005DnonHolonomic-PathFollow=00005B3D=00005D})
and prove Theorem \ref{thm:=00005BMainResult=00005D}. While the conditions
in these two theorems are similar in style, and indeed, Theorem \ref{thm:=00005BMainResult=00005D}
could have been proved through the results (esp. Theorem \ref{thm:=00005BMainResult=00005DSolution-to-GeneralProblem})
in Section \ref{sec:Structures-of-Constraint}, we still conduct in
this section a similar but independent analysis on $\mathcal{X}$
directly based on (\ref{eq:=00005B=00005CX=00005DnonHolonomic-PathFollow=00005B3D=00005D}).
The discussion here is (almost) self-contained, and the only thing
needed from Section \ref{sec:Structures-of-Constraint} is the smoothness
of $\Theta$.

As in Section \ref{sec:Structures-of-Constraint}, we assume $\beta\big(\frac{\partial}{\partial\theta}\big)=1$
for convenience. According to (\ref{eq:ConvergenceAnalysis-ConvergenceTerm})
and (\ref{eq:ConvergenceAnalysis-RotationTerm}), it holds
\[
d\mathfrak{H}\big(\mathcal{X}\big)=\bar{b}\cdot\bigg(\mathcal{V}_{\beta}\times\big(\frac{\partial}{\partial\theta}\times\nabla\mathfrak{H}\big)\bigg)\cdot\nabla\mathfrak{H}=-\bar{b}\cdot\big|\big|\nabla\mathfrak{H}\big|\big|^{2},
\]
and hence the condition $\bar{b}>0$ on $\mathcal{U}_{*}$ implies
the convergence of $\varphi_{\mathcal{X}}^{s}(\bar{p})$ to the desired
path $\mathcal{P}:\mathfrak{H}=0$ as $s\rightarrow\infty$. The analysis
in this section will then focus on proving the requirement of circling
(\ref{eq:Circling}) in Problem \ref{prob:MainProblem-concrete=000026specified}
under the conditions on $\bar{a},\bar{b}$ given in Theorem \ref{thm:=00005BMainResult=00005D}.
\begin{rem}
\label{rem:invariance-of-=00005CU}It is important to note that, $\mathcal{U}=\mathcal{B}_{\delta}\times S^{1}$
is the sub-level set $\mathfrak{H}\leq\delta$ , and hence $\bar{b}\geq0$
implies $\mathcal{U}$ to be a forward-invariant set under the flow
$\varphi_{\mathcal{X}}$. Moreover, due to the compactness of $\mathcal{U}$,
$\varphi_{\mathcal{X}}^{s}(p)$ is well defined for all $(p,s)\in\mathcal{U}\times[0,\infty)$. 
\end{rem}
The proof of (\ref{eq:Circling}) (i.e., $\int_{\eta}d\theta=\infty$)
for Theorem \ref{thm:=00005BMainResult=00005D} will be similar to
the proof of Theorem \ref{thm:CharacteristicHelical} about $\bar{\partial}_{\phi}$.
Indeed, with $\bar{\mathbf{z}}=\mathfrak{p}(\bar{p})$ and the horizontal
lift $p_{t}$ of $\mathbf{z}_{t}=t\bar{\mathbf{z}}$, the ``vector
field'' $\partial_{\bar{\mathbf{z}}}^{\mathcal{X}}:=\varphi_{\mathcal{X},*}^{s}(\dot{p}_{t})$
admits a decomposition similar to (\ref{eq:CharacteristicRotationField}):
\begin{equation}
\partial_{\bar{\mathbf{z}}}^{\mathcal{X}}=\bar{\kappa}_{s,t}\frac{\partial}{\partial\theta}+\xi_{s,t}\ \text{with }\,\xi_{s,t}\in\ker\beta,\label{eq:StructuralDecomposition}
\end{equation}
in which the quantity $\bar{\kappa}_{s,t}$ satisfies an equation
similar to (\ref{eq:StructuralEq-for-=00005CS=000026=00005Cbeta})
\begin{equation}
\frac{\partial}{\partial s}\bar{\kappa}_{s,t}-\bar{\kappa}_{s,t}\cdot d\beta(\mathcal{X},\frac{\partial}{\partial\theta})=d\beta(\mathcal{X},\xi),\ \forall t>0,\label{eq:=00005B=00005CX=00005DStructuralEquation}
\end{equation}
or equivalently,
\begin{equation}
\frac{\partial}{\partial s}\bar{\kappa}_{s,t}=d\beta(\mathcal{X},\partial_{\bar{\mathbf{z}}}^{\mathcal{X}}),\ \forall t>0.\label{eq:=00005B=00005CX=00005DStructuralEquation=00005Bcompact=00005D}
\end{equation}
Roughly speaking, the idea is to first show $\bar{\kappa}_{s,t}>0$
with the conditions in Theorem \ref{thm:=00005BMainResult=00005D}
and then prove (\ref{eq:Circling}) with the help of the parallel
projection $\Theta$.
\begin{notation}
\label{nota:SimplifedNotation}For convenience, we will use the simplified
notations $\mathcal{X}_{s,t}:=\mathcal{X}\big|_{\varphi_{\mathcal{X}}^{s}(p_{t})}$,
$\partial_{\bar{\mathbf{z}}}^{\mathcal{X}}\big|_{s,t}:=\partial_{\bar{\mathbf{z}}}^{\mathcal{X}}\big|_{\varphi_{\mathcal{X}}^{s}(p_{t})}$,
$\bar{a}_{s,t}:=\bar{a}\big|_{\varphi_{\mathcal{X}}^{s}(p_{t})}$
and $\bar{b}_{s,t}:=\bar{b}\big|_{\varphi_{\mathcal{X}}^{s}(p_{t})}$,
and similarly, for the other vector fields and functions on $\mathcal{U}$.
\end{notation}

\subsection{Derivation of (\ref{eq:=00005B=00005CX=00005DStructuralEquation})
and (\ref{eq:=00005B=00005CX=00005DStructuralEquation=00005Bcompact=00005D})
and Some Technical Results}

In this subsection, we consider $\mathcal{X}$ given by \eqref{eq:=00005B=00005CX=00005DnonHolonomic-PathFollow=00005B3D=00005D}
and impose no requirement on the weight functions $\bar{a},\bar{b}$
except for smoothness. 

Define a map $\Phi$ from $\mathbb{R}_{\geq}\times[0,1]$ to $\mathcal{U}=\mathcal{B}_{\delta}\times S^{1}$
by $\Phi(s,t)=\varphi_{\mathcal{X}}^{s}(p_{t})$. It holds $\Psi_{*}(\frac{\partial}{\partial s})=\mathcal{X}$
and
\begin{equation}
\Phi_{*}(\frac{\partial}{\partial t})=\varphi_{\mathcal{X},*}^{s}(\dot{p}_{t})=\bar{\kappa}_{s,t}\frac{\partial}{\partial\theta}+\xi_{s,t}\ \text{ with }\,\xi_{s,t}\in\ker\beta.\label{eq:CoreDecomposition}
\end{equation}
Let $\bar{\beta}$ be the pullback of $\beta$ through $\Phi$ on
$\mathbb{R}_{\geq0}\times[0,1]$, i.e., $\bar{\beta}=\Phi^{*}(\beta)$.
Check that $\bar{\beta}(\frac{\partial}{\partial s})=\beta(\mathcal{X})=0$
and
\[
\frac{\partial}{\partial s}\bar{\beta}(\frac{\partial}{\partial t})=\frac{\partial}{\partial s}\beta\circ\varphi^{s}(\dot{p}_{t})=\frac{\partial}{\partial s}\bar{\kappa}_{s,t},
\]
and then
\[
\frac{\partial}{\partial s}\bar{\beta}(\frac{\partial}{\partial t})-\frac{\partial}{\partial t}\bar{\beta}(\frac{\partial}{\partial s})-\bar{\beta}\big([\frac{\partial}{\partial s},\frac{\partial}{\partial t}\big]\big)=d\bar{\beta}(\frac{\partial}{\partial s},\frac{\partial}{\partial t})
\]
implies $\frac{\partial}{\partial s}\bar{\kappa}_{s,t}=d\bar{\beta}(\frac{\partial}{\partial s},\frac{\partial}{\partial t})$.
Observe that $d\bar{\beta}(\frac{\partial}{\partial s},\frac{\partial}{\partial t})=d\beta(\mathcal{X},\partial_{\bar{\mathbf{z}}}^{\mathcal{X}})$,
and hence
\[
\frac{\partial}{\partial s}\bar{\kappa}_{s,t}=d\beta(\mathcal{X},\partial_{\bar{\mathbf{z}}}^{\mathcal{X}})=\bar{\kappa}_{s,t}d\beta(\mathcal{X}_{s,t},\frac{\partial}{\partial\theta})+d\beta(\mathcal{X}_{s,t},\xi_{s,t}),
\]
yielding the structure equations (\ref{eq:=00005B=00005CX=00005DStructuralEquation})
and \eqref{eq:=00005B=00005CX=00005DStructuralEquation=00005Bcompact=00005D}.
Also, note that $\partial_{\bar{\mathbf{z}}}^{\mathcal{X}}\big|_{0,t}=\dot{p}_{t}\in\ker\beta$,
and therefore $\xi_{0,t}=\dot{p}_{t}$ and $\bar{\kappa}_{0,t}=0$.
Solving (\ref{eq:=00005B=00005CX=00005DStructuralEquation}) with
this initial condition $\bar{\kappa}_{0,t}=0$ gives
\begin{equation}
\bar{\kappa}_{s,t}=e^{\int_{0}^{s}d\beta(\mathcal{X}_{s,t},\frac{\partial}{\partial\theta})}\cdot\int_{0}^{s}e^{-\int_{0}^{s'}d\beta(\mathcal{X}_{s'',t},\frac{\partial}{\partial\theta})ds''}d\beta(\mathcal{X}_{s',t},\xi_{s',t})ds'.\label{eq:=00005B=00005CX=00005DStructuralCharacteristic}
\end{equation}
We formalize the this conclusion into the following theorem for later
reference.
\begin{thm}
\label{thm:StructureEq.=00005B=00005CX=00005D}Let $\mathcal{X}$
be a vector field defined by \eqref{eq:=00005B=00005CX=00005DnonHolonomic-PathFollow=00005B3D=00005D}
with arbitrary $\bar{a}$ and $\bar{b}$, and $\varphi_{\mathcal{X}}$
be its flow. Given any $\bar{\mathbf{z}}\in\mathcal{B}_{\delta}$
and the line $\mathbf{z}_{t}=t\bar{\mathbf{z}}$, let $p_{t}$ be
a horizontal lift of $\mathbf{z}_{t}$. In the decomposition \eqref{eq:CoreDecomposition}
of $\partial_{\bar{\mathbf{z}}}^{\mathcal{X}}:=\varphi_{\mathcal{X},*}^{s}(\dot{p}_{t})$,
the quantity $\bar{\kappa}_{s,t}$ satisfies the differential equations
(\ref{eq:=00005B=00005CX=00005DStructuralEquation}) and (\ref{eq:=00005B=00005CX=00005DStructuralEquation=00005Bcompact=00005D})
with $\bar{\kappa}_{0,t}=0$. Consequently, $\bar{\kappa}_{s,t}$
is expressed by (\ref{eq:=00005B=00005CX=00005DStructuralCharacteristic}).
\end{thm}
The following result follows directly from Equation (\ref{eq:=00005B=00005CX=00005DStructuralEquation=00005Bcompact=00005D})
for any $\mathcal{X}$ given in \eqref{eq:=00005B=00005CX=00005DnonHolonomic-PathFollow=00005B3D=00005D},
and will be used in later discussion.
\begin{thm}
\label{thm:=00005CX-PositiveHolonomyCoefficient}Given any $\mathcal{X}$
in (\ref{eq:=00005B=00005CX=00005DnonHolonomic-PathFollow=00005B3D=00005D}),
there exists $\bar{\epsilon}>0$ such that, for any $\bar{t}\in(0,1]$,
\[
d\beta(\mathcal{X}_{0,\bar{t}},\partial_{\bar{\mathbf{z}}}^{\mathcal{X}}\big|_{0,\bar{t}})>0\ \iff\ d\beta(\mathcal{X}_{s,\bar{t}},\partial_{\bar{\mathbf{z}}}^{\mathcal{X}}\big|_{s,\bar{t}})>0,\forall s\in[0,\bar{\epsilon}].
\]
Consequently, if $d\beta(\mathcal{X},\bar{\partial}_{r})>0$ holds
on $\mathcal{U}_{*}$, $\bar{\kappa}_{s,t}>0$ holds for all $s\in(0,\bar{\epsilon}]$
and $t\in(0,1]$.
\end{thm}
\begin{proof}
Let $\bar{\epsilon}>0$ be the same as in Lemma \ref{lem:TechnicalLemma-1}
below. Suppose that $d\beta(\mathcal{X}_{0,\bar{t}},\partial_{\bar{\mathbf{z}}}^{\mathcal{X}}\big|_{0,\bar{t}})>0$.
Since $\varphi_{\mathcal{X},*}^{s}\big(\bar{\partial}_{r}\big|_{0,\bar{t}}\big)=\partial_{\bar{\mathbf{z}}}^{\mathcal{X}}\big|_{s,\bar{t}}$
and $\varphi_{\mathcal{X},*}^{s}\big(\mathcal{X}_{0,\bar{t}}\big)=\mathcal{X}_{s,\bar{t}}$,
by the continuity of $d\beta^{s}$ in $s\in[0,\bar{\epsilon}]$, we
have
\[
d\beta(\mathcal{X}_{s,\bar{t}},\partial_{\bar{\mathbf{z}}}^{\mathcal{X}}\big|_{s,\bar{t}})=d\beta^{s}(\mathcal{X}_{0,\bar{t}},\bar{\partial}_{r}\big|_{0,\bar{t}})>0.
\]
The other direction of the equivalence relation is trivial.

Now suppose that $d\beta(\mathcal{X},\bar{\partial}_{r})>0$. Then,
\[
d\beta(\mathcal{X}_{0,\bar{t}},\partial_{\bar{\mathbf{z}}}^{\mathcal{X}}\big|_{0,\bar{t}})=d\beta(\mathcal{X}_{0,\bar{t}},\bar{\partial}_{r}\big|_{0,\bar{t}})>0\implies d\beta(\mathcal{X}_{s,\bar{t}},\partial_{\bar{\mathbf{z}}}^{\mathcal{X}}\big|_{s,\bar{t}})>0.
\]
 From \eqref{eq:=00005B=00005CX=00005DStructuralEquation=00005Bcompact=00005D}
and the condition $\bar{\kappa}_{0,\bar{t}}=0$ we get for each $s\in(0,\bar{\epsilon}]$
\[
\bar{\kappa}_{s,\bar{t}}=\int_{0}^{s}d\beta(\mathcal{X}_{s',\bar{t}},\partial_{\bar{\mathbf{z}}}^{\mathcal{X}}\big|_{s',\bar{t}})ds'>0.
\]
\end{proof}
Note that Lemma \ref{lem:TechnicalLemma-1} works for any smooth vector
field on $\mathcal{U}$, not just for the $\mathcal{X}$ given by
(\ref{eq:=00005B=00005CX=00005DnonHolonomic-PathFollow=00005B3D=00005D}).
\begin{lem}
\label{lem:TechnicalLemma-1}Suppose that $\beta\wedge d\beta$ is
nondegenerate on $\mathcal{U}$. For any vector field $X$ on $\mathcal{U}$
with flow $\varphi_{X}$, if $\mathcal{U}$ is a forward-invariant
set of $\varphi_{X}$, then there exists $\bar{\epsilon}>0$ such
that, for each $s\in[0,\bar{\epsilon}]$ and the pullback $\beta^{s}:=\varphi_{X}^{s,*}\big(\beta\big)$,
the restriction $d\beta^{s}\big|_{\ker\beta}$ of $d\beta^{s}$to
$\ker\beta$ is nondegenerate.
\end{lem}
\begin{proof}
Since the vector fields $\bar{\partial}_{x}$ and $\bar{\partial}_{y}$
constitute a frame of the distribution $\ker\beta$, $d\beta^{s}\big|_{\ker\beta}\neq0$
exactly means $d\beta^{s}(\bar{\partial}_{x}\big|_{p},\bar{\partial}_{y}\big|_{p})\neq0$.
Note that the function
\[
\hslash:[0,\infty)\times\mathcal{U}\ni(s,p)\mapsto d\beta^{s}(\bar{\partial}_{x}\big|_{p},\bar{\partial}_{y}\big|_{p})\in\mathbb{R}
\]
is well defined and smooth, and the condition $\beta\wedge d\beta\neq0$
everywhere on $\mathcal{U}$ just implies for each $p\in\mathcal{U}$
and $s=0$
\[
\hslash(0,p)=d\beta(\bar{\partial}_{x},\bar{\partial}_{y})\neq0.
\]
The continuity of $\hslash$ and the compactness of $\mathcal{U}$
then imply the existence of $\bar{\epsilon}>0$ such that $\hslash\neq0$
on $[0,\bar{\epsilon}]\times\mathcal{U}$, that is, $d\beta^{s}(\bar{\partial}_{x}\big|_{p},\bar{\partial}_{y}\big|_{p})\neq0$
for all $(s,p)\in[0,\bar{\epsilon}]\times\mathcal{U}$, which concludes
the proof.
\end{proof}
For showing $\bar{\kappa}_{s,t}>0$ alternatively through Equation
(\ref{eq:=00005B=00005CX=00005DStructuralEquation}), Lemma \ref{lem:TechnicalLemma-2}
(given below) can be used in place of Lemma \ref{lem:TechnicalLemma-1}
for getting the condition $d\beta(\mathcal{X}_{s,t},\xi_{s,t})>0$,
the proof for which follows a standard line of argument by continuity
and compactness.
\begin{lem}
\label{lem:TechnicalLemma-2}Suppose that $\beta\wedge d\beta$ is
nondegenerate on $\mathcal{U}$. Let $\mathcal{X}$ be a vector field
given by (\ref{eq:=00005B=00005CX=00005DnonHolonomic-PathFollow=00005B3D=00005D})
with $\bar{b}\geq0$ and $d\beta\big(\mathcal{X}_{0,t},\xi_{0,t}\big)\neq0$
for all $t\in(0,1]$. There exists $\bar{\epsilon}>0$ such that for
all $(s,t)\in[0,\bar{\epsilon}]\times(0,1]$, $d\beta(\mathcal{X}_{s,t},\xi_{s,t})\neq0$.
\end{lem}
\begin{proof}
Now that $\varphi_{\mathcal{X},*}^{s}(\dot{p}_{t})=\bar{\kappa}_{s,t}\frac{\partial}{\partial\theta}+\xi_{s,t}$
and $\mathcal{X}_{s,t}=\varphi_{\mathcal{X},*}^{s}(\mathcal{X}_{0,t})$,
we have
\[
\beta\wedge d\beta\big(\frac{\partial}{\partial\theta},\bar{\kappa}_{s,t}\frac{\partial}{\partial\theta}+\xi_{s,t},\mathcal{X}_{s,t}\big)=\beta\wedge d\beta\big(\frac{\partial}{\partial\theta},\xi_{s,t},\mathcal{X}_{s,t}\big)=d\beta\big(\xi_{s,t},\mathcal{X}_{s,t}\big).
\]
Therefore, $d\beta(\mathcal{X}_{s,t},\xi_{s,t})\neq0$ if and only
if $\frac{\partial}{\partial\theta}$, $\varphi_{\mathcal{X},*}^{s}(\dot{p}_{t})$
and $\mathcal{X}_{s,t}$ are linearly independent. Since all the points
on $\mathcal{P}$ are fixed points of $\varphi_{\mathcal{X}}^{s}$
for all $s$, it holds $\varphi_{\mathcal{X},*}^{s}(\frac{\partial}{\partial\theta})=\frac{\partial}{\partial\theta}$
on $\mathcal{P}$. As a result, for any $s\in\mathbb{R}$, $\frac{\partial}{\partial\theta}\big|_{\mathcal{P}}=\varphi_{\mathcal{X},*}^{s}\big(\frac{\partial}{\partial\theta}\big|_{\mathcal{P}}\big)$,
$\varphi_{\mathcal{X},*}^{s}\big(\bar{\partial}_{x}\big|_{\mathcal{P}}\big)$
and $\varphi_{\mathcal{X},*}^{s}\big(\bar{\partial}_{y}\big|_{\mathcal{P}}\big)$
are linearly independent. By the compactness of $\mathcal{P}$, for
any $\phi>0$, there exists a small neighborhood $\mathcal{U}_{\phi}$
of $\mathcal{P}$, such that for every $(s,p)\in[-\phi,\phi]\times\mathcal{U}_{\phi}$,
$\frac{\partial}{\partial\theta}\big|_{\varphi^{s}(p)}$, $\varphi_{\mathcal{X},*}^{s}\big(\bar{\partial}_{x}\big|_{p}\big)$
and $\varphi_{\mathcal{X},*}^{s}\big(\bar{\partial}_{y}\big|_{p}\big)$
are linearly independent. Note that $\dot{p}_{t}$ and $\mathcal{X}_{0,t}$
are linearly independent in $\ker\beta\big|_{p_{t}}=\text{span}\{\bar{\partial}_{x}\big|_{p_{t}},\bar{\partial}_{y}\big|_{p_{t}}\}$,
and so for any $\bar{t}>0$ s.t. $p_{[0,\bar{t}]}\subset\mathcal{U}_{\phi}$,
it holds for each $(s,t)\in[-\phi,\phi]\times(0,\bar{t}]$ that
\[
\text{span}\big\{\varphi_{\mathcal{X},*}^{s}(\dot{p}_{t}),\mathcal{X}_{s,t}\big\}=\text{span}\big\{\varphi_{\mathcal{X},*}^{s}(\bar{\partial}_{x}),\varphi_{\mathcal{X},*}^{s}(\bar{\partial}_{y})\big\}.
\]
and thence $\frac{\partial}{\partial\theta}$, $\varphi_{\mathcal{X},*}^{s}(\dot{p}_{t})$
and $\mathcal{X}_{s,t}=\varphi_{\mathcal{X},*}^{s}(\mathcal{X}_{0,t})$
are linearly independent. 

On the other hand, for each $t\in(0,1]$, $\frac{\partial}{\partial\theta}\big|_{p_{t}}$,
$\dot{p}_{t}$ and $\mathcal{X}_{0,t}$ are linearly independent,
and then there exists a neighborhood $U_{t}$ of $p_{t}$ and some
small number $\epsilon_{t}$, such that for any $(s,t)\in U_{t}\times[0,\epsilon_{t}]$,
the vectors $\frac{\partial}{\partial\theta}\big|_{\varphi_{\mathcal{X}}^{s}(p_{t})}$,
$\varphi_{\mathcal{X},*}^{s}(\dot{p}_{t})$ and $\mathcal{X}_{s,t}$
are also linearly independent. Since the segment $p_{[0,1]}$ is compact,
there exists a finite group $t_{1},...,t_{k}$ s.t. $\mathcal{U}_{\phi}$,
$U_{t_{1}}$, ... $U_{t_{k}}$ covers $p_{[0,1]}$. Take $\bar{\epsilon}=\min\{\phi,\epsilon_{t_{1}},..,\epsilon_{t_{k}}\}$
and then $d\beta(\mathcal{X}_{s,t},\xi_{s,t})\neq0$ holds for all
$(s,t)\in[0,\bar{\epsilon}]\times(0,1]$. 
\end{proof}
\begin{rem}
\label{rem:Technical=00005Cepsilon}We can take $\bar{\epsilon}>0$
to be small enough such that it satisfies Theorem \ref{thm:=00005CX-PositiveHolonomyCoefficient}
as well as Lemmas \ref{lem:TechnicalLemma-1} and \ref{lem:TechnicalLemma-2}.
Since this number will be frequently used in the following discussion,
we will keep the symbol $\bar{\epsilon}$ and call it ``$\textit{the technical epsilon}$
for $\mathcal{X}$''.
\end{rem}

\subsection{\label{subsec:MonotoneCircling=00005B=00005CX=00005D}Monotone Circling
of $\xi_{s}^{\mathcal{X}}:=\Theta\circ\varphi_{\mathcal{X}}^{s}(\bar{p})$}

Now that an arbitrary $\mathcal{X}$ given by (\ref{eq:=00005B=00005CX=00005DnonHolonomic-PathFollow=00005B3D=00005D})
has a structural equation (\ref{eq:=00005B=00005CX=00005DStructuralEquation})
similar to the structural equation (\ref{eq:StructuralEq-for-=00005CS=000026=00005Cbeta})
in the case $\mathcal{X}=\bar{\partial}_{\phi}$ (i.e., the case $\bar{a}=\frac{1}{\lambda}$
and $\bar{b}=0$), we may expect a similar result to Theorem \ref{thm:CharacteristicHelical}
for a general $\mathcal{X}$ with $d\beta(\mathcal{X},\bar{\partial}_{r})>0$.
Indeed, we have the following result:
\begin{thm}
\label{thm:=00005B=00005CX=00005DPositiveHolonomy}Suppose that $d\beta(\mathcal{X},\bar{\partial}_{r})>0$
on $\mathcal{U}_{*}$, $|\bar{a}|>0,\bar{b}\geq0$ on $\mathcal{U}$,
and, $\sup_{\mathcal{U}_{*}}\frac{|\bar{b}|}{r^{2}}<\infty$. There
exists a neighborhood $\mathcal{U}^{\mathfrak{f}}$ of $\mathcal{P}$
in $\mathcal{U}$ such that, for each $\bar{p}\in\mathcal{U}_{*}^{\mathfrak{f}}$
and $\eta_{s}^{\mathcal{X}}:=\varphi_{\mathcal{X}}^{s}(\bar{p})$,
any continuous function $\hat{\theta}_{s}$ for $\xi_{s}^{\mathcal{X}}:=\mathbf{0}^{\varphi_{\mathcal{X}}^{s}(\bar{p})}=(\mathbf{0},e^{i\hat{\theta}_{s}})$
increases in $s$. Here, $\mathcal{U}_{*}^{\mathfrak{f}}=\mathcal{U}^{\mathfrak{f}}\setminus\mathcal{P}$
and $\mathcal{U}^{\mathfrak{f}}=\mathcal{B}_{\delta_{\mathfrak{f}}}\times S^{1}$
with $\delta_{\mathfrak{f}}<\delta$.
\end{thm}
\begin{rem}
Now that $\mathcal{X},\bar{\partial}_{r}$ are tangent to $\ker\beta$
and $\beta\big(\frac{\partial}{\partial\theta}\big)=1$, from $d\beta(\mathcal{X},\bar{\partial}_{r})\neq0$
we have
\[
\beta\wedge d\beta\big(\frac{\partial}{\partial\theta},\mathcal{X},\bar{\partial}_{r}\big)=\beta\big(\frac{\partial}{\partial\theta}\big)\cdot d\beta(\mathcal{X},\bar{\partial}_{r})\neq0.
\]
Recall that the constraint $\beta=0$ being completely non-holonomic
is equivalent to the non-degeneracy of $\beta\wedge d\beta$, and therefore
the condition $d\beta(\mathcal{X},\bar{\partial}_{r})>0$ on $\mathcal{U}_{*}$
implies the constraint to be completely non-holonomic on $\mathcal{U}_{*}$.
\end{rem}
In order to conduct similar analysis to that in Theorem \ref{thm:CharacteristicHelical}
for $\bar{\partial}_{\phi}$, we still need to make up for a significant
difference between the flows of $\mathcal{X}$ and $\bar{\partial}_{\phi}$.
While $\frac{d}{dt}\bar{\varphi}_{\phi}^{s}(p_{t})=\kappa\frac{\partial}{\partial\theta}+\frac{1}{t}\bar{\partial}_{r}$,
we have $\frac{d}{dt}\varphi_{\mathcal{X}}^{s}(p_{t})=\bar{\kappa}_{s,t}\frac{\partial}{\partial\theta}+\xi_{s,t}$
and $\xi_{s,t}$ is not necessarily parallel to $\bar{\partial}_{r}$,
which prevents direct implementation of the same line of argument
for $\varphi_{\mathcal{X}}$. To fix this, we introduce a modified
vector field $\mathcal{X}_{\mathfrak{f}}=\mathfrak{f}\cdot\mathcal{X}$
with some positive function $\mathfrak{f}$, such that the flow $\varphi_{\mathfrak{f}}$
has the property
\begin{equation}
\varphi_{\mathfrak{f},*}^{s}(\dot{p}_{t})=\kappa_{s,t}^{\mathfrak{f}}\frac{\partial}{\partial\theta}+\nu_{s,t}\bar{\partial}_{r}.\label{eq:=00005Bf=00005CX=00005DHomogeneousStructure}
\end{equation}
Note that the trajectories of $\mathcal{X}_{\mathfrak{f}}$ coincide
with those $\mathcal{X}$ with identical orientations, and hence the
circling property of $\mathcal{X}$ can be deduced from that of $\mathcal{X}_{\mathfrak{f}}$. 

We shall deduce a specific construction for $\mathfrak{f}$ from the
desired property (\ref{eq:=00005Bf=00005CX=00005DHomogeneousStructure}).
Note that $t\mapsto p_{t}$ is the horizontal lift of $\mathbf{z}_{t}=t\bar{\mathbf{z}}$
and
\[
\dot{p}_{t}=\frac{1}{t}\bar{\partial}_{r}\in\text{span}\{\frac{\partial}{\partial\theta},\bar{\partial}_{r}\},
\]
while the right-hand side of the equation (\ref{eq:=00005Bf=00005CX=00005DHomogeneousStructure})
also lies in the space $\text{span}\{\frac{\partial}{\partial\theta},\bar{\partial}_{r}\}$.
The idea is then to construct $\mathfrak{f}$ so that the flow $\varphi_{\mathfrak{f}}$
preserves $\text{span}\{\frac{\partial}{\partial\theta},\bar{\partial}_{r}\}$.
With the closed form $d\phi=\frac{xdy-ydx}{r^{2}}$ on $\mathcal{U}_{*}$,
check that
\[
\text{span}\{\frac{\partial}{\partial\theta},\bar{\partial}_{r}\}=\ker d\phi\ \text{ on }\mathcal{U}_{*}.
\]
Therefore, what we need is simply $\varphi_{\mathfrak{f}}^{s,*}d\phi\equiv d\phi$,
or equivalently,
\[
\frac{d}{ds}\varphi_{\mathfrak{f}}^{s,*}d\phi=\varphi_{\mathfrak{f}}^{s,*}\mathcal{L}_{\mathcal{X}_{\mathfrak{f}}}d\phi\equiv0.
\]
Here $\mathcal{L}_{\mathcal{X}_{\mathfrak{f}}}$ denotes the Lie derivative
of $d\phi$ by $\mathcal{X}_{\mathfrak{f}}$, and the equation above
holds if and only if $\mathcal{L}_{\mathcal{X}_{\mathfrak{f}}}d\phi\equiv0$.

Since $d\phi$ is a closed $1$-form on $\mathcal{U}_{*}$, by Cartan's
Magic Formula (e.g., see \cite{Lee2012SmoothManifold,marsden1999introduction}),
the Lie derivative of $d\phi$ by any vector field $\mathcal{V}$
is
\begin{equation}
\mathcal{L}_{\mathcal{V}}d\phi=d\iota_{\mathcal{V}}d\phi+\iota_{\mathcal{V}}d\bigg(d\phi\bigg)=d\bigg(d\phi(\mathcal{V})\bigg),\label{eq:LieDerivative=00005Bd=00005Cphi=00005D}
\end{equation}
which means that $\mathcal{L}_{\mathcal{V}}d\phi=0$ if and only if
$d\phi(\mathcal{V})\equiv\text{const}$. Also note that $\varphi_{\mathcal{V}}^{s,*}\big(d\phi\big)=d\phi\circ\varphi_{\mathcal{V},*}^{s}$
and hence $\varphi_{\mathcal{V}}^{s,*}\big(d\phi\big)=d\phi$ implies
\[
\varphi_{\mathcal{V},*}^{s}\big(\ker d\phi\big)\subset\ker d\phi.
\]

So what we need is exactly $d\phi(\mathcal{X}_{\mathfrak{f}})=\mathfrak{f}\cdot d\phi(\mathcal{X})\equiv\text{constant}$,
or simply, $\mathfrak{f}=\frac{1}{\big|d\phi(\mathcal{X})\big|}$,
which requires $d\phi(\mathcal{X})\neq0$ (at least in a vicinity
of $\mathcal{P}$). Check that on $\mathcal{U}_{*}$, it holds

\[
\begin{aligned}d\phi(\mathcal{X})= & \bar{a}\lambda\cdot d\phi(\bar{\partial}_{\phi})+\bar{b}\cdot\big(\mathcal{V}_{\beta}\cdot\nabla\mathfrak{H}\big)d\phi(\frac{\partial}{\partial\theta})-\bar{b}\cdot d\phi(\nabla\mathfrak{H})\\
= & \bar{a}\lambda-\bar{b}\cdot d\phi(\nabla\mathfrak{H}).
\end{aligned}
\]
The conditions $\bar{b}\geq0$ and $\underset{\mathcal{U}_{*}}{\sup}\frac{\bar{b}}{r^{2}}<\infty$
imply $\underset{\mathcal{U}_{*}}{\sup}\frac{|\bar{b}|}{r^{2}}<\infty$,
and hence we have
\[
\underset{\mathcal{U}_{*}}{\sup}\big|\bar{b}\cdot d\phi(\nabla\mathfrak{H})\big|=\underset{\mathcal{U}_{*}}{\sup}\bigg(\frac{|\bar{b}|}{r^{2}}\cdot\big|\overline{d\phi}(\nabla\mathfrak{H})\big|\bigg)\leq\underset{\mathcal{U}_{*}}{\sup}\frac{|\bar{b}|}{r^{2}}\cdot\max_{\mathcal{U}}\big|\overline{d\phi}(\nabla\mathfrak{H})\big|<\infty,
\]
where $\overline{d\phi}:=xdy-ydx$ and it is smoothly defined on the
whole $\mathcal{U}$. Furthermore, since
\[
\big|\bar{b}\cdot d\phi(\nabla\mathfrak{H})\big|=\frac{|\bar{b}|}{r^{2}}\cdot\big|\overline{d\phi}(\nabla\mathfrak{H})\big|\leq\bigg(\underset{\mathcal{U}_{*}}{\sup}\frac{|\bar{b}|}{r^{2}}\bigg)\cdot\big|\overline{d\phi}(\nabla\mathfrak{H})\big|,
\]
we deduce $\underset{r\rightarrow0}{\lim}\big|\bar{b}\cdot d\phi(\nabla\mathfrak{H})\big|=0$
with the Squeezing Theorem from
\[
0\leq\lim_{r\rightarrow0}\big|\bar{b}\cdot d\phi(\nabla\mathfrak{H})\big|\leq\bigg(\underset{\mathcal{U}_{*}}{\sup}\frac{|\bar{b}|}{r^{2}}\bigg)\cdot\lim_{r\rightarrow0}\big|\overline{d\phi}(\nabla\mathfrak{H})\big|=0.
\]
On the other hand, since $\lambda,\bar{a}\neq0$ everywhere on the
compact space $\mathcal{U}$, it holds $\underset{\mathcal{U}}{\inf}|\bar{a}|,\underset{\mathcal{U}}{\inf}|\lambda|>0$.
As a result, $d\phi(\mathcal{X})=\bar{a}\lambda-\bar{b}\cdot d\phi(\nabla\mathfrak{H})$
converges to the nonzero function $\bar{a}\lambda$ as $r\rightarrow0$.
In other words, the function $d\phi(\mathcal{X})$ extends to a continuous
function on $\mathcal{U}$ by setting $d\phi(\mathcal{X})=\bar{a}\lambda$
on $\mathcal{P}$. We then have the desirable property $d\phi(\mathcal{X})\neq0$
in a vicinity of $\mathcal{P}$. In particular, there exists a neighborhood
$\mathcal{U}^{\mathfrak{f}}=\mathcal{B}_{\delta_{\mathfrak{f}}}\times S^{1}$
of $\mathcal{P}$ in $\mathcal{U}=\mathcal{B}_{\delta}\times S^{1}$
(with $\delta_{\mathfrak{f}}<\delta$), on which

\begin{equation}
|\bar{a}|>\frac{|\bar{b}|\cdot\big|d\phi(\nabla\mathfrak{H})\big|}{|\lambda|}.\label{eq:ExtraCondition-for-=00005Cf}
\end{equation}
The function $\mathfrak{f}=\frac{1}{\big|d\phi(\mathcal{X})\big|}$
is then well defined and continuous on $\mathcal{U}^{\mathfrak{f}}$
and is smooth on $\mathcal{U}_{*}^{\mathfrak{f}}=\mathcal{U}^{\mathfrak{f}}\setminus\mathcal{P}$.
Define $\mathcal{X}_{\mathfrak{f}}:=\mathfrak{f}\mathcal{X}$, and
then $\big|d\phi(\mathcal{X}_{\mathfrak{f}})\big|\equiv1$ on $\mathcal{U}_{*}^{\mathfrak{f}}$.
The connectedness of $\mathcal{U}_{*}^{\mathfrak{f}}$ then implies
\[
\text{either}\ d\phi(\mathcal{X}_{\mathfrak{f}})\equiv1\ \text{ or }\ d\phi(\mathcal{X}_{\mathfrak{f}})\equiv-1\ \text{ on }\mathcal{U}_{*}^{\mathfrak{f}}.
\]
Applying (\ref{eq:LieDerivative=00005Bd=00005Cphi=00005D}) we get
\begin{equation}
\mathcal{L}_{\mathcal{X}_{\mathfrak{f}}}d\phi=0.\label{eq:LieDerivative=00005Bd=00005Cphi-by-=00005CX_f=00005D}
\end{equation}
Moreover, for any integral curve $s\mapsto\gamma(s)$ of $\mathcal{X}$
in $\mathcal{U}_{\mathfrak{f}}$, the composite $s\mapsto\gamma(h_{s})$
is an integral curve of $\mathcal{X}_{\mathfrak{f}}$, where the function
$s\mapsto h_{s}$ is determined by the following ODE:
\[
\frac{dh_{s}}{ds}=\mathfrak{f}\circ\gamma(h_{s})\ \text{ with}\ h_{0}=0.
\]
Since $\mathfrak{f}=\frac{1}{\big|d\phi(\mathcal{X})\big|}$ is continuous
on the compact space $\mathcal{U}^{\mathfrak{f}}$, it holds
\[
0<\inf_{\mathcal{U}}\mathfrak{f}\leq\sup_{\mathcal{U}}\mathfrak{f}<\infty,
\]
which implies that $s\mapsto h_{s}$ is strictly increasing. Since
the curve $\gamma$ is defined at all $s\in[0,\infty)$, the part
$\sup_{\mathcal{U}}\mathfrak{f}<\infty$ further implies $s\mapsto h_{s}$
to be well defined on the whole interval $[0,\infty)$, and hence
$s\mapsto\gamma(h_{s})$ is an integral curve on $\mathcal{X}_{\mathfrak{f}}$
in $\mathcal{U}_{*}^{\mathfrak{f}}$ which exists on $[0,\infty)$.
This means that $\mathcal{U}_{*}^{\mathfrak{f}}$ is a forward-invariant
set of both the flows $\varphi_{\mathcal{X}}$ and $\varphi_{\mathfrak{f}}$,
and, $\varphi_{\mathfrak{f}}^{s}$ (as well as $\varphi_{\mathcal{X}}^{s}$)
exists on $\mathcal{U}_{*}^{\mathfrak{f}}$ for all $s\geq0$. It
then follows from (\ref{eq:LieDerivative=00005Bd=00005Cphi-by-=00005CX_f=00005D})
that $\varphi_{\mathfrak{f}}^{s,*}\big(d\phi\big)=d\phi$, and as
a result, when $\bar{p}$ falls in $\mathcal{U}_{*}^{\mathfrak{f}}=\mathcal{U}^{\mathfrak{f}}\setminus\mathcal{P}$,
\[
d\phi\big(\varphi_{\mathfrak{f},*}^{s}(\dot{p}_{t})\big)=d\phi\big(\frac{1}{t}\bar{\partial}_{r}\big)=0.
\]
This means that $\varphi_{\mathfrak{f},*}^{s}(\dot{p}_{t})$ stays
in $\ker_{d\phi}=\text{span}\{\frac{\partial}{\partial\theta},\bar{\partial}_{r}\}$
for all $s\in\mathbb{R}$, and hence (\ref{eq:=00005Bf=00005CX=00005DHomogeneousStructure})
holds for all $\bar{p}\in\mathcal{U}_{*}^{\mathfrak{f}}$. For later
reference, we sum up all these basic results about $\mathcal{X}_{\mathfrak{f}}$
and its flow $\varphi_{\mathfrak{f}}$ in the following proposition:
\begin{prop}
\label{prop:=00005B=00005CX_f=00005DSmoothFlow-SemiGlobal}Suppose
that $|\bar{a}|>0,\bar{b}\geq0$ on $\mathcal{U}$ with $\underset{\mathcal{U}_{*}}{\sup}\frac{|\bar{b}|}{r^{2}}<\infty$.
There exists a compact neighborhood $\mathcal{U}^{\mathfrak{f}}=\mathcal{B}_{\delta_{\mathfrak{f}}}\times S^{1}$
of $\mathcal{P}$ in $\mathcal{U}$, on which the function $\mathfrak{f}=\frac{1}{|d\phi(\mathcal{X})|}$
is well defined and continuous. Moreover, $\mathfrak{f}$ is smooth
on $\mathcal{U}_{*}^{\mathfrak{f}}:=\mathcal{U}^{\mathfrak{f}}\setminus\mathcal{P}$,
and the vector field $\mathcal{X}_{\mathfrak{f}}:=\mathfrak{f}\cdot\mathcal{X}$
on $\mathcal{U}_{*}^{\mathfrak{f}}$ has a smooth flow map
\[
\mathcal{U}_{*}^{\mathfrak{f}}\times[0,\infty)\ni(p,s)\xmapsto{\varphi_{\mathfrak{f}}}\varphi_{\mathfrak{f}}^{s}(p)\in\mathcal{U}_{*}^{\mathfrak{f}}.
\]
For each $s\in[0,\infty)$, it holds $\varphi_{\mathfrak{f}}^{s,*}\big(d\phi\big)=d\phi$
on $\mathcal{U}_{*}^{\mathfrak{f}}$, and thence
\[
\varphi_{\mathfrak{f},*}^{s}\big(\ker d\phi\big)=\ker d\phi.
\]
\end{prop}
The quantity $\kappa_{s,t}^{\mathfrak{f}}$ in (\ref{eq:=00005Bf=00005CX=00005DHomogeneousStructure})
also satisfies the ODEs (\ref{eq:=00005B=00005CX=00005DStructuralEquation})
and (\ref{eq:=00005B=00005CX=00005DStructuralEquation=00005Bcompact=00005D}).
For clarity we write down these equations for $\kappa_{s,t}^{\mathfrak{f}}$
as below:
\begin{equation}
\frac{\partial}{\partial s}\kappa_{s,t}^{\mathfrak{f}}-\kappa_{s,t}^{\mathfrak{f}}\cdot d\beta(\mathcal{X}_{\mathfrak{f}},\frac{\partial}{\partial\theta})=\nu_{s,t}\cdot d\beta(\mathcal{X}_{\mathfrak{f}},\bar{\partial}_{r}),\ \forall t>0,\label{eq:=00005Bf=00005CX=00005DStructuralEquation}
\end{equation}
and with $\partial_{\bar{\mathbf{z}}}^{\mathfrak{f}}=\varphi_{\mathfrak{f}}^{*}(\dot{p}_{t})$,
\begin{equation}
\frac{\partial}{\partial s}\kappa_{s,t}^{\mathfrak{f}}=d\beta(\mathcal{X}_{\mathfrak{f}},\partial_{\bar{\mathbf{z}}}^{\mathfrak{f}}\big|_{s,t}),\ \forall t>0.\label{eq:=00005Bf=00005CX=00005DStructuralEquation=00005Bcompact=00005D}
\end{equation}
Note that $\partial_{\bar{\mathbf{z}}}^{\mathfrak{f}}\big|_{0,t}=\dot{p}_{t}$
and then
\[
d\beta(\mathcal{X}_{\mathfrak{f}}\big|_{0,t},\partial_{\bar{\mathbf{z}}}^{\mathfrak{f}}\big|_{0,t})=\mathfrak{f}\cdot d\beta(\mathcal{X},\bar{\partial}_{r})\big|_{p_{t}}.
\]
The following lemma is then a direct consequence of Lemma \ref{lem:TechnicalLemma-1}
for $X=\mathcal{X}_{\mathfrak{f}}$ on $\mathcal{U}^{\mathfrak{f}}$,
with the ``$\textit{technical epsilon}$'' $\bar{\epsilon}$ therein
becomes correspondingly the $\bar{\epsilon}_{\mathfrak{f}}$ in the
statement:
\begin{lem}
\label{lem:=00005Cf-modified=00005B=00005CX=00005D}Let $\mathcal{X}$
be a vector field given by (\ref{eq:=00005B=00005CX=00005DnonHolonomic-PathFollow=00005B3D=00005D})
with $|\bar{a}|>0$ and $\bar{b}\geq0$ on $\mathcal{U}$, $\underset{\mathcal{U}_{*}}{\sup}\frac{\bar{b}}{r^{2}}<\infty$,
and, $d\beta(\mathcal{X},\bar{\partial}_{r})>0$ on $\mathcal{U}_{*}$.
There exists $\bar{\epsilon}_{\mathfrak{f}}>0$, such that for any
$\bar{p}$ in $\mathcal{U}_{*}^{\mathfrak{f}}=\mathcal{U}^{\mathfrak{f}}\setminus\mathcal{P}$,
$d\beta(\mathcal{X}_{\mathfrak{f}}\big|_{s,t},\partial_{\bar{\mathbf{z}}}^{\mathfrak{f}}\big|_{s,t})>0$
holds for all $(s,t)\in(0,\bar{\epsilon}_{\mathfrak{f}}]\times(0,1]$,
and as a result, $\kappa_{s,t}^{\mathfrak{f}}>0$ on $(0,\bar{\epsilon}_{\mathfrak{f}}]\times(0,1]$.
\end{lem}
\begin{proof}
Now that $d\beta(\mathcal{X},\bar{\partial}_{r})>0$ on $\mathcal{U}$,
for any $\bar{p}\in\mathcal{U}_{*}^{\mathfrak{f}}$ it holds
\[
d\beta(\mathcal{X}_{\mathfrak{f}}\big|_{0,t},\partial_{\bar{\mathbf{z}}}^{\mathfrak{f}}\big|_{0,t})=\mathfrak{f}\cdot d\beta(\mathcal{X},\bar{\partial}_{r})\big|_{p_{t}}>0.
\]
The existence of $\bar{\epsilon}_{\mathfrak{f}}>0$ then follows from
Lemma \ref{lem:TechnicalLemma-1}, and the proof is concluded by
\[
\kappa_{s,t}^{\mathfrak{f}}=\int_{0}^{s}d\beta(\mathcal{X}_{\mathfrak{f}}\big|_{s',t},\partial_{\bar{\mathbf{z}}}^{\mathfrak{f}}\big|_{s',t})ds'>0,\ \forall s\in[0,\bar{\epsilon}_{\mathfrak{f}}].
\]
\end{proof}
Theorem \ref{thm:=00005B=00005CX=00005DPositiveHolonomy} will then
come as a consequence of Theorem \ref{thm:=00005Bf=00005CX=00005DHolonomy}
below. However, before proving Theorem \ref{thm:=00005Bf=00005CX=00005DHolonomy},
we need to take a close look at the behavior of the flow $\varphi_{\mathfrak{f}}$
near the central fiber $\mathcal{P}$. Note that in the discussion
above we confirm the existence of $\varphi_{\mathfrak{f}}$ on $\mathcal{U}_{*}^{\mathfrak{f}}\times[0,\infty)$,
and, to certain extent, we are also clear about the curve $\varphi_{\mathfrak{f}}^{s}(p_{t})$
for all $t\in(0,1]$. In the proof of Theorem \ref{thm:=00005Bf=00005CX=00005DHolonomy},
we will also need the limit $\underset{t\rightarrow0}{\lim}\varphi_{\mathfrak{f}}^{s}(p_{t})=\mathbf{0}^{\bar{p}}$.
It turns out that we can have a stronger result, which is stated as
the following proposition:
\begin{prop}
\label{prop:=00005B=00005CX_f=00005DGlobalFlow-ContinuousExtension}The
flow map $\varphi_{\mathfrak{f}}$ extends continuously to $\mathcal{P}\times[0,\infty)$
such that $\varphi_{\mathfrak{f}}^{s}(p)=p$ for each $(p,s)\in\mathcal{P}\times[0,\infty)$.
As a result, the extended map (also denoted by $\varphi_{\mathfrak{f}}$)
maps $\mathcal{U}_{*}^{\mathfrak{f}}\times[0,\infty)$ to $\mathcal{U}_{*}^{\mathfrak{f}}$,
and, $\mathcal{P}\times[0,\infty)$ to $\mathcal{P}$.
\end{prop}
\begin{proof}
Extend $\varphi_{\mathfrak{f}}$ to $\mathcal{P}\times[0,\infty)$
by setting $\varphi_{\mathfrak{f}}^{s}(p)=p$ for each $(p,s)\in\mathcal{P}\times[0,\infty)$.
We need to show that the extended $\varphi_{\mathfrak{f}}$ is a continuous
map from $\mathcal{U}^{\mathfrak{f}}\times[0,\infty)$ to $\mathcal{U}^{\mathfrak{f}}$.
More specifically, we need to prove the continuity of the extension
at each $(\bar{q},\bar{s})\in\mathcal{P}\times[0,\infty)$.

We first make some preparation. Given an arbitrary $p\in\mathcal{U}_{*}^{\mathfrak{f}}$,
the flows $\varphi_{\mathfrak{f}}$ and $\varphi_{\mathcal{X}}$ have
the following relation
\begin{equation}
\varphi_{\mathfrak{f}}^{s}(p)=\varphi_{\mathcal{X}}^{h_{s}^{p}}(p)\label{eq:TechnicalReparametrization}
\end{equation}
with the function $s\mapsto h_{s}^{p}$ determined by the ODE:
\[
\frac{dh_{s}^{p}}{ds}=\mathfrak{f}\circ\varphi_{\mathcal{X}}^{h_{s}}(p)\ \text{ with}\ h_{0}^{p}=0.
\]
Since $\mathfrak{f}$ is continuous on $\mathcal{U}^{\mathfrak{f}}$
and $\mathcal{U}^{\mathfrak{f}}$ is compact, it holds
\[
0<\min\mathfrak{f}\leq\frac{dh_{s}^{p}}{ds}\leq\max\mathfrak{f}<\infty,
\]
and then with the initial condition $h_{0}^{p}=0$ we have the following
result:
\begin{equation}
0\leq s\cdot\min\mathfrak{f}\leq h_{s}^{p}\leq s\cdot\max\mathfrak{f}<\infty,\ \forall p\in\mathcal{U}_{*}^{\mathfrak{f}}.\label{eq:=00005B=00005Cf=00005DTimeExpansion}
\end{equation}
In the meantime, $\varphi_{\mathcal{X}}$ is smoothly defined on $\mathcal{U}\times[0,\infty)$,
and, $\varphi_{\mathcal{X}}^{s}(p)=p$ for all $(p,s)\in\mathcal{P}\times[0,\infty)$.
Given any $T>0$, by the continuity of $\varphi_{\mathcal{X}}$ and
the compactness of $\{p\}\times[0,T]$, we know that, for any neighborhood
$\hat{U}$ of $p$, there exists a small neighborhood $U_{0}$ of
$p$ such that
\[
\varphi_{\mathcal{X}}^{[0,T]}(U_{0})\subset\hat{U}.
\]

Now given $(\bar{q},\bar{s})\in\mathcal{P}\times[0,\infty)$ and an
arbitrary neighborhood $\hat{U}$ of $\bar{q}$, we look for a neighborhood
$U_{0}\times I$ of $(\bar{q},\bar{s})$ in $\mathcal{U}^{\mathfrak{f}}\times[0,\infty)$
such that $\varphi_{\mathfrak{f}}^{I}(U_{0})\subset\hat{U}$. Take
$T=(\bar{s}+1)\cdot(1+\max\mathfrak{f})$. Now that $\varphi_{\mathcal{X}}$
is smoothly defined on $\mathcal{U}\times[0,\infty)$ with $\varphi_{\mathcal{X}}^{s}(\bar{q})=\bar{q}\in\hat{U}$
for all $s\in[0,\infty)$, it then follows from the compactness of
$\{\bar{q}\}\times[0,T]$ that there is some (smaller) neighborhood
$U_{0}$ (contained in $\hat{U}$) of $\bar{q}$ such that $\varphi_{\mathcal{X}}^{[0,T]}(U_{0})\subset\hat{U}$.
Meanwhile, check that, for any $s\geq0$ with $|s-\bar{s}|<\frac{1}{2}$,
it holds
\[
0\leq s\cdot\min\mathfrak{f}\leq s\cdot\max\mathfrak{f}\leq T,
\]
and hence $h_{s}^{p}\in[0,T]$ for every $p\in\mathcal{U}_{*}^{\mathfrak{f}}$,
and in particular, for every $p\in U_{0}\setminus\mathcal{P}$. Take
$I=(\bar{s}-\frac{1}{2},\bar{s}+\frac{1}{2})\cap[0,\infty)$. Then
$I$ is a neighborhood of $\bar{s}$ in $[0,\infty)$, and for any
$(p,s)\in\big(U_{0}\setminus\mathcal{P}\big)\times I$ it holds
\[
\varphi_{\mathfrak{f}}^{s}(p)=\varphi_{\mathcal{X}}^{h_{s}^{p}}(p)\in\varphi_{\mathcal{X}}^{[0,T]}(U_{0})\subset\hat{U}.
\]
Since $U_{0}$ is particularly taken to be a subset of $\hat{U}$,
for those $(p,s)\in\big(U_{0}\cap\mathcal{P}\big)\times I$, it follows
directly from the definition of the extended $\varphi_{\mathfrak{f}}$
that
\[
\varphi_{\mathfrak{f}}^{s}(p)=p\in\hat{U}.
\]
As a result, we have $\varphi_{\mathfrak{f}}^{s}(p)\in\hat{U}$ for
all $(p,s)\in U_{0}\times I$, which confirms the continuity of $\varphi_{\mathfrak{f}}$
at (an arbitrary point) $(\bar{q},\bar{s})$ in $\mathcal{P}\times[0,\infty)$.
\end{proof}
\begin{cor}
\label{cor:BasicLimit}With the extended flow map $\varphi_{\mathfrak{f}}$,
given any $\bar{p}\in\mathcal{U}_{*}^{\mathfrak{f}}$, the path $t\mapsto\varphi_{\mathfrak{f}}^{s}(p_{t})$
with $t\in[0,1]$ is continuous, and, 
\[
\underset{t\rightarrow0}{\lim}\varphi_{\mathfrak{f}}^{s}(p_{t})=p_{0}=\mathbf{0}^{\bar{p}}.
\]
Moreover, with $\bar{\mathbf{z}}^{s}:=\mathfrak{p}\circ\varphi_{\mathfrak{f}}^{s}(\bar{p})$
and $\mathbf{z}_{t}^{s}=t\bar{\mathbf{z}}^{s}$, the path $t\mapsto\varphi_{\mathfrak{f}}^{s}(p_{t})$
lies in the slice $\mathbf{z}_{[0,1]}^{s}\times S^{1}$. 
\end{cor}
\begin{proof}
The limit $\underset{t\rightarrow0}{\lim}\varphi_{\mathfrak{f}}^{s}(p_{t})=\mathbf{0}^{\bar{p}}$
follows directly from the continuity of the extension of $\varphi_{\mathfrak{f}}$
on $\mathcal{U}^{\mathfrak{f}}\times[0,\infty)$. That the path $t\mapsto\varphi_{\mathfrak{f}}^{s}(p_{t})$
lies in the slice $\mathbf{z}_{[0,1]}^{s}\times S^{1}$ is just the
consequence of the equation (\ref{eq:=00005Bf=00005CX=00005DHomogeneousStructure})
about $\frac{d}{dt}\varphi_{\mathfrak{f}}^{s}(p_{t})$ for $t\in(0,1]$
and the continuity of the path at $t=0$.
\end{proof}
\begin{thm}
\label{thm:=00005Bf=00005CX=00005DHolonomy}Let $\mathcal{X}$ be
a vector field in (\ref{eq:=00005B=00005CX=00005DnonHolonomic-PathFollow=00005B3D=00005D})
such that $|\bar{a}|>0$ and $\bar{b}\geq0$ on $\mathcal{U}$, $\underset{\mathcal{U}_{*}}{\sup}\frac{\bar{b}}{r^{2}}<\infty$,
and, $d\beta(\mathcal{X},\bar{\partial}_{r})>0$ on $\mathcal{U}_{*}$.
Then, for each $\bar{p}\in\mathcal{U}_{*}^{\mathfrak{f}}$, any continuous
function $s\mapsto\hat{\theta}_{s}^{\mathfrak{f}}$ with $\mathbf{0}^{\varphi_{\mathfrak{f}}^{s}(\bar{p})}=(\mathbf{0},e^{i\hat{\theta}_{s}^{\mathfrak{f}}})$
increases in $s$ on $[0,\infty)$. 
\end{thm}
\begin{proof}
We first show that it suffices to prove $\hat{\theta}_{s}^{\mathfrak{f}}>\hat{\theta}_{0}^{\mathfrak{f}}$
for any $s\in(0,\bar{\epsilon}_{\mathfrak{f}}]$. For any (other)
$s_{0}<s_{1}$ in $[0,\bar{\epsilon}_{\mathfrak{f}}]$, since $\bar{p}_{s_{0}}=\varphi_{\mathfrak{f}}^{s_{0}}(\bar{p})\in\mathcal{U}_{*}^{\mathfrak{f}}$
and $\Delta s=s_{1}-s_{0}\in(0,\bar{\epsilon}_{\mathfrak{f}}]$, this
result will imply for any continuous function $s\mapsto\vartheta_{s}^{s_{0}}$
with $\mathbf{0}^{\varphi_{\mathfrak{f}}^{s}(\bar{p}_{s_{0}})}=(\mathbf{0},e^{i\vartheta_{s}^{s_{0}}})$
the relation $\vartheta_{s=s_{1}}^{s_{0}}>\vartheta_{s=0}^{s_{0}}$.
Since $\varphi_{\mathfrak{f}}^{s}(\bar{p}_{s_{0}})=\varphi_{\mathfrak{f}}^{s+s_{0}}(\bar{p})$,
we have $(\mathbf{0},e^{i\hat{\theta}_{s+s_{0}}^{\mathfrak{f}}})=(\mathbf{0},e^{i\vartheta_{s}^{s_{0}}})$
for all $s$, and hence there is an integer $k\in\mathbb{Z}$ such
that $\hat{\theta}_{s+s_{0}}^{\mathfrak{f}}=\vartheta_{s}^{s_{0}}+2k\pi$.
The proof for the whole theorem will then be concluded by noting 
\[
\hat{\theta}_{s_{1}}^{\mathfrak{f}}=\vartheta_{\Delta s}^{s_{0}}+2k\pi>\vartheta_{0}^{s_{0}}+2k\pi=\hat{\theta}_{s_{0}}^{\mathfrak{f}}.
\]

To show $\hat{\theta}_{s}^{\mathfrak{f}}>\hat{\theta}_{0}^{\mathfrak{f}}$
for arbitrary $s\in(0,\bar{\epsilon}_{\mathfrak{f}}]$, we consider
two paths, $\eta^{s}:s'\mapsto\varphi_{\mathfrak{f}}^{s'\cdot s}(p_{t=1})$
with $s'\in[0,1]$, and, $\zeta^{s}:t\mapsto\varphi_{\mathfrak{f}}^{s}(p_{t})$
with $t\in[0,1]$. Note that at the ends $\eta^{s}$ and $\zeta^{s}$,
\begin{equation}
\mathbf{0}^{\eta^{s}(1)}=(\mathbf{0},e^{i\hat{\theta}_{s}^{\mathfrak{f}}})=\mathbf{0}^{\zeta^{s}(1)}\label{eq:LeftEndProjection}
\end{equation}
since $\eta^{s}(1)=\varphi_{\mathfrak{f}}^{s}(\bar{p})=\zeta^{s}(1)$,
and, 
\begin{equation}
\mathbf{0}^{\gamma(0)}=(\mathbf{0},e^{i\hat{\theta}_{0}^{\mathfrak{f}}})=\mathbf{0}^{\zeta(0)}\label{eq:RightEndProjection}
\end{equation}
since $\mathbf{0}^{\eta^{s}(0)}=\mathbf{0}^{\bar{p}}=\mathbf{0}^{p_{0}}$
and $\zeta^{s}(0)=\varphi_{\mathfrak{f}}^{s}(p_{t=0})=p_{0}$. It
is straightforward to see that $\mathbf{0}^{\eta^{s}(s')}=(\mathbf{0},e^{i\hat{\theta}_{s'\cdot s}^{\mathfrak{f}}})$.
On the other hand, there is some continuous function $t\mapsto\vartheta_{t}^{s}$
such that $\mathbf{0}^{\zeta^{s}(t)}=(\mathbf{0},e^{i\vartheta_{t}^{s}})$.
From (\ref{eq:LeftEndProjection}) and (\ref{eq:RightEndProjection})
we know that $e^{i\hat{\theta}_{s}^{\mathfrak{f}}}=e^{i\vartheta_{t=1}^{s}}$
and $e^{i\hat{\theta}_{0}^{\mathfrak{f}}}=e^{i\vartheta_{t=0}^{s}}$,
and then there is some integer $k_{s}\in\mathbb{Z}$ such that
\[
\hat{\theta}_{s}^{\mathfrak{f}}-\hat{\theta}_{0}^{\mathfrak{f}}=\vartheta_{t=1}^{s}-\vartheta_{t=0}^{s}+2k_{s}\pi.
\]
We shall show that $k_{s}=0$. To see this, note that the above equation
holds for all $s\in[0,\bar{\epsilon}_{\mathfrak{f}}]$, while both
the differences $\Delta_{s}^{\theta}:=\hat{\theta}_{s}^{\mathfrak{f}}-\hat{\theta}_{0}^{\mathfrak{f}}$
and $\Delta_{s}^{\vartheta}:=\vartheta_{1}^{s}-\vartheta_{0}^{s}$
continuously depend on $s$. The continuity of $s\mapsto\Delta_{s}^{\theta}$
follows directly from its definition and the continuity of $s\mapsto\hat{\theta}_{s}^{\mathfrak{f}}$.
For the continuity of $s\mapsto\Delta_{s}^{\vartheta}$, note that
$(s,t)\mapsto e^{i\vartheta_{t}^{s}}$ is continuous and it factors
through $\mathbb{R}$ by
\[
(s,t)\mapsto\vartheta_{t}^{s}\xmapsto{\exp}e^{i\vartheta_{t}^{s}}\in S^{1}.
\]
Since $\vartheta\xmapsto{\exp}e^{i\vartheta}$ is a quotient map (from
$\mathbb{R}$ to $S^{1}$), the continuity of $(s,t)\mapsto e^{i\vartheta_{t}^{s}}$
implies the continuity of $(s,t)\mapsto\vartheta_{t}^{s}$, and hence
$s\mapsto\Delta_{s}^{\vartheta}$ is continuous. Consequently, $k_{s}=\frac{\Delta_{s}^{\theta}-\Delta_{s}^{\vartheta}}{2\pi}$
also continuously depends on $s$ and thus it is a constant. At $s=0$,
we have $\zeta^{0}(t)=p_{t}$ and hence
\[
\mathbf{0}^{\zeta^{0}(t)}=\mathbf{0}^{p_{t}}\equiv\mathbf{0}^{\bar{p}}(\text{also}=p_{0}).
\]
As a result, $\Delta_{s=0}^{\vartheta}=0$ and
\[
k_{s}\equiv k_{0}=\Delta_{s=0}^{\theta}-\Delta_{s=0}^{\vartheta}=0-0.
\]

Therefore, we simply need to study the difference $\Delta_{s}^{\theta}=\Delta_{s}^{\vartheta}$
through the path $t\mapsto\zeta^{s}(t)$. To this end, we take $\bar{\mathbf{z}}^{s}:=\mathfrak{p}\circ\varphi_{\mathfrak{f}}^{s}(\bar{p})$
and $\mathbf{z}_{t}^{s}:=t\bar{\mathbf{z}}^{s}$, and then look at
the path $\zeta^{s}$ through the parallel parametrization (see Definition
\ref{def:ParallelParametrization})
\[
\psi_{\bar{\mathbf{z}}^{s}}:S^{1}\times[0,1]\mapsto\mathbf{z}_{[0,1]}^{s}\times S^{1}.
\]
As is shown in Corollary \ref{cor:BasicLimit}, the curve $t\mapsto\varphi_{\mathfrak{f}}^{s}(p_{t})$
lies entirely on the sheet $\mathbf{z}_{[0,1]}^{s}\times S^{1}$,
and hence there are continuous functions $t\mapsto(\vartheta_{t},l_{t})$
such that $\zeta^{s}(t)=\varphi_{\mathfrak{f}}^{s}(p_{t})=\psi_{\bar{\mathbf{z}}^{s}}^{l_{t}}(e^{i\vartheta_{t}})$.
For any fixed $s$, the $\vartheta_{t}$ here is simply the $\vartheta_{t}^{s}$
above (and hence $\hat{\theta}_{s}^{\mathfrak{f}}-\hat{\theta}_{0}^{\mathfrak{f}}=\vartheta_{1}-\vartheta_{0}$).
For this proof we also need to show the smoothness of $t\mapsto(\vartheta_{t},l_{t})$
for $t\in(0,1]$.

Since $\varphi_{\mathfrak{f}}^{s}$ is smooth on $\mathcal{U}_{*}^{\mathfrak{f}}$
and $p_{t}\in\mathcal{U}_{*}^{\mathfrak{f}}$ for all $t\in(0,1]$,
as the composite of two smooth maps, 
\[
\zeta^{s}:t\mapsto p_{t}\mapsto\varphi_{\mathfrak{f}}^{s}(p_{t})=\psi_{\bar{\mathbf{z}}^{s}}^{l_{t}}(e^{i\vartheta_{t}})
\]
is a smooth function of $t\in(0,1]$. Note that $(\vartheta,l)\mapsto(e^{i\vartheta},l)$
is a local diffeomorphism and the map $\zeta^{s}$ admits the factorization
\[
\zeta^{s}:t\mapsto(\vartheta_{t},l_{t})\mapsto\psi_{\bar{\mathbf{z}}^{s}}^{l_{t}}(e^{i\vartheta_{t}}),
\]
the smoothness of $\zeta^{s}$ on $(0,1]$ then implies the smoothness
of 
\[
(0,1]\ni t\mapsto(\vartheta_{t},l_{t})\in\mathbb{R}^{2}.
\]

It follows from Definition \ref{def:ParallelParametrization} that
\[
\frac{d}{dl}\psi_{\bar{\mathbf{z}}^{s}}^{l}(e^{i\vartheta})=\frac{1}{l}\bar{\partial}_{r}\big|_{\psi_{\bar{\mathbf{z}}^{s}}^{l}(e^{i\vartheta})},\ \forall l\in(0,1],
\]
and, there is some positive function $\mu$ in $\vartheta,l$ such
that
\[
\frac{d}{d\vartheta}\psi_{\bar{\mathbf{z}}^{s}}^{l}(e^{i\vartheta})=\mu_{\vartheta,l}\cdot\frac{\partial}{\partial\theta}\big|_{\psi_{\bar{\mathbf{z}}^{s}}^{l}(e^{i\vartheta})}.
\]
Exploiting the smoothness of $t\mapsto(\vartheta_{t},l_{t})$ on $(0,1]$
and applying the chain rule to $\frac{d}{dt}\varphi^{s}(p_{t})=\frac{d}{dt}\psi_{\bar{\mathbf{z}}^{s}}^{l_{t}}(e^{i\vartheta_{t}})$,
we get
\[
\kappa_{s,t}^{\mathfrak{f}}\frac{\partial}{\partial\theta}+\nu_{s,t}\bar{\partial}_{r}=\frac{d}{dt}\psi_{\bar{\mathbf{z}}^{s}}^{l_{t}}(e^{i\vartheta_{t}})=\mu_{\vartheta,l}\frac{d\vartheta}{dt}\frac{\partial}{\partial\theta}+\frac{1}{l}\frac{dl}{dt}\bar{\partial}_{r}.
\]
Comparing both sides of the above equation and applying Lemma \ref{lem:=00005Cf-modified=00005B=00005CX=00005D}
yields
\[
\frac{d\vartheta}{dt}=\frac{\kappa_{s,t}^{\mathfrak{f}}}{\mu_{\vartheta,l}}>0\ \ \forall t\in(0,1].
\]
Now that $t\mapsto\vartheta_{t}$ is continuous on $[0,1]$ and smooth
on $(0,1]$, the proof is concluded with the calculation below:
\[
\hat{\theta}_{s}^{\mathfrak{f}}-\hat{\theta}_{0}^{\mathfrak{f}}=\vartheta_{1}-\vartheta_{0}=\int_{0}^{1}\frac{d\vartheta}{dt}dt>0.
\]
\end{proof}
\begin{proof}[Proof of Theorem \ref{thm:=00005B=00005CX=00005DPositiveHolonomy}]
We already know that the function $s\mapsto h_{s}$ in (\ref{eq:TechnicalReparametrization})
is increasing and bijective on $[0,\infty)$. Consequently, it has
an inverse $h^{-1}$, which is also increasing on $[0,\infty)$. Check
that
\[
(\mathbf{0},e^{i\hat{\theta}_{s}})=\mathbf{0}^{\varphi_{\mathcal{X}}^{s}(p)}=\mathbf{0}^{\varphi_{\mathfrak{f}}^{h^{-1}(s)}(p)}=(\mathbf{0},e^{i\hat{\theta}_{h^{-1}(s)}^{\mathfrak{f}}})
\]
and hence $\hat{\theta}_{s}=\hat{\theta}_{h^{-1}(s)}^{\mathfrak{f}}+2k\pi$
for some constant $k\in\mathbb{Z}$. According to Theorem \ref{thm:=00005Bf=00005CX=00005DHolonomy},
for any $\bar{p}\in\mathcal{U}_{*}^{\mathfrak{f}}$ and $\mathbf{0}^{\varphi_{\mathfrak{f}}^{s}(\bar{p})}=(\mathbf{0},e^{i\hat{\theta}_{s}^{\mathfrak{f}}})$
$s\mapsto\hat{\theta}_{s}^{\mathfrak{f}}$ is increasing. As a result,
the composite (of two increasing functions) $s\mapsto h^{-1}(s)\mapsto\hat{\theta}_{h^{-1}(s)}^{\mathfrak{f}}$
also increases on $[0,\infty)$, concluding the proof for Theorem
\ref{thm:=00005B=00005CX=00005DPositiveHolonomy}.
\end{proof}

\subsection{Specific Construction of $\bar{a},\bar{b}$}

In the previous discussion (Subsection \ref{subsec:MonotoneCircling=00005B=00005CX=00005D}),
we have shown that if the conditions $|\bar{a}|>0,\bar{b}\geq0$ on
$\mathcal{U}$, and, $d\beta(\mathcal{X},\bar{\partial}_{r})>0$ on
$\mathcal{U}_{*}$ are satisfied, the dynamics generated by the vector
field $\mathcal{X}$ demonstrates certain desirable properties, such
as the monotonic circling in Theorem \ref{thm:=00005B=00005CX=00005DPositiveHolonomy}.
Now we shall show that, when the constraint is completely non-holonomic,
it is indeed possible to construct $\bar{a},\bar{b}$ to meet these
conditions. To this end, it suffices to figure out the construction
of $\bar{a}$ for an arbitrary $\bar{b}$ such that both $\inf_{\mathcal{U}}|\bar{a}|>0$
and $d\beta(\mathcal{X},\bar{\partial}_{r})>0$ on $\mathcal{U}_{*}$
hold.

Combining (\ref{eq:=00005B=00005CX=00005DnonHolonomic-PathFollow=00005B3D=00005D})
and Lemma \ref{lem:RotationComponent} yields
\[
\mathcal{X}=\bar{a}\lambda\cdot\bar{\partial}_{\phi}+\bar{b}\cdot\mathcal{V}_{\beta}\times\big(\frac{\partial}{\partial\theta}\times\nabla\mathfrak{H}\big).
\]
Check that
\begin{equation}
\begin{aligned}d\beta(\mathcal{X},\bar{\partial}_{r})= & \bar{a}\lambda\cdot d\beta\big(\bar{\partial}_{\phi},\bar{\partial}_{r}\big)+\bar{b}\bigg(\big(\mathcal{V}_{\beta}\cdot\nabla\mathfrak{H}\big)\cdot d\beta\big(\frac{\partial}{\partial\theta},\bar{\partial}_{r}\big)-d\beta\big(\nabla\mathfrak{H},\bar{\partial}_{r}\big)\bigg)\\
= & \bar{a}\lambda\cdot\beta\wedge d\beta\big(\frac{\partial}{\partial\theta},\bar{\partial}_{\phi},\bar{\partial}_{r}\big)+\bar{b}\beta\wedge d\beta\big(\nabla\mathfrak{H},\frac{\partial}{\partial\theta},\bar{\partial}_{r}\big)
\end{aligned}
,\label{eq:d=00005Cbeta(=00005CX,=00005Cpartial_r)}
\end{equation}
for which the relation $\beta\wedge d\beta\big(\frac{\partial}{\partial\theta},\bar{\partial}_{\phi},\bar{\partial}_{r}\big)=d\beta\big(\bar{\partial}_{\phi},\bar{\partial}_{r}\big)$
is applied. According to Lemma \ref{lem:RotationComponent}, $\lambda\neq0$
everywhere on $\mathcal{U}$. Also, when the constraint is completely
non-holonomic, $\beta\wedge d\beta\neq0$ everywhere on $\mathcal{U}$,
as a result of which
\[
\beta\big([\bar{\partial}_{r},\bar{\partial}_{\phi}]\big)=d\beta\big(\bar{\partial}_{\phi},\bar{\partial}_{r}\big)\neq0\ \text{ on }\mathcal{U}_{*}.
\]
So, given any weight function $\bar{b}$ on $\mathcal{U}$, at least
it is possible to construct a weight function $\bar{a}$ on $\mathcal{U}_{*}$
such that 
\begin{equation}
|\bar{a}|\geq\frac{|\bar{b}|\cdot\bigg|\beta\wedge d\beta\big(\nabla\mathfrak{H},\frac{\partial}{\partial\theta},\bar{\partial}_{r}\big)\bigg|}{|\lambda|\cdot\bigg|\beta\wedge d\beta\big(\frac{\partial}{\partial\theta},\bar{\partial}_{\phi},\bar{\partial}_{r}\big)\bigg|}+\varepsilon_{0},\label{eq:=00005Blowerbound=00005DWeightFunction-=00005Ca}
\end{equation}
with an arbitrary $\varepsilon_{0}>0$, which will then imply $d\beta(\mathcal{X},\bar{\partial}_{r})\neq0$
on $\mathcal{U}_{*}$. However, to obtain a vector field $\mathcal{X}$
defined on the whole $\mathcal{U}$ we still need $\bar{a}$ to be
a smooth function well defined on $\mathcal{U}$. Fortunately, it
turns out that the right-hand side of (\ref{eq:=00005Blowerbound=00005DWeightFunction-=00005Ca})
has an upper bound near $\mathcal{P}$, and the existence of $\bar{a}$
as a smooth function well defined on $\mathcal{U}$ satisfying (\ref{eq:=00005Blowerbound=00005DWeightFunction-=00005Ca})
is guaranteed. Since $\inf_{\mathcal{U}}|\lambda|>0$ and hence $\frac{1}{|\lambda|}=\pm\frac{1}{\lambda}$
is well defined and smooth on $\mathcal{U}$, all we need is the following
result:
\begin{lem}
\label{lem:EssentialUpperBound}If $\beta\wedge d\beta$ is nondegenerate
on $\mathcal{U}$, then
\[
\sup_{\mathcal{U}_{*}}\frac{\bigg|\beta\wedge d\beta\big(\nabla\mathfrak{H},\frac{\partial}{\partial\theta},\bar{\partial}_{r}\big)\bigg|}{\bigg|\beta\wedge d\beta\big(\frac{\partial}{\partial\theta},\bar{\partial}_{\phi},\bar{\partial}_{r}\big)\bigg|}<\infty.
\]
\end{lem}
\begin{proof}
Check that $\bar{\partial}_{r}=x\bar{\partial}_{x}+y\bar{\partial}_{y}$,
and, $\bar{\partial}_{\phi}=x\bar{\partial}_{y}-y\bar{\partial}_{x}$.
As a result, 
\[
\begin{aligned}\beta\wedge d\beta\big(\frac{\partial}{\partial\theta},\bar{\partial}_{\phi},\bar{\partial}_{r}\big)= & \beta\wedge d\beta\big(\frac{\partial}{\partial\theta},x\bar{\partial}_{x},x\bar{\partial}_{y}\big)+\beta\wedge d\beta\big(\frac{\partial}{\partial\theta},y\bar{\partial}_{y},-y\bar{\partial}_{x}\big)\\
= & r^{2}\beta\wedge d\beta\big(\frac{\partial}{\partial\theta},\bar{\partial}_{x},\bar{\partial}_{y}\big).
\end{aligned}
\]
On the other hand, we have
\[
\begin{aligned}\beta\wedge d\beta\big(\nabla\mathfrak{H},\frac{\partial}{\partial\theta},\bar{\partial}_{r}\big)= & x\cdot\beta\wedge d\beta\big(\frac{\partial}{\partial\theta},\bar{\partial}_{x},\nabla\mathfrak{H}\big)+y\cdot\beta\wedge d\beta\big(\frac{\partial}{\partial\theta},\bar{\partial}_{y},\nabla\mathfrak{H}\big)\\
= & x\cdot\alpha_{x}(\nabla\mathfrak{H})+y\cdot\alpha_{y}(\nabla\mathfrak{H})\\
= & x\cdot\langle\mathcal{V}_{\alpha_{x}},\nabla\mathfrak{H}\rangle+y\cdot\langle\mathcal{V}_{\alpha_{y}},\nabla\mathfrak{H}\rangle.
\end{aligned}
\]
where $\alpha_{x}(\cdot)=\beta\wedge d\beta\big(\frac{\partial}{\partial\theta},\bar{\partial}_{x},\cdot\big)$
and $\alpha_{y}(\cdot)=\beta\wedge d\beta\big(\frac{\partial}{\partial\theta},\bar{\partial}_{y},\cdot\big)$
are differential $1$-forms, and $\mathcal{V}_{\alpha_{x}},\mathcal{V}_{\alpha_{y}}$
are the corresponding Riesz representations. Note that $d\mathfrak{H}=2\big(x\cdot dx+y\cdot dy\big)$,
and for any smooth vector field $\mathcal{V}$ on $\mathcal{U}$,
\[
\mathcal{V}=u_{\theta}\frac{\partial}{\partial\theta}+u_{x}\frac{\partial}{\partial x}+u_{y}\frac{\partial}{\partial y}.
\]
Consequently, it holds for any $L>\sup\sqrt{u_{x}^{2}+u_{y}^{2}}$,
\[
\big|\langle\mathcal{V},\nabla\mathfrak{H}\rangle\big|=\big|d\mathfrak{H}(\mathcal{V})\big|=2\big|xu_{x}+yu_{y}\big|\leq2\cdot r\cdot L.
\]
Take $L>0$ to be large enough such that it satisfies the above inequality
for both $\mathcal{V}=\mathcal{V}_{\alpha_{x}}$ and $\mathcal{V}=\mathcal{V}_{\alpha_{y}}$,
and then
\[
\bigg|\beta\wedge d\beta\big(\nabla\mathfrak{H},\frac{\partial}{\partial\theta},\bar{\partial}_{r}\big)\bigg|\leq2\big(|x|+|y|\big)\cdot r\cdot L\leq2\sqrt{2}\cdot r^{2}\cdot L.
\]
As a result, we have on $\mathcal{U}_{*}$ the following inequality:
\[
\frac{\bigg|\beta\wedge d\beta\big(\nabla\mathfrak{H},\frac{\partial}{\partial\theta},\bar{\partial}_{r}\big)\bigg|}{\big|\beta\wedge d\beta\big(\frac{\partial}{\partial\theta},\bar{\partial}_{\phi},\bar{\partial}_{r}\big)\big|}=\frac{2\sqrt{2}\cdot L}{\bigg|\beta\wedge d\beta\big(\frac{\partial}{\partial\theta},\bar{\partial}_{x},\bar{\partial}_{y}\big)\bigg|}.
\]
The condition $\beta\wedge d\beta\neq0$ everywhere on $\mathcal{U}$
implies $\beta\wedge d\beta\big(\frac{\partial}{\partial\theta},\bar{\partial}_{x},\bar{\partial}_{y}\big)\neq0$
everywhere on $\mathcal{U}$, and by the compactness of $\mathcal{U}$
it means
\[
\inf_{\mathcal{U}}\bigg|\beta\wedge d\beta\big(\frac{\partial}{\partial\theta},\bar{\partial}_{x},\bar{\partial}_{y}\big)\bigg|>0,
\]
by which we get
\[
\sup_{\mathcal{U}_{*}}\frac{\bigg|\beta\wedge d\beta\big(\nabla\mathfrak{H},\frac{\partial}{\partial\theta},\bar{\partial}_{r}\big)\bigg|}{\big|\beta\wedge d\beta\big(\frac{\partial}{\partial\theta},\bar{\partial}_{\phi},\bar{\partial}_{r}\big)\big|}\leq\frac{2\sqrt{2}\cdot L}{\underset{\mathcal{\mathcal{U}}}{\inf}\bigg|\beta\wedge d\beta\big(\frac{\partial}{\partial\theta},\bar{\partial}_{x},\bar{\partial}_{y}\big)\bigg|}<\infty.
\]
\end{proof}
Based on Lemma \ref{lem:EssentialUpperBound}, given any smooth function
$\bar{b}$ on $\mathcal{U}$, by simply taking
\begin{equation}
\bar{a}=\pm\bigg(\frac{|\bar{b}|}{|\lambda|}\cdot\sup_{\mathcal{U}}\frac{\bigg|\beta\wedge d\beta\big(\nabla\mathfrak{H},\frac{\partial}{\partial\theta},\bar{\partial}_{r}\big)\bigg|}{\bigg|\beta\wedge d\beta\big(\frac{\partial}{\partial\theta},\bar{\partial}_{\phi},\bar{\partial}_{r}\big)\bigg|}+\varepsilon_{0}\bigg),\label{eq:=00005BSimpleCandidate=00005D-=00005Ca}
\end{equation}
we have both $|\bar{a}|>0$ and $d\beta(\mathcal{X},\bar{\partial}_{r})\neq0$
hold on $\mathcal{U}$. Properly choosing the sign on the right-hand
side will then yield both $|\bar{a}|>0$ and $d\beta(\mathcal{X},\bar{\partial}_{r})>0$.
In fact, we have the following result about the construction of $\bar{a}$:
\begin{thm}
\label{thm:WeightFunctionConstruction}Given $\bar{a},\bar{b}$ on
$\mathcal{U}$, if both the inequalities (\ref{eq:=00005Blowerbound=00005DWeightFunction-=00005Ca})
and $\bar{a}\lambda\cdot d\beta\big(\bar{\partial}_{x},\bar{\partial}_{y}\big)>0$
are satisfied, then it holds
\begin{equation}
0<\underset{\mathcal{U}_{*}}{\inf}\frac{d\beta(\mathcal{X},\bar{\partial}_{r})}{r^{2}}\leq\underset{\mathcal{U}_{*}}{\sup}\frac{d\beta(\mathcal{X},\bar{\partial}_{r})}{r^{2}}<\infty.\label{eq:EssentialBounds}
\end{equation}
When the constraint $\ker\beta$ is completely non-holonomic, for
any smooth function $\bar{b}$ on $\mathcal{U}$, $\bar{a}$ can be
constructed accordingly such that both (\ref{eq:=00005Blowerbound=00005DWeightFunction-=00005Ca})
and $\bar{a}\lambda\cdot d\beta\big(\bar{\partial}_{x},\bar{\partial}_{y}\big)>0$
are satisfied.
\end{thm}
\begin{rem}
$\bar{a}\lambda\cdot d\beta\big(\bar{\partial}_{x},\bar{\partial}_{y}\big)>0$
on $\mathcal{U}$ implies $|\bar{a}|>0$ on $\mathcal{U}$, and that
$d\beta(\mathcal{X},\bar{\partial}_{r})>0$ on $\mathcal{U}_{*}$
is simply an implication of (\ref{eq:EssentialBounds}).
\end{rem}
\begin{proof}
First of all, we show that \eqref{eq:EssentialBounds} holds whenever
both the inequalities (\ref{eq:=00005Blowerbound=00005DWeightFunction-=00005Ca})
and $\bar{a}\lambda\cdot d\beta\big(\bar{\partial}_{x},\bar{\partial}_{y}\big)>0$
are satisfied. It follows from (\ref{eq:d=00005Cbeta(=00005CX,=00005Cpartial_r)})
that
\[
\big|d\beta\big(\mathcal{X},\bar{\partial}_{r}\big)\big|\geq\big|\bar{a}\big|\cdot\big|\lambda\big|\cdot\bigg|\beta\wedge d\beta\big(\frac{\partial}{\partial\theta},\bar{\partial}_{\phi},\bar{\partial}_{r}\big)\bigg|-|\bar{b}|\cdot\bigg|\beta\wedge d\beta\big(\nabla\mathfrak{H},\frac{\partial}{\partial\theta},\bar{\partial}_{r}\big)\bigg|.
\]
For any $\bar{a}$ satisfying inequality (\ref{eq:=00005Blowerbound=00005DWeightFunction-=00005Ca}),
it holds 
\[
\big|d\beta\big(\mathcal{X},\bar{\partial}_{r}\big)\big|>\varepsilon_{0}\big|\lambda\big|\cdot\bigg|\beta\wedge d\beta\big(\frac{\partial}{\partial\theta},\bar{\partial}_{\phi},\bar{\partial}_{r}\big)\bigg|.
\]
Direct computation shows $[\bar{\partial}_{r},\bar{\partial}_{\phi}]=r^{2}[\bar{\partial}_{x},\bar{\partial}_{y}]$,
and hence
\[
d\beta\big(\bar{\partial}_{\phi},\bar{\partial}_{r}\big)=\beta\big([\bar{\partial}_{r},\bar{\partial}_{\phi}]\big)=r^{2}\beta\big([\bar{\partial}_{x},\bar{\partial}_{y}]\big)=-r^{2}d\beta\big(\bar{\partial}_{x},\bar{\partial}_{y}\big)
\]
Therefore, with the function $\bar{a}$ satisfying (\ref{eq:=00005Blowerbound=00005DWeightFunction-=00005Ca}),
we get on $\mathcal{U}$
\[
\big|d\beta\big(\mathcal{X},\bar{\partial}_{r}\big)\big|>\varepsilon_{0}\big|\lambda\big|\cdot\big|d\beta\big(\bar{\partial}_{\phi},\bar{\partial}_{r}\big)\big|=\varepsilon_{0}r^{2}\big|\lambda\big|\cdot\big|d\beta\big(\bar{\partial}_{x},\bar{\partial}_{y}\big)\big|,
\]
and hence on $\mathcal{U}_{*}$, it holds
\[
\frac{\big|d\beta\big(\mathcal{X},\bar{\partial}_{r}\big)\big|}{r^{2}}>\varepsilon_{0}\cdot\inf_{\mathcal{U}}\bigg|\lambda\cdot d\beta\big(\bar{\partial}_{x},\bar{\partial}_{y}\big)\bigg|>0.
\]
On the other hand, also due to Lemma \ref{lem:EssentialUpperBound},
there exists $K>0$ such that
\[
|\bar{b}|\cdot\bigg|\beta\wedge d\beta\big(\nabla\mathfrak{H},\frac{\partial}{\partial\theta},\bar{\partial}_{r}\big)\bigg|<K\cdot\big|\lambda\big|\cdot\bigg|\beta\wedge d\beta\big(\frac{\partial}{\partial\theta},\bar{\partial}_{\phi},\bar{\partial}_{r}\big)\bigg|.
\]
Then, with $\bar{a}$ constructed above, it holds on $\mathcal{U}$
\[
\begin{aligned}\big|d\beta\big(\mathcal{X},\bar{\partial}_{r}\big)\big|\leq & \big(\big|\bar{a}\big|+K\big)\big|\lambda\big|\cdot\bigg|\beta\wedge d\beta\big(\frac{\partial}{\partial\theta},\bar{\partial}_{\phi},\bar{\partial}_{r}\big)\bigg|\\
= & \big(\big|\bar{a}\big|+K\big)\big|\lambda\big|\cdot\bigg|d\beta\big(\bar{\partial}_{\phi},\bar{\partial}_{r}\big)\bigg|\\
= & r^{2}\cdot\big(\big|\bar{a}\big|+K\big)\big|\lambda\big|\cdot\bigg|d\beta\big(\bar{\partial}_{x},\bar{\partial}_{y}\big)\bigg|,
\end{aligned}
\]
and hence on $\mathcal{U}_{*}$ it holds
\[
\frac{\big|d\beta\big(\mathcal{X},\bar{\partial}_{r}\big)\big|}{r^{2}}\leq\sup_{\mathcal{U}}\big(\big|\bar{a}\big|+K\big)\big|\lambda\big|\cdot\bigg|d\beta\big(\bar{\partial}_{x},\bar{\partial}_{y}\big)\bigg|<\infty.
\]

To complete the proof, it remains to show that, when the constraint
is completely non-holonomic, there exists (at least) a smooth function
$\bar{a}$ such that both (\ref{eq:=00005Blowerbound=00005DWeightFunction-=00005Ca})
and $\bar{a}\lambda\cdot d\beta\big(\bar{\partial}_{x},\bar{\partial}_{y}\big)>0$
hold simultaneously. Note that the constraint $\beta=0$ is completely
non-holonomic if and only if $\beta\wedge d\beta\neq0$ everywhere
on $\mathcal{U}$. Thanks to Lemma \ref{lem:EssentialUpperBound},
we can simply take $\bar{a}$ to be the function in (\ref{eq:=00005BSimpleCandidate=00005D-=00005Ca})
with the sign on the right-hand side chosen to be the same as the sign
of $\lambda\cdot d\beta\big(\bar{\partial}_{x},\bar{\partial}_{y}\big)$
on $\mathcal{U}$, and then both (\ref{eq:=00005Blowerbound=00005DWeightFunction-=00005Ca})
and $\bar{a}\lambda\cdot d\beta\big(\bar{\partial}_{x},\bar{\partial}_{y}\big)>0$
are fulfilled.
\end{proof}
While the above results holds with the existence of $\mathcal{U}^{\mathfrak{f}}$
in general for any $\mathcal{X}$ with $\bar{a},\bar{b}$ in Theorem
\ref{thm:WeightFunctionConstruction}, we may construct $\bar{a},\bar{b}$
to satisfy the additional condition (\ref{eq:ExtraCondition-for-=00005Cf})
on $\mathcal{U}$, which then leads to $\mathcal{U}^{\mathfrak{f}}=\mathcal{U}$
and $\mathcal{U}_{*}^{\mathfrak{f}}=\mathcal{U}_{*}$. This is stated
as the following corollary:
\begin{cor}
\label{cor:WeightFun-for-Control} Let $\mathcal{X}$ be a vector
field in (\ref{eq:=00005B=00005CX=00005DnonHolonomic-PathFollow=00005B3D=00005D})
such that $|\bar{a}|>0,\bar{b}\geq0$ on $\mathcal{U}$, $\underset{\mathcal{U}_{*}}{\sup}\frac{\bar{b}}{r^{2}}<\infty$
and $d\beta(\mathcal{X},\bar{\partial}_{r})>0$ on $\mathcal{U}_{*}$.
If the extra condition (\ref{eq:ExtraCondition-for-=00005Cf}) also
holds on $\mathcal{U}$, then $\mathfrak{f}$ and $\mathcal{X}_{\mathfrak{f}}$
are well defined on $\mathcal{U}$, and, for any $\bar{p}\in\mathcal{U}_{*}$,
any continuous function $\hat{\theta}_{s}^{\mathfrak{f}}$ with $\mathbf{0}^{\varphi_{\mathfrak{f}}^{s}(\bar{p})}=(\mathbf{0},e^{i\hat{\theta}_{s}^{\mathfrak{f}}})$
increases in $s$ on $[0,\infty)$.
\end{cor}
\begin{proof}
Construct $\bar{a},\bar{b}$ on $\mathcal{U}$ as in Theorem \ref{thm:WeightFunctionConstruction}
with the extra condition (\ref{eq:ExtraCondition-for-=00005Cf}),
which is possible, for example, by simply taking
\[
\bar{a}=\pm\bigg(\frac{|\bar{b}|}{|\lambda|}\cdot\sup_{\mathcal{U}}\frac{\bigg|\beta\wedge d\beta\big(\nabla\mathfrak{H},\frac{\partial}{\partial\theta},\bar{\partial}_{r}\big)\bigg|}{\bigg|\beta\wedge d\beta\big(\frac{\partial}{\partial\theta},\bar{\partial}_{\phi},\bar{\partial}_{r}\big)\bigg|}+1+\frac{\bar{b}\cdot\big|d\phi(\nabla\mathfrak{H})\big|}{\underset{\mathcal{\mathcal{U}}}{\inf}|\lambda|}\bigg).
\]
The sign $\pm$ on the right-hand side is to be taken so that $\bar{a}\lambda\cdot d\beta\big(\bar{\partial}_{x},\bar{\partial}_{y}\big)>0$.
\end{proof}
Now we shall discuss the conditions required in the main result, Theorem
\ref{thm:=00005BMainResult=00005D}. Recall from Lemma \ref{lem:RotationComponent}
that the function $\lambda$ is determined by the relation $\bar{\Omega}=\lambda\Omega_{\mathfrak{e}}$,
and hence
\[
\lambda=\frac{\bar{\Omega}\big(\frac{\partial}{\partial\theta},\bar{\partial}_{x},\bar{\partial}_{y}\big)}{\Omega_{\mathfrak{e}}\big(\frac{\partial}{\partial\theta},\bar{\partial}_{x},\bar{\partial}_{y}\big)}=\frac{1}{\Omega_{\mathfrak{e}}\big(\frac{\partial}{\partial\theta},\bar{\partial}_{x},\bar{\partial}_{y}\big)}.
\]
Meanwhile, $d\beta\big(\bar{\partial}_{x},\bar{\partial}_{y}\big)=\beta\wedge d\beta\big(\frac{\partial}{\partial\theta},\bar{\partial}_{x},\bar{\partial}_{y}\big)$,
and therefore
\begin{equation}
\bar{a}\lambda\cdot d\beta\big(\bar{\partial}_{x},\bar{\partial}_{y}\big)=\bar{a}\frac{\beta\wedge d\beta\big(\frac{\partial}{\partial\theta},\bar{\partial}_{x},\bar{\partial}_{y}\big)}{\Omega_{\mathfrak{e}}\big(\frac{\partial}{\partial\theta},\bar{\partial}_{x},\bar{\partial}_{y}\big)}=\bar{a}\lambda_{\beta}.\label{eq:EssentialCondition=00005BMainTheorem=00005D}
\end{equation}
That is, the condition $\bar{a}\lambda\cdot d\beta\big(\bar{\partial}_{x},\bar{\partial}_{y}\big)>0$
in Theorem \ref{thm:WeightFunctionConstruction} is equivalent to
the condition $\bar{a}\lambda_{\beta}>0$ in Theorem \ref{thm:=00005BMainResult=00005D}.

Note that the conditions $\bar{a}\lambda_{\beta}>0$ and $\underset{\mathcal{U}_{*}}{\sup}\frac{\bar{b}}{r^{2}}<\infty$
in Theorem \ref{thm:=00005BMainResult=00005D} imply $\underset{\mathcal{U}}{\inf}|\bar{a}|>0$
and $\underset{r\rightarrow0}{\lim}\bar{b}=0$. Combined with Lemma
\ref{lem:EssentialUpperBound}, it then implies the inequality (\ref{eq:=00005Blowerbound=00005DWeightFunction-=00005Ca})
to hold on a sufficiently small neighborhood $\mathcal{O}$ of $\mathcal{P}$.
Applying Theorem \ref{thm:WeightFunctionConstruction} we obtain the
following result:
\begin{thm}
\label{thm:EssentialConditions}Suppose that $\underset{\mathcal{U}}{\inf}\;\bar{a}\lambda_{\beta}>0$,
$\underset{\mathcal{U}}{\min}\;\bar{b}\geq0$ and $\underset{\mathcal{U}_{*}}{\sup} \;\frac{\bar{b}}{r^{2}}<\infty$.
Then, there is some neighborhood $\mathcal{O}=\mathcal{B}_{\delta_{o}}\times S^{1}$
of $\mathcal{P}$ in $\mathcal{U}$, on which the inequalities $\bar{b}\geq0$,
(\ref{eq:=00005Blowerbound=00005DWeightFunction-=00005Ca}) and $\bar{a}\lambda\cdot d\beta\big(\bar{\partial}_{x},\bar{\partial}_{y}\big)>0$
hold. As a result, the inequalities below hold with $\mathcal{O}_{*}=\mathcal{O}\setminus\mathcal{P}$:
\[
0<\underset{\mathcal{O}_{*}}{\inf}\frac{d\beta(\mathcal{X},\bar{\partial}_{r})}{r^{2}}\leq\underset{\mathcal{O}_{*}}{\sup}\frac{d\beta(\mathcal{X},\bar{\partial}_{r})}{r^{2}}<\infty.
\]
Moreover, there exists $\mathcal{U}^{\mathfrak{f}}=\mathcal{B}_{\delta_{\mathfrak{f}}}\times S^{1}$
with $\delta_{\mathfrak{f}}<\delta_{o}$, such that for any $\bar{p}\in\mathcal{U}_{*}^{\mathfrak{f}}$
and $\mathbf{0}^{\varphi_{\mathcal{X}}^{s}(\bar{p})}=(\mathbf{0},e^{i\hat{\theta}_{s}})$,
the continuous function $s\mapsto\hat{\theta}_{s}$ increases on $[0,\infty)$.
\end{thm}

\subsection{Proof for Theorem \ref{thm:=00005BMainResult=00005D}}

With $\bar{b}\geq0$ on $\mathcal{U}$ and $\bar{b}>0$ on $\mathcal{U}_{*}$,
we know that each orbit of $\mathcal{X}$ will eventually converge
to $\mathcal{P}$. Therefore, for proving the main result (Theorem
\ref{thm:=00005BMainResult=00005D}), it remains to prove that, with
the construction of $\bar{a},\bar{b}$ therein, the orbit $\eta_{s}^{\mathcal{X}}=\varphi_{\mathcal{X}}^{s}(\bar{p})$
fulfills the circling (\ref{eq:Circling}) requirement in Problem
\ref{prob:=00005BGeneralMainProblem=00005D3D-nonHolonomic-PathFollowing}.
Since every orbit of $\varphi_{\mathcal{X}}$ will eventually enter
and stay in $\mathcal{U}^{\mathfrak{f}}=\mathcal{B}_{\delta_{\mathfrak{f}}}\times S^{1}$,
we may simply assume $\bar{p}\in\mathcal{U}_{*}^{\mathfrak{f}}$ and
focus on the (restricted) system on $\mathcal{U}^{\mathfrak{f}}$.
The orbits of $\mathcal{X}_{\mathfrak{f}}$ and $\mathcal{X}$ in
$\mathcal{U}_{*}^{\mathfrak{f}}$ are related by (\ref{eq:TechnicalReparametrization}),
and hence we only need to show (\ref{eq:Circling}) for the orbit
$\eta_{s}^{\mathfrak{f}}=\varphi_{\mathfrak{f}}^{s}(\bar{p})$. Resorting
to Lemma \ref{lem:CirclingByParallelProjection}, it suffices to show
for the function $s\mapsto\hat{\theta}_{s}^{\mathfrak{f}}$ of $\Theta(\eta_{s}^{\mathfrak{f}})=(\mathbf{0},e^{i\hat{\theta}_{s}^{\mathfrak{f}}})$
the following limit:
\[
\lim_{s\rightarrow\infty}\hat{\theta}_{s}^{\mathfrak{f}}=\infty.
\]
Under the conditions in Theorem \ref{thm:=00005BMainResult=00005D},
by taking $\mathcal{U}^{\mathfrak{f}}$ to be the same as that in
the statement of Theorem \ref{thm:EssentialConditions}, we have the
following inequality 
\[
\underset{\mathcal{U}_{*}^{\mathfrak{f}}}{\inf}\frac{d\beta(\mathcal{X},\bar{\partial}_{r})}{r^{2}}>0,
\]
and then the proof for Theorem \ref{thm:=00005BMainResult=00005D}
is completed by showing the following proposition:
\begin{prop}
\label{finalprop:=00005BCircling=00005D}If $\underset{\mathcal{U_{*}^{\mathfrak{f}}}}{\sup}\frac{\bar{b}}{r^{2}}>0$
and $\underset{\mathcal{U}_{*}^{\mathfrak{f}}}{\inf}\frac{d\beta(\mathcal{X},\bar{\partial}_{r})}{r^{2}}>0$,
then $\underset{s\rightarrow\infty}{\lim}\hat{\theta}_{s}^{\mathfrak{f}}=\infty.$
\end{prop}

\subsubsection*{Proof of Proposition \ref{finalprop:=00005BCircling=00005D}}

At each time moment $s>0$, the curve $t\mapsto\varphi_{\mathfrak{f}}^{s}(p_{t})$
has the parallel projection
\begin{equation}
\Theta\circ\varphi_{\mathfrak{f}}^{s}(p_{t})=(\mathbf{0},e^{i\bar{\vartheta}_{t}^{s}}),\label{eq:FundamentalRelation}
\end{equation}
and then $\Theta\circ\varphi_{\mathfrak{f}}^{s}(p_{t})\bigg|_{t=1}=\Theta\circ\varphi_{\mathfrak{f}}^{s}(\bar{p})=(\mathbf{0},e^{i\hat{\theta}_{s}^{\mathfrak{f}}})$,
i.e., $\bar{\vartheta}_{t=1}^{s}=\hat{\theta}_{s}^{\mathfrak{f}}$.
Meanwhile, since $\varphi_{\mathfrak{f}}^{s}(p_{0})\equiv p_{0}$
with $\mathbf{0}^{p_{0}}=\mathbf{0}^{\bar{p}}=(\mathbf{0},e^{i\hat{\theta}})$,
it holds $\bar{\vartheta}_{t=0}^{s}=\hat{\theta}_{s=0}^{\mathfrak{f}}$.
As a result, it holds
\[
\hat{\theta}_{s}^{\mathfrak{f}}-\hat{\theta}_{0}^{\mathfrak{f}}=\bar{\vartheta}_{t=1}^{s}-\bar{\vartheta}_{t=0}^{s}.
\]
Therefore, for showing $\underset{s\rightarrow\infty}{\lim}\hat{\theta}_{s}^{\mathfrak{f}}=\infty$,
it suffices to prove that
\[
\underset{s\rightarrow\infty}{\lim}\bigg(\bar{\vartheta}_{t=1}^{s}-\bar{\vartheta}_{t=0}^{s}\bigg)=\infty.
\]
To this end, we need a lower bound estimation similar to that in (\ref{eq:MinimumCircling}). 

From \eqref{eq:FundamentalRelation} it holds that
\[
\Theta_{*}(\partial_{\bar{\mathbf{z}}}^{\mathfrak{f}})=\Theta_{*}\circ\varphi_{\mathfrak{f}*}^{s}(\dot{p}_{t})=\frac{d\bar{\vartheta}_{t}^{s}}{dt}\cdot\frac{\partial}{\partial\theta},
\]
and hence
\[
\bar{\vartheta}_{t=1}^{s}-\bar{\vartheta}_{t=0}^{s}=\int_{0}^{1}d\theta\circ\Theta_{*}(\partial_{\bar{\mathbf{z}}}^{\mathfrak{f}})dt.
\]
Since $\partial_{\bar{\mathbf{z}}}^{\mathfrak{f}}\big|_{s,t}=\varphi_{\mathfrak{f},*}^{s}(\dot{p}_{t})=\kappa_{s,t}^{\mathfrak{f}}\frac{\partial}{\partial\theta}+\nu_{s,t}\bar{\partial}_{r}$,
by (\ref{eq:RadiusAnnihilation*-=00005CTheta}) and (\ref{eq:FiberStretch*-=00005CTheta})
it holds that
\[
\Theta_{*}(\partial_{\bar{\mathbf{z}}}^{\mathfrak{f}})=\kappa_{s,t}^{\mathfrak{f}}\cdot\hat{\mu}\frac{\partial}{\partial\theta},
\]
and then
\[
\bar{\vartheta}_{t=1}^{s}-\bar{\vartheta}_{t=0}^{s}=\int_{0}^{1}\kappa_{s,t}^{\mathfrak{f}}\cdot\hat{\mu}_{s,t}dt.
\]
For those $s\in(0,\bar{\epsilon}_{\mathfrak{f}}]$ with ``the technical
espsilon'' $\bar{\epsilon}_{\mathfrak{f}}>0$ in \ref{rem:Technical=00005Cepsilon},
$\kappa_{s,t}^{\mathfrak{f}}>0$, and hence
\[
\bar{\vartheta}_{t=1}^{s}-\bar{\vartheta}_{t=0}^{s}\geq\inf_{\mathcal{U}}\hat{\mu}\cdot\int_{0}^{1}\kappa_{s,t}^{\mathfrak{f}}dt.
\]
Due to the compactness of $\mathcal{U}$, we have
\[
0<\inf_{\mathcal{U}}\hat{\mu}\leq\hat{\mu}\leq\sup_{\mathcal{U}}\hat{\mu}<\infty,
\]
and hence for proving the limit in Proposition \ref{finalprop:=00005BCircling=00005D},
it suffices to obtain a lower bound for $\int_{0}^{1}\kappa_{s,t}^{\mathfrak{f}}dt$.

Note that $\mathcal{X}_{\mathfrak{f}}$ can also be expressed by \eqref{eq:=00005B=00005CX=00005DnonHolonomic-PathFollow=00005B3D=00005D},
i.e., 
\[
\mathcal{X}_{\mathfrak{f}}=\bar{a}_{\mathfrak{f}}\cdot\mathcal{V}_{\beta}\times\nabla\mathfrak{H}+\bar{b}_{\mathfrak{f}}\cdot\mathcal{V}_{\beta}\times\big(\frac{\partial}{\partial\theta}\times\nabla\mathfrak{H}\big),
\]
with the weight functions $\bar{a}_{\mathfrak{f}}=\mathfrak{f}\bar{a}$
and $\bar{b}_{\mathfrak{f}}=\mathfrak{f}\bar{b}$. Therefore, Lemma
\ref{lem:TechnicalLemma-2} also applies to $\mathcal{X}_{\mathfrak{f}}$
and $\partial_{\bar{\mathbf{z}}}^{\mathfrak{f}}$, in which $\xi_{s,t}^{\mathfrak{f}}=\nu_{s,t}\bar{\partial}_{r}\big|_{\varphi_{\mathfrak{f}}^{s}(p_{t})}$
with
\[
d\beta(\mathcal{X}_{\mathfrak{f}}\big|_{0,r},\xi_{0,t}^{\mathfrak{f}})=\mathfrak{f}\cdot d\beta(\mathcal{X},\frac{1}{t}\bar{\partial}_{r})>0,\ \forall t\in(0,1].
\]
As a result, there exists $\bar{\epsilon}_{\mathfrak{f}}>0$ such
that for each $(s,t)\in[0,\bar{\epsilon}_{\mathfrak{f}}]\times(0,1]$.
\[
d\beta(\mathcal{X}_{\mathfrak{f}},\xi_{s,t})=\mathfrak{f}\nu_{s,t}\cdot d\beta(\mathcal{X},\bar{\partial}_{r})>0.
\]
Since $d\beta(\mathcal{X},\bar{\partial}_{r})>0$ on $\mathcal{U}_{*}$,
and then $\nu_{s,t}>0$ for $(s,t)\in[0,\bar{\epsilon}_{\mathfrak{f}}]\times(0,1]$.
In fact, we have the following result:
\begin{lem}
\label{lem:horizontal-integral}It holds $\nu_{s,t}>0$ for $(s,t)\in[0,\bar{\epsilon}_{\mathfrak{f}}]\times(0,1]$,
and, 
\[
\int_{0}^{1}2\nu_{s,t}\cdot r^{2}\big|_{\varphi_{\mathfrak{f}}^{s}(p_{t})}dt=r^{2}\big|_{\varphi_{\mathfrak{f}}^{s}(\bar{p})}.
\]
 
\end{lem}
\begin{proof}
With $r^{2}=x^{2}+y^{2}$ ($=\mathfrak{H}$) and $\partial_{\bar{\mathbf{z}}}^{\mathfrak{f}}\big|_{s,t}=\kappa_{s,t}^{\mathfrak{f}}\frac{\partial}{\partial\theta}+\nu_{s,t}\bar{\partial}_{r}$,
it holds
\[
dr^{2}\big(\partial_{\bar{\mathbf{z}}}^{\mathfrak{f}}\big)=\nu_{s,t}dr^{2}\big(\bar{\partial}_{r}\big)=2\nu_{s,t}r^{2}\big|_{\varphi_{\mathfrak{f}}^{s}(p_{t})}.
\]
On the other hand, since $\partial_{\bar{\mathbf{z}}}^{\mathfrak{f}}=\frac{d}{dt}\varphi_{\mathfrak{f}}^{s}(p_{t})$,
we have
\[
dr^{2}\big(\partial_{\bar{\mathbf{z}}}^{\mathfrak{f}}\big)=\frac{d}{dt}r^{2}\circ\varphi_{\mathfrak{f}}^{s}(p_{t})
\]
As a result, 
\[
\int_{0}^{1}2\nu_{s,t}\cdot r^{2}\big|_{\varphi_{\mathfrak{f}}^{s}(p_{t})}dt=r^{2}\circ\varphi_{\mathfrak{f}}^{s}(p_{t})\bigg|_{t=0}^{t=1}=r^{2}\big|_{\varphi_{\mathfrak{f}}^{s}(\bar{p})}.
\]
\end{proof}
Based on Theorem \ref{thm:EssentialConditions}, there exists some
constant $\varkappa_{0}>0$ such that at each point $p=(x,y,e^{i\theta})$
in $\mathcal{U}_{*}$,
\[
d\beta(\mathcal{X},\bar{\partial}_{r})\big|_{p}>\varkappa_{0}\cdot r^{2}=\varkappa_{0}\cdot(x^{2}+y^{2}),
\]
and then applying Lemma \ref{lem:horizontal-integral} yields
\[
\int_{0}^{1}2\nu_{s',t}\cdot d\beta(\mathcal{X},\bar{\partial}_{r})\big|_{\varphi_{\mathfrak{f}}^{s}(p_{t})}dt>\varkappa_{0}\cdot\int_{0}^{1}2\nu_{s',t}\cdot r^{2}\big|_{\varphi_{\mathfrak{f}}^{s}(p_{t})}dt=\varkappa_{0}\cdot r^{2}\big|_{\varphi_{\mathfrak{f}}^{s}(\bar{p})},
\]
that is,
\begin{equation}
\int_{0}^{1}\nu_{s',t}\cdot d\beta(\mathcal{X},\bar{\partial}_{r})\big|_{\varphi_{\mathfrak{f}}^{s}(p_{t})}dt>\frac{\varkappa_{0}}{2}\cdot r^{2}\big|_{\varphi_{\mathfrak{f}}^{s}(\bar{p})}.\label{eq:HorizontalLowerBound}
\end{equation}

Solving Equation (\ref{eq:=00005Bf=00005CX=00005DStructuralEquation})
yields
\[
\kappa_{s,t}^{\mathfrak{f}}=\int_{0}^{s}e^{\int_{s'}^{s}d\beta(\mathcal{X}_{\mathfrak{f}},\frac{\partial}{\partial\theta})ds''}d\beta(\mathcal{X}_{\mathfrak{f}},\nu_{s',t}\bar{\partial}_{r})ds'.
\]
Integrating $\kappa_{s,t}^{\mathfrak{f}}$ in $t$ over $[0,1]$ gives
\[
\int_{0}^{1}\kappa_{s,t}^{\mathfrak{f}}dt=\int_{0}^{1}\int_{0}^{s}e^{\int_{s'}^{s}d\beta(\mathcal{X}_{\mathfrak{f}},\frac{\partial}{\partial\theta})ds''}\mathfrak{f}\nu_{s',t}\cdot d\beta(\mathcal{X},\bar{\partial}_{r})ds'dt.
\]
Note that the factors of the integrand are all positive on $\mathcal{U}$,
and hence 
\[
e^{\int_{s'}^{s}d\beta(\mathcal{X}_{\mathfrak{f}},\frac{\partial}{\partial\theta})ds''}\mathfrak{f}\nu_{s',t}\cdot d\beta(\mathcal{X},\bar{\partial}_{r})\geq e^{(s-s')\cdot\underset{\mathcal{U}^{\mathfrak{f}}}{\inf}d\beta(\mathcal{X}_{\mathfrak{f}},\frac{\partial}{\partial\theta})}\cdot\underset{\mathcal{U}^{\mathfrak{f}}}{\inf}\mathfrak{f}\cdot\nu_{s',t}\cdot d\beta(\mathcal{X},\bar{\partial}_{r}).
\]
Interchanging the order of integration in $s,t$ yields
\begin{equation}
\begin{aligned}\int_{0}^{1}\kappa_{s,t}^{\mathfrak{f}}dt\geq & \int_{0}^{s}e^{(s-s')\cdot\underset{\mathcal{U}^{\mathfrak{f}}}{\inf}d\beta(\mathcal{X}_{\mathfrak{f}},\frac{\partial}{\partial\theta})}\cdot\underset{\mathcal{U}^{\mathfrak{f}}}{\inf}\mathfrak{f}\cdot\bigg(\int_{0}^{1}\nu_{s',t}\cdot d\beta(\mathcal{X},\bar{\partial}_{r})dt\bigg)ds'\\
\geq & \underset{\mathcal{U}^{\mathfrak{f}}}{\inf}\mathfrak{f}\cdot\frac{\varkappa_{0}}{2}\cdot\int_{0}^{s}e^{(s-s')\cdot\underset{\mathcal{U}^{\mathfrak{f}}}{\inf}d\beta(\mathcal{X}_{\mathfrak{f}},\frac{\partial}{\partial\theta})}r^{2}\big|_{\varphi_{\mathfrak{f}}^{s'}(\bar{p})}ds'.
\end{aligned}
\label{eq:=00005Cint=00007B=00005Ckappa=00007Ddt}
\end{equation}
Note that (\ref{eq:HorizontalLowerBound}) has been applied for the
second estimation in (\ref{eq:=00005Cint=00007B=00005Ckappa=00007Ddt})
above, and (\ref{eq:=00005Cint=00007B=00005Ckappa=00007Ddt}) holds
for all $s\in[0,\bar{\epsilon}_{\mathfrak{f}}]$. 

To further quantify the last integral in (\ref{eq:=00005Cint=00007B=00005Ckappa=00007Ddt}),
we deduce a lower bound for $r^{2}\big|_{\varphi_{\mathfrak{f}}^{s'}(\bar{p})}$.
From the analysis in Section \ref{sec:PathFollo-R^3} we already see
that, the variation of $r^{2}$ along the orbit $s\mapsto\varphi_{\mathfrak{f}}^{s}(\bar{p})$
is controlled solely by the part
\[
\mathfrak{f}\bar{b}\cdot\mathcal{V}_{\beta}\times\big(\frac{\partial}{\partial\theta}\times\nabla\mathfrak{H}\big)=\mathfrak{f}\bar{b}\cdot\big(\mathcal{V}_{\beta}\cdot\nabla\mathfrak{H}\big)\frac{\partial}{\partial\theta}-\mathfrak{f}\bar{b}\cdot\nabla\mathfrak{H},
\]
more specifically, by $-\mathfrak{f}\bar{b}\cdot\nabla\mathfrak{H}$.
In fact, in the chart $\mathcal{U}\xrightarrow{\mathfrak{Eh}}\mathcal{B}_{\delta}\times S^{1}$,
we have $r^{2}=\mathfrak{H}$, and then
\[
\begin{aligned}r^{2}\big|_{\varphi_{\mathfrak{f}}^{s}(\bar{p})}= & r^{2}\big|_{\bar{p}}+\int_{0}^{s}\frac{d}{ds'}r^{2}\circ\varphi_{\mathfrak{f}}^{s'}(\bar{p})ds'\\
= & r^{2}\big|_{\bar{p}}-\int_{0}^{s}\mathfrak{f}\bar{b}\cdot d\mathfrak{H}\big(\nabla\mathfrak{H}\big)\big|_{\varphi^{s'}(\bar{p})}ds'\\
= & r^{2}\big|_{\bar{p}}-\int_{0}^{s}\mathfrak{f}\bar{b}\cdot\big|\nabla\mathfrak{H}\big|^{2}\bigg|{}_{\varphi^{s'}(\bar{p})}ds'.
\end{aligned}
\]
Treating $r^{2}$ as a variable, the function $s\mapsto r^{2}\big|_{\varphi_{\mathfrak{f}}^{s}(\bar{p})}$
has an inverse on $[0,\bar{\epsilon}_{\mathfrak{f}}]$ with the differentiation
\[
\frac{ds}{d\big(r^{2}\big)}=-\frac{1}{\mathfrak{f}\bar{b}\cdot\big|\nabla\mathfrak{H}\big|^{2}},
\]
and then 
\[
r^{2}\big|_{\varphi_{\mathfrak{f}}^{s'}(\bar{p})}ds'=r^{2}\frac{ds'}{d\big(r^{2}\big)}d\big(r^{2}\big)=-\frac{1}{2\mathfrak{f}\bar{b}\cdot\big|\nabla\mathfrak{H}\big|^{2}}d\big(r^{4}\big).
\]
With the change of variable $r\mapsto s$, the integral $\int_{0}^{s}e^{(s-s')\cdot\underset{\mathcal{U}^{\mathfrak{f}}}{\inf}d\beta(\mathcal{X}_{\mathfrak{f}},\frac{\partial}{\partial\theta})}r^{2}\big|_{\varphi_{\mathfrak{f}}^{s'}(\bar{p})}ds'$
becomes
\[
\begin{aligned}-\int_{\bar{r}}^{r_{s}}\frac{e^{(s-s_{r})\cdot\underset{\mathcal{U}^{\mathfrak{f}}}{\inf}d\beta(\mathcal{X}_{\mathfrak{f}},\frac{\partial}{\partial\theta})}}{2\mathfrak{f}\bar{b}\cdot\big|\nabla\mathfrak{H}\big|^{2}}d\big(r^{4}\big)= & \int_{r_{s}}^{\bar{r}}\frac{e^{(s-s_{r})\cdot\underset{\mathcal{U}^{\mathfrak{f}}}{\inf}d\beta(\mathcal{X}_{\mathfrak{f}},\frac{\partial}{\partial\theta})}}{2\mathfrak{f}\bar{b}\cdot\big|\nabla\mathfrak{H}\big|^{2}}d\big(r^{4}\big)\\
= & \int_{r_{s}}^{\bar{r}}\frac{2e^{(s-s_{r})\cdot\underset{\mathcal{U}^{\mathfrak{f}}}{\inf}d\beta(\mathcal{X}_{\mathfrak{f}},\frac{\partial}{\partial\theta})}}{\mathfrak{f}\bar{b}\cdot\big|\nabla\mathfrak{H}\big|^{2}}r^{3}dr.
\end{aligned}
\]
Note that $\underset{\mathcal{U}^{\mathfrak{f}}}{\inf}d\beta(\mathcal{X}_{\mathfrak{f}},\frac{\partial}{\partial\theta})\leq0$
and hence
\[
e^{(s-s_{r})\cdot\underset{\mathcal{U}^{\mathfrak{f}}}{\inf}d\beta(\mathcal{X}_{\mathfrak{f}},\frac{\partial}{\partial\theta})}\geq e^{\bar{\epsilon}_{\mathfrak{f}}\underset{\mathcal{U}^{\mathfrak{f}}}{\inf}d\beta(\mathcal{X}_{\mathfrak{f}},\frac{\partial}{\partial\theta})}.
\]
Also, $\frac{1}{\mathfrak{f}}\geq\frac{1}{\sup\mathfrak{f}}>0$. Putting
all these back with (\ref{eq:=00005Cint=00007B=00005Ckappa=00007Ddt})
we get the following lower bound for $\int_{0}^{1}\kappa_{s,t}^{\mathfrak{f}}dt$
at each $s\in[0,\bar{\epsilon}_{\mathfrak{f}}]$
\begin{equation}
\begin{aligned}\int_{0}^{1}\kappa_{s,t}^{\mathfrak{f}}dt\geq & \bigg(\frac{\inf_{\mathcal{U}^{\mathfrak{f}}}\mathfrak{f}}{\sup_{\mathcal{U}^{\mathfrak{f}}}\mathfrak{f}}\bigg)\cdot\frac{\varkappa_{0}}{2}\cdot e^{\bar{\epsilon}_{\mathfrak{f}}\underset{\mathcal{U}^{\mathfrak{f}}}{\inf}d\beta(\mathcal{X}_{\mathfrak{f}},\frac{\partial}{\partial\theta})}\cdot\inf_{\mathcal{U}_{*}}\frac{r^{2}}{\big|\nabla\mathfrak{H}\big|^{2}}\cdot\int_{r_{s}}^{\bar{r}}\frac{2r}{\bar{b}}dr\\
= & \overline{\varkappa_{0}}\cdot\int_{r_{s}}^{\bar{r}}\frac{2r}{\bar{b}}dr,
\end{aligned}
\label{eq:MiniCircling}
\end{equation}
where
\[
\overline{\varkappa_{0}}:=\bigg(\frac{\inf_{\mathcal{U}^{\mathfrak{f}}}\mathfrak{f}}{\sup_{\mathcal{U}^{\mathfrak{f}}}\mathfrak{f}}\bigg)\cdot\frac{\varkappa_{0}}{2}\cdot e^{\bar{\epsilon}_{\mathfrak{f}}\underset{\mathcal{U}^{\mathfrak{f}}}{\inf}d\beta(\mathcal{X}_{\mathfrak{f}},\frac{\partial}{\partial\theta})}\cdot\inf_{\mathcal{U}_{*}}\frac{r^{2}}{\big|\nabla\mathfrak{H}\big|^{2}}.
\]
From inequality (\ref{Ineq:InfinitesimalOrder=gradH}),
we have $\inf_{\mathcal{U}_{*}}\frac{r^{2}}{\big|\nabla\mathfrak{H}\big|^{2}}>0$,
and then the number $\overline{\varkappa_{0}}$ is well defined with
$\overline{\varkappa_{0}}>0$.

Now that $0<\frac{\bar{b}}{r^{2}}\leq\underset{\mathcal{U}_{*}^{\mathfrak{f}}}{\sup}\frac{\bar{b}}{r^{2}}<\infty$
holds on $\mathcal{U}^{\mathfrak{f}}$, we have $\frac{1}{\bar{b}}\geq\frac{1}{\underset{\mathcal{U}_{*}^{\mathfrak{f}}}{\sup}\frac{\bar{b}}{r^{2}}}\cdot\frac{1}{r^{2}}$,
combining which with (\ref{eq:MiniCircling}) yields 
\[
\int_{0}^{1}\kappa_{s,t}^{\mathfrak{f}}dt\geq\frac{\overline{\varkappa_{0}}}{\underset{\mathcal{U}_{*}^{\mathfrak{f}}}{\sup}\frac{\bar{b}}{r^{2}}}\cdot\int_{r_{s}}^{\bar{r}}\frac{2r}{r^{2}}dr=\frac{\overline{\varkappa_{0}}}{\underset{\mathcal{U}_{*}^{\mathfrak{f}}}{\sup}\frac{\bar{b}}{r^{2}}}\cdot\ln\frac{\bar{r}^{2}}{r_{s}^{2}}.
\]
Since $s\mapsto r_{s}^{s}$ is decreasing monotonically, the right-hand
side above is increasing in $s$, and then
\[
\begin{aligned}\hat{\theta}_{s}^{\mathfrak{f}}-\hat{\theta}_{0}^{\mathfrak{f}}=\bar{\vartheta}_{t=1}^{s}-\bar{\vartheta}_{t=0}^{s}= & \bigg(\inf_{\mathcal{U}}\hat{\mu}\bigg)\cdot\int_{0}^{1}\kappa_{s,t}^{\mathfrak{f}}dt\\
\geq & \bigg(\inf_{\mathcal{U}}\hat{\mu}\bigg)\cdot\frac{\overline{\varkappa_{0}}}{\underset{\mathcal{U}_{*}^{\mathfrak{f}}}{\sup}\frac{\bar{b}}{r^{2}}}\cdot\bigg(\ln\bar{r}^{2}-\ln r_{s}^{2}\bigg).
\end{aligned}
\]
Note that the relation above holds for any $\bar{p}\in\mathcal{U}_{*}^{\mathfrak{f}}$
with any $s\in[0,\bar{\epsilon}_{\mathfrak{f}}]$, while the relevant
parameters $\bigg(\inf_{\mathcal{U}}\hat{\mu}\bigg)$, $\frac{\overline{\varkappa_{0}}}{\underset{\mathcal{U}_{*}^{\mathfrak{f}}}{\sup}\frac{\bar{b}}{r^{2}}}$
and $\bar{\epsilon}_{\mathfrak{f}}$ are all independent of the choice
of $\bar{p}$. Therefore, for arbitrary $s>0$, we can simply partition
the interval $[0,\bar{s}]$ with a sequence of intermediate points
\[
0=\bar{s}_{0}<\bar{s}_{1}<...<\bar{s}_{k}=s
\]
such that $\bar{s}_{j+1}-\bar{s}_{j}<\bar{\epsilon}_{\mathfrak{f}}$.
Then, with $\bar{p}_{j}=\varphi_{\mathfrak{f}}^{s_{j}}(\bar{p})$
and $\Theta(\bar{p}_{j})=(\mathbf{0},e^{i\hat{\theta}_{s_{j}}^{\mathfrak{f}}})$,
it holds
\[
\hat{\theta}_{s_{j+1}}^{\mathfrak{f}}-\hat{\theta}_{s_{j}}^{\mathfrak{f}}\geq\bigg(\inf_{\mathcal{U}}\hat{\mu}\bigg)\cdot\frac{\overline{\varkappa_{0}}}{\underset{\mathcal{U}_{*}^{\mathfrak{f}}}{\sup}\frac{\bar{b}}{r^{2}}}\cdot\bigg(\ln r_{s_{j}}^{2}-\ln r_{s_{j+1}}^{2}\bigg).
\]
As a result, for arbitrary $s>0$ we have
\[
\hat{\theta}_{s}^{\mathfrak{f}}-\hat{\theta}_{0}^{\mathfrak{f}}=\sum_{j=0}^{k-1}\bigg(\hat{\theta}_{s_{j+1}}^{\mathfrak{f}}-\hat{\theta}_{s_{j}}^{\mathfrak{f}}\bigg)\geq\bigg(\inf_{\mathcal{U}}\hat{\mu}\bigg)\cdot\frac{\overline{\varkappa_{0}}}{\underset{\mathcal{U}_{*}^{\mathfrak{f}}}{\sup}\frac{\bar{b}}{r^{2}}}\cdot\sum_{j=0}^{k-1}\bigg(\ln r_{s_{j}}^{2}-\ln r_{s_{j+1}}^{2}\bigg),
\]
that is, 
\[
\hat{\theta}_{s}^{\mathfrak{f}}-\hat{\theta}_{0}^{\mathfrak{f}}\geq\bigg(\inf_{\mathcal{U}}\hat{\mu}\bigg)\cdot\frac{\overline{\varkappa_{0}}}{\underset{\mathcal{U}_{*}^{\mathfrak{f}}}{\sup}\frac{\bar{b}}{r^{2}}}\cdot\bigg(\ln\bar{r}^{2}-\ln r_{s}^{2}\bigg).
\]
Since $r_{s}^{2}\xrightarrow{s\rightarrow\infty}0_{+}$, we obtain
$\underset{s\rightarrow\infty}{\lim}\bigg(\hat{\theta}_{s}^{\mathfrak{f}}-\hat{\theta}_{0}^{\mathfrak{f}}\bigg)=\infty$
and conclude the proof.

\pagebreak{}

\section*{Appendices}

\subsection*{A. Proof for Inequality (\ref{Ineq:InfinitesimalOrder=gradH})}
\begin{proof}
The proof follows a standard line of argument. Taking the inner product
$\langle\nabla\mathfrak{H},\nabla\mathfrak{H}\rangle$ and dividing
it by $r^{2}$, we get on $\mathcal{U}_{*}$:
\[
\frac{\big|\big|\nabla\mathfrak{H}\big|\big|^{2}}{r^{2}}=\bigg(\frac{x^{2}}{r^{2}}\big|\big|\nabla f\big|\big|^{2}+\frac{y^{2}}{r^{2}}\big|\big|\nabla g\big|\big|^{2}\bigg)+\frac{2xy\nabla f\cdot\nabla g}{r^{2}}.
\]
The linear independence of $\nabla f$ and $\nabla g$ implies $\frac{\big|\nabla f\cdot\nabla g\big|}{\big|\big|\nabla f\big|\big|\cdot\big|\big|\nabla g\big|\big|}<1$
on $\mathcal{U}$. With the compactness of $\mathcal{U}$, this implies
$\frac{\big|\nabla f\cdot\nabla g\big|}{\big|\big|\nabla f\big|\big|\cdot\big|\big|\nabla g\big|\big|}\leq\varepsilon$
for some $\varepsilon\in(0,1)$,and thence we have
\[
\bigg|\frac{2xy\nabla f\cdot\nabla g}{r^{2}}\bigg|\leq\varepsilon\cdot\frac{2|xy|\cdot\big|\big|\nabla f\big|\big|\cdot\big|\big|\nabla g\big|\big|}{r^{2}}\leq\varepsilon\cdot\bigg(\frac{x^{2}}{r^{2}}\big|\big|\nabla f\big|\big|^{2}+\frac{y^{2}}{r^{2}}\big|\big|\nabla g\big|\big|^{2}\bigg).
\]
As a result, the quantity $\frac{\big|\big|\nabla\mathfrak{H}\big|\big|^{2}}{r^{2}}$
is bounded on $\mathcal{U}_{*}$ by
\[
(1-\varepsilon)\cdot\bigg(\frac{x^{2}}{r^{2}}\big|\big|\nabla f\big|\big|^{2}+\frac{y^{2}}{r^{2}}\big|\big|\nabla g\big|\big|^{2}\bigg)\leq\frac{\big|\big|\nabla\mathfrak{H}\big|\big|^{2}}{r^{2}}\leq(1+\varepsilon)\cdot\bigg(\frac{x^{2}}{r^{2}}\big|\big|\nabla f\big|\big|^{2}+\frac{y^{2}}{r^{2}}\big|\big|\nabla g\big|\big|^{2}\bigg).
\]
At each $p=(x,y,e^{i\theta})$ in $\mathcal{U}_{*}$, $\frac{x^{2}}{r^{2}}+\frac{y^{2}}{r^{2}}=1$
and then
\[
\min\bigg\{\big|\big|\nabla f(p)\big|\big|^{2},\big|\big|\nabla g(p)\big|\big|^{2}\bigg\}\leq\frac{x^{2}}{r^{2}}\big|\big|\nabla f\big|\big|^{2}+\frac{y^{2}}{r^{2}}\big|\big|\nabla g\big|\big|^{2}\leq\max\bigg\{\big|\big|\nabla f(p)\big|\big|^{2},\big|\big|\nabla g(p)\big|\big|^{2}\bigg\},
\]
from which we deduce
\[
(1-\varepsilon)\cdot\min_{\mathcal{U}}\bigg\{\big|\big|\nabla f\big|\big|^{2},\big|\big|\nabla g\big|\big|^{2}\bigg\}\leq\frac{\big|\big|\nabla\mathfrak{H}\big|\big|^{2}}{r^{2}}\leq(1+\varepsilon)\cdot\max_{\mathcal{U}}\bigg\{\big|\big|\nabla f\big|\big|^{2},\big|\big|\nabla g\big|\big|^{2}\bigg\}
\]
and conclude the proof with the compactness of $\mathcal{U}$.
\end{proof}

\subsection*{B. Specific Construction for $\mathcal{U}\xrightarrow{\mathfrak{Eh}}\mathcal{B}_{\delta}\times S^{1}$}

Considering the fact that it serves as the foundation of this whole
work, for the sake of completeness we shall demonstrate a specific
construction of the diffeomorphism $\mathcal{U}\xrightarrow{\mathfrak{Eh}}\mathcal{B}_{\delta}\times S^{1}$
to ensure its existence. In contrast to a standard construction
using parallel transport with respect to the connection $\ker\beta$, the one provided
here fully exploits the dynamics of $\nabla f \times \nabla g$ and only works in the case of dimension $3$.

Recall that $\mathfrak{P}:\mathbb{R}^{3}\rightarrow\mathbb{R}^{2}$
is the smooth map defined by $\mathfrak{P}(p)=\big(f(p),g(p)\big)$,
and $\mathcal{P}$ is a compact component of the level set $\mathfrak{P}^{-1}(\mathbf{0})$
consisting of regular points. Take any point $\hat{p}\in\mathcal{P}$.
As a result of the regularity of $\mathfrak{P}$ at $\hat{p}$, there
exists a neighborhood $\mathcal{B}_{\delta}$ of the point $\mathfrak{P}(\hat{p})=\mathbf{0}$
in $\mathbb{R}^{2}$ together with a smooth map $\mathfrak{q}:\mathcal{B}_{\delta}\rightarrow\mathbb{R}^{3}$
such that $\mathfrak{P}\circ\mathfrak{q}(\mathbf{z})=\mathbf{z}$
for all $\mathbf{z}\in\mathcal{B}_{\delta}$ and $\mathfrak{q}(\mathbf{0})=\hat{p}$.
Let $\mathcal{X}_{\vartheta}=\nabla f\times\nabla g$, and then $\mathcal{X}_{\vartheta}\neq0$
everywhere on $\mathcal{P}$. Moreover, $\mathcal{P}$ is perpendicular
to $\nabla f,\nabla g$ and hence is tangent to $\mathcal{X}_{\vartheta}$.
Therefore, $\mathcal{P}$ is an integral curve of $\mathcal{X}_{\vartheta}$.
Also, by continuity, there is a neighborhood $\mathcal{U}_{1}$ of
$\mathcal{P}$ such that $\mathcal{X}_{\vartheta}\neq0$ everywhere
on $\mathcal{U}_{1}$. 

We may replace $\mathcal{X}_{\vartheta}$ with $\bar{\mathcal{X}}_{\vartheta}=\frac{\mathcal{X}_{\vartheta}}{1+||\mathcal{X}_{\vartheta}||^{2}}$,
and then $\bar{\mathcal{X}}_{\vartheta}$ is a complete vector field
on $\mathbb{R}^{3}$. That is, the life span of each orbit of $\bar{\mathcal{X}}_{\vartheta}$
is $\mathbb{R}$, and hence its flow $\bar{\varphi}_{\vartheta}$
is defined globally on $\mathbb{R}^{3}\times\mathbb{R}$. From its
construction we know that $\bar{\mathcal{X}}_{\vartheta}\neq0$ holds
everywhere on $\mathcal{U}_{1}$, and, $\mathcal{P}$ is an orbit
of $\bar{\mathcal{X}}_{\vartheta}$. In fact, $\bar{\mathcal{X}}_{\vartheta}$
is always perpendicular to $\nabla f,\nabla g$, and hence each of
its orbits remains on a single level set of $\mathfrak{P}$, i.e.,
$\mathfrak{P}\circ\bar{\varphi}_{\vartheta}^{t}=\mathfrak{P}$.

Define a map $\Phi_{\vartheta}$ by
\[
\mathcal{B}_{\delta}\times\mathbb{R}\ni(\mathbf{z},t)\xmapsto{\Phi_{\vartheta}}\bar{\varphi}_{\vartheta}^{t}\circ\mathfrak{q}(\mathbf{z})\in\mathbb{R}^{3}.
\]
Note that $\text{Im}\mathfrak{q}=\mathfrak{q}(\mathcal{B}_{\delta})$
is an embedded surface (a disk) in $\mathbb{R}^{3}$, and it intersects
transversely with $\mathcal{P}$ at a single point $\hat{p}=\mathfrak{q}(\mathbf{0})$.
In other words, the vector field $\mathcal{X}_{\vartheta}$ is transverse
to $\text{Im}\mathfrak{q}$ at $\hat{p}$, and by continuity, $\mathcal{X}_{\vartheta}$
is transverse to $\text{Im}\mathfrak{q}$ in a vincinity of $\hat{p}$.
By shrinking $\mathcal{B}_{\delta}$ if necessary, we may assume $\mathcal{X}_{\vartheta}$
to be transverse to the whole surface $\text{Im}\mathfrak{q}$. Since
$\frac{\partial}{\partial t}\Phi_{\vartheta}^{t}(\mathbf{z})=\mathcal{X}_{\vartheta}$,
this means that the map $\Phi_{\vartheta}$ is transverse to the submanifold
$\text{Im}\mathfrak{q}$. It is then a standard result from transversality
that the subset $\Gamma$ given below is an embedded submanifold of
$\mathcal{B}_{\delta}\times\mathbb{R}$ of codimension $1$:
\[
\Gamma:=\Phi_{\vartheta}^{-1}\big(\text{Im}\mathfrak{q}\big)=\bigg\{(\mathbf{z},t)\in\mathcal{B}_{\delta}\times\mathbb{R}\bigg|\ \Phi_{\vartheta}^{t}(\mathbf{z})\in\text{Im}\mathfrak{q}\bigg\}.
\]

Now we are ready to explain the idea for the contruction of $\mathfrak{Eh}$.
It turns out $\Gamma=\underset{k\in\mathbb{Z}}{\bigsqcup}\Gamma_{k}$,
in which each $\Gamma_{k}$ is a connected component of $\Gamma$
and takes the form
\begin{equation}
\Gamma_{k}=\bigg\{\big(\mathbf{z},\mathfrak{t}_{k}(\mathbf{z})\big)\bigg|\ \forall\mathbf{z}\in\mathcal{B}_{\delta}\bigg\}.\label{eq:Sheets-of-RecurrenceTime}
\end{equation}
Here, $\mathfrak{t}_{k}:\mathcal{B}_{\delta}\rightarrow\mathbb{R}$,
$k\in\mathbb{Z}$ are smooth functions with the properties $\mathfrak{t}_{k}=k\cdot\mathfrak{t}_{1}$
and $\mathfrak{t}_{1}>0$, and as a result,
\begin{equation}
...<\mathfrak{t}_{-2}<\mathfrak{t}_{-1}<\mathfrak{t}_{0}\equiv0<\mathfrak{t}_{1}<\mathfrak{t}_{2}<...\,.\label{eq:SmoothLattice-=00005BRecurrenceTime=00005D}
\end{equation}
It follows directly from the definition of $\Gamma$ that the $\mathfrak{t}_{k}(\mathbf{z})$'s
are the time moments of the orbit $t\mapsto\bar{\varphi}_{\vartheta}^{t}\circ\mathfrak{q}(\mathbf{z})$
returning to the surface $\text{Im}\mathfrak{q}$, and in fact, exactly
at $\mathfrak{q}(\mathbf{z})$. That is,
\[
\bar{\varphi}_{\vartheta}^{\mathfrak{t}_{k}(\mathbf{z})}\circ\mathfrak{q}(\mathbf{z})=\mathfrak{q}(\mathbf{z}).
\]
Thus, $\mathfrak{t}_{1}$ is the function of the first positive recurrence
times, and for any $0\leq t_{0}<t_{1}\leq\mathfrak{t}_{1}(\mathbf{z})$,
$\Phi_{\vartheta}^{t_{0}}(\mathbf{z})=\Phi_{\vartheta}^{t_{1}}(\mathbf{z})$
only happens when $t_{0}=0$ and $t_{1}=\mathfrak{t}_{1}(\mathbf{z})$.
Define a map $\overline{\mathfrak{he}}$ by
\[
\mathcal{B}_{\delta}\times[0,1]\ni(\mathbf{z},t)\xmapsto{\overline{\mathfrak{he}}}\Phi_{\vartheta}^{t\cdot\mathfrak{t}_{1}(\mathbf{z})}(\mathbf{z})\in\mathbb{R}^{3}.
\]
The map has the property $\overline{\mathfrak{he}}(\mathbf{z},0)=\overline{\mathfrak{he}}(\mathbf{z},1)$
and hence it factors through $\mathcal{B}_{\delta}\times S^{1}$ by
\[
\overline{\mathfrak{he}}:(\mathbf{z},t)\mapsto(\mathbf{z},e^{2\pi t\cdot i})\xmapsto{\mathfrak{he}}\Phi_{\vartheta}^{t\cdot\mathfrak{t}_{1}(\mathbf{z})}(\mathbf{z}).
\]
It can be shown that the map $\mathfrak{he}$ is a diffeomorphism
between $\mathcal{B}_{\delta}\times S^{1}$ and its image $\mathcal{U}=\text{Im}\mathfrak{he}$.
Moreover, we check that 
\[
\mathfrak{P}\circ\mathfrak{he}(\mathbf{z},e^{2\pi t\cdot i})=\mathfrak{P}\circ\bar{\varphi}_{\vartheta}^{t\cdot\mathfrak{t}_{k}(\mathbf{z})}\circ\mathfrak{q}(\mathbf{z})=\mathbf{z},
\]
that is, $\mathfrak{P}\circ\mathfrak{he}=\mathfrak{p}$. Based on
these results, the construction of $\mathfrak{Eh}$ is done by taking
\[
\mathfrak{Eh}=\mathfrak{he}^{-1}.
\]

\begin{rem}
While the exposition above has been brief, most of the details omitted
there are just standard lines of argument, except for those about
$\Gamma_{j}$ taking the form of (\ref{eq:Sheets-of-RecurrenceTime})
with each $\mathfrak{t}_{k}$ being a smooth function on $\mathcal{B}_{\delta}$
and satisfying (\ref{eq:SmoothLattice-=00005BRecurrenceTime=00005D}).
The essential part of the question here is that, why $\mathfrak{t}_{1}$,
defined to be the function of all first recurrence moments, is a smooth
one. Or, from another perspective, it is to answer why there happens
to be a component of $\Gamma$ in which each $t$ of $(\mathbf{z},t)$
is the first recurrence time. The key for showing this is to recognize
that $\Gamma$ is an invariant set under the group action $(\mathbf{z},t)*k=(\mathbf{z},t\cdot k)$
on $\mathcal{B}_{\delta}\times\mathbb{R}$ by $\mathbb{Z}$, and thence
for any component $\Gamma_{j}$, $\Gamma_{j}*k$ is another component.
This has the implication that, if $(\mathbf{0},t_{0})\in\Gamma_{1}$
and $t_{0}$ is the first recurrence time for $\mathfrak{q}(\mathbf{0})$,
then for any other $(\mathbf{z},t)\in\Gamma_{1}$, $t$ is the first
recurrence time for $\mathfrak{q}(\mathbf{z})$.
\end{rem}

\subsection*{C. The Closed $1$-form $d\phi$}

Taking the polar coordinates on $\mathcal{B}_{\delta}$: $(r,\phi)\mapsto(x,y)=(r\cos\phi,r\sin\phi)$,
we have
\[
dx=\cos\phi\cdot dr-r\sin\phi\cdot d\phi
\]
and
\[
dy=\sin\phi\cdot dr+r\cos\phi\cdot d\phi,
\]
from which we solve $d\phi$ and obtain
\[
d\phi=\frac{x}{r^{2}}dy-\frac{y}{r^{2}}dx.
\]

\subsection*{D. Proof for Proposition \ref{prop:onDiscSignChange}}
\begin{proof}
Let $\mathcal{V}_{\theta}$ be the vector field defined by $\langle\mathcal{V}_{\theta},\cdot\rangle=d\theta(\cdot)$,
and then $\langle\mathcal{V}_{\theta},\frac{\partial}{\partial\theta}\rangle=d\theta(\frac{\partial}{\partial\theta})=1$.
Also $\mathcal{V}_{\theta}$ is perpendicular to the space $\mathrm{span}\big\{\frac{\partial}{\partial x},\frac{\partial}{\partial y}\big\}$
spanned by $\frac{\partial}{\partial x},\frac{\partial}{\partial y}$
since
\[
\langle\mathcal{V}_{\theta},\frac{\partial}{\partial x}\rangle=d\theta(\frac{\partial}{\partial x})=\langle\mathcal{V}_{\theta},\frac{\partial}{\partial y}\rangle=d\theta(\frac{\partial}{\partial y})=0.
\]
Check that
\[
\begin{aligned}d\theta(\mathcal{X})= & \bar{a}\big(\mathcal{V}_{\beta}\times\nabla\mathfrak{H}\big)\cdot\mathcal{V}_{\theta}+\bar{b}\big(\mathcal{V}_{\beta}\cdot\nabla\mathfrak{H}\big)-\bar{b}\beta\big(\frac{\partial}{\partial\theta}\big)\nabla\mathfrak{H}\cdot\mathcal{V}_{\theta}\\
= & \bar{a}\big(\mathcal{V}_{\theta}\times\mathcal{V}_{\beta}\big)\cdot\nabla\mathfrak{H}+\bar{b}\big(\mathcal{V}_{\beta}-\beta\big(\frac{\partial}{\partial\theta}\big)\mathcal{V}_{\theta}\big)\cdot\nabla\mathfrak{H}.\\
= & d\mathfrak{H}\bigg(\bar{a}\big(\mathcal{V}_{\theta}\times\mathcal{V}_{\beta}\big)+\bar{b}\big(\mathcal{V}_{\beta}-\beta\big(\frac{\partial}{\partial\theta}\big)\mathcal{V}_{\theta}\big)\bigg).
\end{aligned}
\]
Now that $d\mathfrak{H}=xdx+ydy$, it suffices to focus on the part
$\bar{u}_{x}\frac{\partial}{\partial x}+\bar{u}_{y}\frac{\partial}{\partial y}$
in the decomposition
\[
\bar{a}\big(\mathcal{V}_{\theta}\times\mathcal{V}_{\beta}\big)+\bar{b}\bigg(\mathcal{V}_{\beta}-\beta\big(\frac{\partial}{\partial\theta}\big)\mathcal{V}_{\theta}\bigg)=\bar{u}_{\theta}\frac{\partial}{\partial\theta}+\bar{u}_{x}\frac{\partial}{\partial x}+\bar{u}_{y}\frac{\partial}{\partial y},
\]
since
\[
d\theta(\mathcal{X})=d\mathfrak{H}\bigg(\bar{u}_{\theta}\frac{\partial}{\partial\theta}+\bar{u}_{x}\frac{\partial}{\partial x}+\bar{u}_{y}\frac{\partial}{\partial y}\bigg)=x\bar{u}_{x}+y\bar{u}_{y}.
\]
On the boundary of the disk $\partial\mathcal{B}_{\delta}\times\{e^{i\bar{\theta}}\}$,
$(x,y)=(\delta\cos\phi,\delta\sin\phi)$, and then $d\theta(\mathcal{X})\neq0$
implies that the map
\[
\partial\mathcal{B}_{\delta}\times\{e^{i\bar{\theta}}\}\ni p\mapsto\bigg(\frac{\bar{u}_{x}}{\sqrt{\bar{u}_{x}^{2}+\bar{u}_{y}^{2}}}\bigg|_{p},\frac{\bar{u}_{y}}{\sqrt{\bar{u}_{x}^{2}+\bar{u}_{y}^{2}}}\bigg|_{p}\bigg)\in S^{1}
\]
has degree $1$. 

$d\theta\wedge\beta\neq0$ means $d\theta$ and $\beta$ to be linearly
independent, and it also means $\mathcal{V}_{\theta}\times\mathcal{V}_{\beta}\neq\mathbf{0}$
and implies $\mathcal{V}_{\beta}-\beta\big(\frac{\partial}{\partial\theta}\big)\mathcal{V}_{\theta}\neq\mathbf{0}$.
Since $\mathcal{V}_{\theta},\mathcal{V}_{\beta}\perp\mathcal{V}_{\theta}\times\mathcal{V}_{\beta}$,
we have
\[
\mathcal{V}_{\theta}\times\mathcal{V}_{\beta}\in\ker\beta\bigcap\mathrm{span}\big\{\frac{\partial}{\partial x},\frac{\partial}{\partial y}\big\}.
\]
Meanwhile, check that $\frac{\partial}{\partial\theta}$ is perpendicular
to $\mathcal{V}_{\beta}-\beta\big(\frac{\partial}{\partial\theta}\big)\mathcal{V}_{\theta}$:
\[
\bigg(\mathcal{V}_{\beta}-\beta\big(\frac{\partial}{\partial\theta}\big)\mathcal{V}_{\theta}\bigg)\cdot\frac{\partial}{\partial\theta}=\mathcal{V}_{\beta}\cdot\frac{\partial}{\partial\theta}-\beta\big(\frac{\partial}{\partial\theta}\big)\cdot1=0.
\]
So now we know that $\mathcal{V}_{\theta}\times\mathcal{V}_{\beta}$
and $\frac{\partial}{\partial\theta}$ are linearly independent (since
$\mathcal{V}_{\theta}\times\mathcal{V}_{\beta}\in\mathrm{span}\big\{\frac{\partial}{\partial x},\frac{\partial}{\partial y}\big\}$)
and both lie in the orthogonal complement of $\mathcal{V}_{\beta}-\beta\big(\frac{\partial}{\partial\theta}\big)\mathcal{V}_{\theta}$.
As a result, $\frac{\partial}{\partial\theta}$, $\mathcal{V}_{\theta}\times\mathcal{V}_{\beta}$
and $\mathcal{Y}:=\mathcal{V}_{\beta}-\beta\big(\frac{\partial}{\partial\theta}\big)\mathcal{V}_{\theta}$
are linearly independent. Define $\mathcal{Y}_{x,y}=\mathcal{Y}-d\theta(\mathcal{Y})\frac{\partial}{\partial\theta}$.
$\mathcal{V}_{\theta}\times\mathcal{V}_{\beta}$ and $\mathcal{Y}_{x,y}$
then form a global frame of the distribution $\mathrm{span}\big\{\frac{\partial}{\partial x},\frac{\partial}{\partial y}\big\}$
on $\mathcal{U}$, and on $\partial\mathcal{B}_{\delta}\times\{e^{i\bar{\theta}}\}$
it holds
\[
\bar{u}_{x}\frac{\partial}{\partial x}+\bar{u}_{y}\frac{\partial}{\partial y}=\bar{a}\big(\mathcal{V}_{\theta}\times\mathcal{V}_{\beta}\big)+\bar{b}\mathcal{Y}_{x,y}.
\]

Since both $\{\frac{\partial}{\partial x},\frac{\partial}{\partial y}\}$
and $\big\{\mathcal{V}_{\theta}\times\mathcal{V}_{\beta},\mathcal{Y}_{x,y}\big\}$
are global frames of the distribution $\mathrm{span}\big\{\frac{\partial}{\partial x},\frac{\partial}{\partial y}\big\}$,
there exists a transition function
\[
\mathbf{K}:\,\mathcal{U}\rightarrow\mathbb{GL}(2)
\]
such that for each $p\in\mathcal{U}$ and $(s,t)$, $(s',t')$ in
$\mathbb{R}^{2}$,
\[
\left[\begin{array}{c}
s'\\
t'
\end{array}\right]=\mathbf{K}_{p}\left[\begin{array}{c}
s\\
t
\end{array}\right]\iff s\frac{\partial}{\partial x}\bigg|_{p}+r\frac{\partial}{\partial y}\bigg|_{p}=s'\big(\mathcal{V}_{\theta}\times\mathcal{V}_{\beta}\big)\bigg|_{p}+t'\mathcal{Y}_{x,y}\bigg|_{p}.
\]
Then on the disk $\mathcal{D}_{\bar{\theta}}:=\mathcal{B}_{\delta}\times\{e^{i\bar{\theta}}\}$,
the following mapping defines an isomorphism of the $S^{1}$-bundle
\[
\mathcal{D}_{\bar{\theta}}\times S^{1}\ni(p,\mathbf{u})\xmapsto{\bar{\mathbf{K}}}\bigg(p,\,\frac{\mathbf{K}_{p}\mathbf{u}}{||\mathbf{K}_{p}\mathbf{u}||}\bigg)\in\mathcal{D}_{\bar{\theta}}\times S^{1}.
\]
It follows directly from the definition of $\bar{\mathbf{K}}$ that,
on the boundary $\partial\mathcal{D}_{\bar{\theta}}=\partial\mathcal{B}_{\delta}\times\{e^{i\bar{\theta}}\}$,
\[
\bigg(p,\,\frac{\bar{u}_{x}}{\sqrt{\bar{u}_{x}^{2}+\bar{u}_{y}^{2}}}\bigg|_{p},\frac{\bar{u}_{y}}{\sqrt{\bar{u}_{x}^{2}+\bar{u}_{y}^{2}}}\bigg|_{p}\bigg)\xmapsto{\bar{\mathbf{K}}}\bigg(p,\,\frac{\bar{a}_{p}}{\sqrt{\bar{a}_{p}^{2}+\bar{b}_{p}^{2}}},\frac{\bar{b}_{p}}{\sqrt{\bar{a}_{p}^{2}+\bar{b}_{p}^{2}}}\bigg).
\]
Since $\bar{\mathbf{K}}$ induces an isomorphism on the fundamental
group
\[
\pi_{1}(\mathcal{D}_{\bar{\theta}}\times S^{1})=\pi_{1}(\mathcal{D}_{\bar{\theta}})\times\pi_{1}(S^{1})=\{0\}\times\mathbb{Z},
\]
the loop $p\mapsto\bigg(\frac{\bar{a}_{p}}{\sqrt{\bar{a}_{p}^{2}+\bar{b}_{p}^{2}}},\frac{\bar{b}_{p}}{\sqrt{\bar{a}_{p}^{2}+\bar{b}_{p}^{2}}}\bigg)$
should be a generator of $\pi_{1}(S^{1})$ as $p\mapsto\bigg(\frac{\bar{u}_{x}}{\sqrt{\bar{u}_{x}^{2}+\bar{u}_{y}^{2}}}\bigg|_{p},\frac{\bar{u}_{y}}{\sqrt{\bar{u}_{x}^{2}+\bar{u}_{y}^{2}}}\bigg|_{p}\bigg)$
is, and therefore the function $\bar{b}$ (and $\bar{a}$ as well)
changes its sign on $\partial\mathcal{B}_{\delta}\times\{e^{i\bar{\theta}}\}$.
As a result, $d\mathfrak{H}(\mathcal{X})$ also changes its sign on
the circle since
\[
d\mathfrak{H}(\mathcal{X})=\mathcal{X}\cdot\nabla\mathfrak{H}=-\bar{b}\beta\big(\frac{\partial}{\partial\theta}\big)||\nabla\mathfrak{H}||^{2}.
\]
\end{proof}
\pagebreak{}

\bibliographystyle{plain}
\bibliography{nonHoloPathFollow}

\end{document}